\documentclass[12pt,a4paper,reqno]{amsart}

\usepackage{xcolor}
\usepackage[all,2cell,arrow,color]{xy}
\newdir{ >}{{}*!/-10pt/@{>}}

\usepackage[english]{babel}
\usepackage{mathtools}
\usepackage{stmaryrd}
\usepackage{tensor}
\numberwithin{equation}{section}
\usepackage{graphicx} 
\usepackage{rotating}
\usepackage{pdflscape}
\usepackage[pagebackref,colorlinks,urlcolor=darkgreen,citecolor=darkgreen,linkcolor=darkgreen]{hyperref}
\usepackage{enumerate,xspace}
\usepackage{xcolor}

\definecolor{darkgreen}{rgb}{0,0.45,0}

  \newtheorem{proposition}{Proposition}[section]
  \newtheorem{lemma}[proposition]{Lemma}
  
  \newtheorem{theorem}[proposition]{Theorem}

  \theoremstyle{definition}
  
  \newtheorem{example}[proposition]{Example}

\theoremstyle{remark}
  \newtheorem{remark}[proposition]{Remark}
 
  \newcounter{c}
  \renewcommand{\[}{\setcounter{c}{1}$$}
  \newcommand{\etyk}[1]{\vspace{-7.4mm}$$\begin{equation}\Label{#1}
  \addtocounter{c}{1}}
  \renewcommand{\]}{\ifnum \value{c}=1 $$\else \end{equation}\fi}
\makeatletter
\newcommand*{\inlineequation}[2][]{%
  \begingroup
    \refstepcounter{equation}%
    \ifx\\#1\\%
    \else
      \label{#1}%
    \fi
    \relpenalty=10000 %
    \binoppenalty=10000 %
    \ensuremath{%
      #2%
    }%
    ~\@eqnnum
  \endgroup
}

\makeatletter
\def\@settitle{\begin{center}%
  \baselineskip14\p@\relax
  \bfseries
  \uppercasenonmath\@title
  \@title
  \ifx\@subtitle\@empty\else
     \\[1ex]\uppercasenonmath\@subtitle
     \footnotesize\mdseries\@subtitle
  \fi
  \end{center}%
}
\def\subtitle#1{\gdef\@subtitle{#1}}
\def\@subtitle{}
\makeatother

\newcommand{\coten}[1]{\raisebox{-7pt}{\ensuremath{\stackrel{\displaystyle  \Box}{\scriptstyle { #1}}}}}
\newcommand{\morcoten}{\scalebox{.7}{\ensuremath{\,\Box \,}}}
\newcommand{\diagcoten}{\scalebox{.55}{\ensuremath \Box}}
\newcommand{\expcoten}[1]{\scalebox{.55}{
{\raisebox{-8pt}{\ensuremath{\stackrel{\displaystyle  \Box}{\scriptstyle { #1}}}}}}}

\begin{document}

\title[Crossed modules of monoids II.]{Crossed modules of monoids II.}
\subtitle{Relative crossed modules}

\author{Gabriella B\"ohm} 
\address{Wigner Research Centre for Physics, H-1525 Budapest 114,
P.O.B.\ 49, Hungary}
\email{bohm.gabriella@wigner.mta.hu}
\date{March 2018}
  
\begin{abstract}
This is the second part of a series of three strongly related papers in which
three equivalent structures are studied:
\begin{itemize}
\item[-] internal categories in categories of monoids; defined in terms of
pullbacks relative to a chosen class of spans
\item[-] crossed modules of monoids relative to this class of spans 
\item[-] simplicial monoids of so-called Moore length 1 relative to this class
  of spans. 
\end{itemize}
The most important examples of monoids that are covered are small categories
(treated as monoids in categories of spans) and bimonoids in symmetric
monoidal categories (regarded as monoids in categories of comonoids). 
In this second part we define relative crossed modules of monoids and prove
their equivalence with the relative categories of Part I. 
\end{abstract}
  
\maketitle


\section*{Introduction} \label{sec:intro}

Since their appearance in \cite{Whitehead}, crossed modules of groups have
been intensively studied and applied in various contexts; see e.g. the reviews
\cite{Porter:Menagerie,Porter:HQFT,Paoli} and the references in them.
They admit several different descriptions: a {\em simplicial group whose
Moore complex is concentrated in degrees 1 and 2} turns out to be the internal
nerve of a {\em strict 2-group} and the Moore complex yields a {\em crossed
module}. These constructions establish, in fact, equivalences between these
three notions.   

The first (to our knowledge) proof of the equivalence between crossed modules and strict 2-groups --- that is, category objects in the category of groups --- can be found in \cite{BrownSpencer}, where it is referred also to the unpublished proof \cite{Duskin}.
Based on the fact that groups constitute a semi-Abelian category, another short and deeply conceptual proof is due to George Janelidze \cite{Janelidze}.    

More recently, however, some results on, and certain applications of crossed
modules of groups were extended to crossed modules of groupoids
\cite{BrownIcen} and of Hopf algebras \cite{Aguiar, Villanueva, Majid,
FariaMartins, Emir}. To these generalizations Janelidze's proof can not be
applied directly. Our aim is therefore to develop a wider theory of crossed
modules of monoids in more general monoidal categories which are not expected
to have all pullbacks (not even along split epimorphisms). We have the above
two main examples in mind: 
\begin{itemize}
\item[{-}] Categories of spans whose monoids are small categories, including
  groupoids in particular.
\item[{-}] Categories of comonoids in symmetric monoidal categories whose
  monoids are bimonoids including Hopf monoids in particular.
\end{itemize}

In the first part \cite{Bohm:Xmod_I} of this series of papers  we discussed
classes of spans satisfying appropriate conditions; and relative pullbacks with
respect to them. Assuming that such pullbacks exist --- as they do in our key
examples --- we introduced a monoidal category with monoidal product provided by
these pullbacks. We defined a relative (to the chosen class of spans) category as a
monoid in this monoidal category. It is given by the usual data
\begin{equation}\label{eq:int_cat} \tag{\mbox{$\ast$}} 
\xymatrix@C=35pt{
B \ar@{ >->}|(.55){\, i \,}[r] &
A \ar@{->>}@<-5pt>[l]_-s \ar@{->>}@<5pt>[l]^- t &
A \coten B A \ar[l]_-d}
\end{equation}
where $\coten B$ is now a relative pullback.

In the current article we make the next step and prove the equivalence of the
following categories for a fixed class of suitable spans in a monoidal category:
\begin{itemize} 
\item[{-}] the category of relative categories in the category of monoids,
\item[{-}] the category of relative crossed modules of monoids.
\end{itemize}

Our methodology is inspired by Janelidze's paper \cite{Janelidze}: In Section
\ref{sec:split_epi} we investigate first some category of split epimorphisms
of monoids. We obtain an equivalent description of a split epimorphism of
monoids 
$\xymatrix@C=15pt{
B \ar@{ >->}@<2pt>^-i [r] &
A \ar@{->>}@<2pt>[l]^-s}$
in terms of a distributive law which allows for handy characterizations of
possible morphisms $t$ and $d$ in \eqref{eq:int_cat}. This is used in Section
\ref{sec:refl_graph} and Section \ref{sec:rel_cat}, respectively, to present
equivalent descriptions of some reflexive graphs of monoids in terms of
relative pre-crossed modules of monoids; and of relative category objects
\eqref{eq:int_cat} in categories of monoids in terms of relative crossed
modules of monoids. Applying our results to categories of spans and to
categories of comonoids, respectively, we re-obtain the definitions of crossed
modules of groupoids in \cite{BrownIcen} and of crossed modules of Hopf
monoids in \cite{Villanueva}, respectively.  

Our next aim is to extend to our setting the equivalence of strict 2-groups and the category of crossed modules of groups to the further category of simplicial groups whose Moore complex has length 1. This will be achieved in Part III of this series \cite{Bohm:Xmod_III}. 

\subsection*{Acknowledgement} 
The author's interest in the subject was triggered by the excellent workshop
{\em `Modelling Topological Phases of Matter -- TQFT, HQFT, premodular and
  higher categories, Yetter-Drinfeld and crossed modules in disguise'} in
Leeds UK, 5-8 July 2016. It is a pleasure to thank the organizers,
Zolt\'an K\'ad\'ar, Jo\~ao Faria Martins, Marcos Cal\c{c}ada and Paul Martin
for the experience and a generous invitation.
Financial support by the Hungarian National Research, Development and
Innovation Office – NKFIH (grant K124138) is gratefully acknowledged.  


\section{Split epimorphisms of monoids versus distributive laws}
\label{sec:split_epi}

We freely use definitions, notation and results from
\cite{Bohm:Xmod_I}. 
Throughout, the composition of some morphisms 
$\xymatrix@C=12pt{A \ar[r]^-g &B}$ and
$\xymatrix@C=12pt{B \ar[r]^-f &C}$
in an arbitrary category will be denoted by 
$\xymatrix@C=16pt{A \ar[r]^-{f.g} &C}$.
Identity morphisms will be denoted by $1$ (without any reference to the
(co)domain object if it causes no confusion).
In any monoidal category $\mathsf C$ the monoidal product will be
denoted by juxtaposition and the monoidal unit will be $I$. For the monoidal
product of $n$ copies of the same object $A$ also the power notation $A^n$
will be used. 
For any monoid $A$ in $\mathsf C$, the multiplication and the unit
morphisms will be denoted by 
$\xymatrix@C=12pt{A^2 \ar[r]^-m &A}$ and
$\xymatrix@C=12pt{I \ar[r]^-u &A}$, respectively.
If $\mathsf C$ is also braided, then for the braiding the symbol $c$ will be
used.

Recall that an {\em admissible class $\mathcal S$ of spans} in an arbitrary
category was defined in \cite[Definition 2.1]{Bohm:Xmod_I}. 
The {\em pullback 
$$
\xymatrix{A\coten B C \ar[r]^-{p_C} \ar[d]_-{p_A} &
C \ar@{..>}[d]^-g \\
A \ar@{..>}[r]_-f &
B}
$$
of the cospan
$\xymatrix@C=15pt{
A \ar[r]^-f & B & \ar[l]_-g C}
$
relative to such a class $\mathcal S$} was introduced in
\cite[Definition 3.1]{Bohm:Xmod_I}. 
\cite[Assumption 4.1]{Bohm:Xmod_I} asserts that there exist the relative
pullbacks of those cospans whose {\em legs are in $\mathcal S$} in the sense of
\cite[Definition 2.9]{Bohm:Xmod_I}. 
Under this assumption it was proven in \cite[Corollary 4.6]{Bohm:Xmod_I} that 
the spans whose legs are in $\mathcal S$ (again in the sense of
\cite[Definition 2.9]{Bohm:Xmod_I}) constitute a monoidal category.
An {\em $\mathcal S$-relative category} is defined as a monoid therein, see 
\cite[Definition 4.9]{Bohm:Xmod_I}.

A class of spans in a monoidal category, which is compatible with
the monoidal structure --- meaning {\em multiplicativity} and {\em unitality}
in a natural sense --- was termed {\em monoidal} in \cite[Definition
2.5]{Bohm:Xmod_I}. It is discussed in \cite[Example 2.8]{Bohm:Xmod_I} that a
monoidal admissible class $\mathcal S$ of spans in a braided monoidal category
$\mathsf C$ induces a monoidal admissible class of spans in the category of
monoids in $\mathsf C$; and it is shown in \cite[Example 4.4]{Bohm:Xmod_I}
that if $\mathcal S$ satisfies \cite[Assumption 4.1]{Bohm:Xmod_I} then so does
the induced class in the category of monoids. This allows for the discussion
of relative categories in the category of monoids.

In this paper we will be interested mainly in these relative categories of
monoids. They contain, in particular, a split epimorphism of monoids
(consisting of the morphisms $i$ and $s$ of \eqref{eq:int_cat} in the
Introduction). So we  start with the analysis of the following category of
split epimorphisms of monoids. 

\begin{theorem} \label{thm:SplitEpi_vs_DLaw}
Consider a monoidal admissible class $\mathcal S$ of spans in a monoidal
category $\mathsf C$ for which \cite[Assumption 4.1]{Bohm:Xmod_I} holds. The
following categories are equivalent.
\begin{itemize}
\item[{$\mathsf{Split}$}]\hspace{-.3cm} $\mathsf{EpiMon}_{\mathcal S}(\mathsf
  C)$ whose \\
\underline{objects} are split epimorphisms 
$\xymatrix@C=15pt{
B \ar@<-2pt>@{ >->}[r]_-i &
A \ar@<-2pt>@{->>}[l]_-s}$ 
of monoids in $\mathsf C$ subject to the following conditions.
\begin{itemize}
\item[{(a)}] 
$\xymatrix@C=12pt{
A \ar@{=}[r] & A \ar[r]^-s & B}\in \mathcal S$; so that by the unitality of
  $\mathcal S$ and \cite[Assumption 4.1]{Bohm:Xmod_I}, there exists the $\mathcal
  S$-relative pullback
$$
\xymatrix{
A \coten B I \ar[r]^-{p_I} \ar[d]_-{p_A} &
I \ar[d]^-u \\
A \ar[r]_-s &
B.}
$$
\item[{(b)}]
$q:=\xymatrix@C=18pt{
(A \coten B I)B \ar[r]^-{p_Ai} &
A^2 \ar[r]^-m &
A}$ is invertible.
\end{itemize}
\underline{morphisms} are pairs of monoid morphisms 
$(\xymatrix@C=12pt{B \ar[r]^-b & B'},\xymatrix@C=12pt{A \ar[r]^-a & A'})$
such that $s'.a=b.s$ and $i'.b=a.i$.
\smallskip

\item[{$\mathsf{Dist}\hspace{.08cm}$}]\hspace{-.38cm} $\mathsf{Law}_{\mathcal
  S}(\mathsf C)$ whose \\ 
\underline{objects} consist of monoids $B$ and $Y$, a monoid morphism
$\xymatrix@C=12pt{Y \ar[r]^-e & I}$ and a distributive law
$\xymatrix@C=12pt{BY \ar[r]^-x & YB}$ subject to the following conditions.
\begin{itemize}
\item[{(a')}] 
$\xymatrix@C=12pt{
Y \ar@{=}[r] & Y \ar[r]^-e & I}\in \mathcal S$
and
$\xymatrix@C=12pt{
B \ar@{=}[r] & B \ar@{=}[r] & B}\in \mathcal S$.
Then by the monoidality of $\mathcal S$ also
$\xymatrix@C=12pt{
YB \ar@{=}[r] & Y \ar[r]^-{e1} & B}\in \mathcal S$
so by \cite[Assumption 4.1]{Bohm:Xmod_I} there exists the $\mathcal S$-relative
pullback $YB \coten B I$ in the diagram below.
\item[{(b')}] $e1.x=1e$.
\item[{(c')}] The morphism $f$ occurring in the diagram below is invertible.
(It is well-defined since by {\rm{(a')}} and condition {\rm{(POST)}} in 
\cite[Definition 2.1]{Bohm:Xmod_I}, 
$\xymatrix@C=12pt{
YB & \ar[l]_-{1u} Y \ar[r]^-e & I}\in \mathcal S$.)
\end{itemize}
$$
\xymatrix@C=10pt@R=10pt{
Y \ar@/^1.2pc/[rrrd]^-e \ar@/_1.2pc/[rddd]_-{1u} \ar@{-->}[rd]^-f \\
& YB \coten B I \ar[rr]^-{p_I} \ar[dd]_(.4){p_{YB}} &&
I \ar[dd]^-u \\
\\
& YB \ar[rr]_-{e1} &&
B.}
$$
\underline{morphisms} are pairs of monoid morphisms 
$(\xymatrix@C=12pt{B \ar[r]^-b & B'},\xymatrix@C=12pt{Y \ar[r]^-y & Y'})$
such that $e'.y=e$ and $x'.by=yb.x$.
\end{itemize}
\end{theorem}

\begin{proof}
We prove the theorem by constructing mutually inverse equivalence
functors. The first one
$\mathsf{SplitEpiMon}_{\mathcal S}(\mathsf C) \to 
\mathsf{DistLaw}_{\mathcal S}(\mathsf C)$ sends 
$$
\xymatrix{
B \ar@<-2pt>@{ >->}[r]_-i \ar[d]_-b&
A \ar@<-2pt>@{->>}[l]_-s \ar[d]^-a \\
B' \ar@<-2pt>@{ >->}[r]_-{i'} &
A' \ar@<-2pt>@{->>}[l]_-{s'}}
\quad \raisebox{-18pt}{$\mapsto$} \quad
\xymatrix@C=12pt@R=14pt{
(A\coten B I, \ar[d]_-{a\morcoten 1}  & 
\hspace{-.5cm}B, \ar[d]^-b \hspace{-.5cm} &
A\coten B I \ar[r]^-{p_I} & I,  &
\hspace{-.5cm} 
B(A\coten B I) \ar[r]^-{ip_A} & A^2 \ar[r]^-m & A \ar[r]^-{q^{-1}} & 
(A\coten B I)B)   \\
(A'\coten {B'} I, &
\hspace{-.5cm}  B', \hspace{-.5cm}  &  
A' \coten {B'} I \ar[r]^-{p_I} & I, &
\hspace{-.5cm} 
B'(A'\coten {B'} I) \ar[r]^-{i'p_{A'}} & A^{\prime 2} \ar[r]^-{m'} & A' 
\ar[r]^-{q^{\prime -1}} & 
(A'\coten {B'} I)B'). }
$$
Let us see that the object map is meaningful. By construction $B$ is a monoid
and $\xymatrix@C=12pt{B \ar[r]^-b & B'}$ is a monoid morphism.
By \cite[Proposition 3.7~(1)]{Bohm:Xmod_I} $A\coten B I$ is a monoid and 
$\xymatrix@C=10pt{A\coten B I \ar[r]^-{p_I} & I}$ is a monoid morphism. 
By \cite[Lemma 1.5]{Bohm:Xmod_I} 
$\xymatrix@C=10pt{
B(A\coten B I) \ar[r]^-{ip_A} & 
A^2 \ar[r]^-m 
& A \ar[r]^-{q^{-1}} & 
(A\coten B I)B}$
is a distributive law.
Concerning property (a'),
$\xymatrix@C=10pt{I \ar@{=}[r] & I \ar@{=}[r] & I}\in \mathcal S$
by 
the unitality of $\mathcal S$; hence by \cite[Lemma 3.4~(2)]{Bohm:Xmod_I}
$\xymatrix@C=12pt{A\coten B I \ar@{=}[r] & A \coten B I \ar[r]^-{p_I} & I}
\in \mathcal S$.
By \cite[Lemma 2.4~(1)]{Bohm:Xmod_I} also
$\xymatrix@C=10pt{
B\ar@{=}[r] & B \ar@{=}[r] & B}$ belongs to $ \mathcal S$. Condition (b')
holds since commutativity of the first diagram of
\begin{equation}\label{eq:s.q}
\xymatrix@C=15pt{
(A \coten B I) B \ar[r]_-{p_Ai} \ar[d]_-{p_I 1} \ar@/^1.1pc/[rr]^-q &
A^2 \ar[r]_-m \ar[d]^-{ss} & 
A \ar[d]^-s \\
B \ar[r]^-{u1} \ar@{=}@/_1.1pc/[rr] &
B^2\ar[r]^-m &
B}\qquad
\xymatrix{
B (A \coten B I) \ar[r]^-{ip_A}\ar[d]_-{1p_I} &
A^2 \ar[r]^-m \ar[d]^-{ss} &
A \ar[r]^-{q^{-1}} \ar[d]^-s &
 (A \coten B I) B\ar[d]^-{p_I 1} \\
B \ar[r]^-{1u} \ar@{=}@/_1.1pc/[rr] &
B^2\ar[r]^-m &
B \ar@{=}[r] &
B}
\vspace{.3cm}
\end{equation}
implies the commutativity of the second diagram. For condition (c') observe
that by the unitality of the monoid morphism $i$ the equality 
$q.1u=
p_A$ holds, equivalently, $q^{-1}.p_A=1u$. With this
identity in mind we see that the morphism $f$ of condition (c') is equal to
$q^{-1}\morcoten 1$ in the first diagram of
$$
\xymatrix{
A \coten B I \ar@/^1.2pc/[rrd]^-{p_I} \ar[dd]_-{p_A}
\ar@{-->}[rd]^-{q^{-1}\morcoten 1} \\
& (A \coten B I)B \coten B I \ar[r]^-{p_I} \ar[d]_-{p_{(A \diagcoten_B I)B}} &
I \ar[d]^-u \\
A \ar[r]_-{q^{-1}} & 
(A \coten B I)B\ar[r]_-{p_I1} &
B} \qquad
\xymatrix{
(A \coten B I)B \coten B I \ar@/^1.2pc/[rrd]^-{p_I} 
\ar[dd]_-{p_{(A \diagcoten_B I)B}} \ar@{-->}[rd]^-{q \morcoten 1}\\
& A \coten B I \ar[r]^-{p_I} \ar[d]_-{p_A} &
I \ar[d]^-u \\
(A \coten B I)B \ar[r]_-q &
A \ar[r]_-s & B.}
$$
Then by \cite[Proposition 3.5~(2)]{Bohm:Xmod_I} it is invertible with
the inverse $q\morcoten 1$ in the second diagram. Both morphisms
$q^{-1}\morcoten 1$ and $q\morcoten 1$ are well-defined by the commutativity
of the first diagram of \eqref{eq:s.q}; see 
\cite[Proposition 3.5~(1)]{Bohm:Xmod_I}. 
This proves that the object map of our candidate functor is meaningful.

Concerning the morphism map, $a\morcoten 1$ is a well-defined morphism in
$\mathsf C$ by the assumption that $b.s=s'.a$ (see 
\cite[Proposition 3.5~(1)]{Bohm:Xmod_I}) and it is a monoid morphism by
\cite[Proposition 3.7~(2)]{Bohm:Xmod_I}. Condition $p_I.(a\morcoten
1)=p_I$ holds by construction and the other equality holds since the
commutativity of the first diagram of 
\begin{equation}\label{eq:q_nat}
\xymatrix@C=15pt{
(A \coten B I) B \ar[r]_-{p_Ai} \ar[d]_-{(a\diagcoten 1)b} 
\ar@/^1.1pc/[rr]^-q &
A^2 \ar[r]_-m \ar[d]^-{aa} & 
A \ar[d]^-a \\
(A' \coten {B'} I) B' \ar[r]^-{p_{A'}i'} \ar@/_1.1pc/[rr]_-{q'} &
A^{\prime 2} \ar[r]^-{m'} & 
A'}\qquad
\xymatrix@C=15pt{
B(A \coten B I)  \ar[r]^-{i p_A} \ar[d]_-{b(a\diagcoten 1)} &
A^2 \ar[r]^-m \ar[d]^-{aa} & 
A \ar[d]^-a \ar[r]^-{q^{-1}} &
(A \coten B I) B  \ar[d]^-{(a\diagcoten 1)b} \\
B'(A' \coten {B'} I)  \ar[r]_-{i'p_{A'}}  &
A^{\prime 2} \ar[r]_-{m'} & 
A' \ar[r]_-{q^{\prime -1}} &
(A' \coten {B'} I) B'}
\end{equation}
implies the commutativity of the second diagram.

In the opposite direction
$\mathsf{DistLaw}_{\mathcal S}(\mathsf C) \to
\mathsf{SplitEpiMon}_{\mathcal S}(\mathsf C)$
we propose a functor sending
$$
\xymatrix@C=12pt@R=20pt{
(Y, \ar[d]_-y  & 
\hspace{-.5cm}B, \ar[d]^-b \hspace{-.5cm} &
Y \ar[r]^-e & I, &
\hspace{-.5cm} 
BY \ar[r]^-x & YB)   \\
(Y', &
\hspace{-.5cm}  B', \hspace{-.5cm}  &  
Y' \ar[r]^-{e'} & I, &
\hspace{-.5cm} 
B'Y' \ar[r]^-{x'} & Y'B')}
\quad \raisebox{-18pt}{$\mapsto$} \quad
\xymatrix{
B \ar@<-2pt>@{ >->}[r]_-{u1} \ar[d]_-b&
YB \ar@<-2pt>@{->>}[l]_-{e1} \ar[d]^-{yb} \\
B' \ar@<-2pt>@{ >->}[r]_-{u'1} &
A. \ar@<-2pt>@{->>}[l]_-{e'1}}
$$
Here $YB$ is considered with the monoid structure induced by the distributive
law $x$, see \cite[Lemma 1.4]{Bohm:Xmod_I}. Then 
$\xymatrix@C=12pt{B \ar[r]^-{u1} & YB}$ is a monoid morphism by \cite[Lemma
1.4]{Bohm:Xmod_I} again. By \cite[Lemma 1.6]{Bohm:Xmod_I} condition (b')
implies that 
$\xymatrix@C=12pt{YB \ar[r]^-{e1} & B}$ is a monoid morphism too.
The  rows are split epimorphisms (of monoids) by the unitality of the monoid
morphism $e$.
By (a') and the multiplicativity of $\mathcal S$, 
$\xymatrix@C=15pt{
YB \ar@{=}[r] &
YB \ar[r]^-{e1} &
B}\in \mathcal S$ 
so that condition (a) holds. For condition (b) note that the commutativity of 
$$
\xymatrix{
YB \ar@{=}[rrr]\ar[dd]_-{f1} \ar[rd]^-{1u1} &&&
YB \ar@{=}[dd]\\
& YB^2 \ar@{=}[r] \ar[d]^-{11u1} &
YB^2 \ar[d]^-{1u11} \ar[ru]^-{1m} \\
(YB \coten B I)B\ar[r]_-{p_{YB}u1} 
&
(YB)^2 \ar[r]_-{1x1} &
Y^2B^2 \ar[r]_-{mm} &
YB}
$$
implies that the bottom row is the inverse of the isomorphism $f1$ in the left column
hence it is invertible. This proves that the object map is well defined. 

Concerning the morphism map, it follows by the assumption $yb.x=x'.by$ that 
$yb$ is a monoid morphism, see \cite[Lemma 1.6]{Bohm:Xmod_I}. The monoid
morphisms $(b,yb)$ are compatible with the monomorphisms 
$\xymatrix@C=15pt{B \ar[r]^-{u1} & YB}$ and 
$\xymatrix@C=15pt{B' \ar[r]^-{u'1} & Y'B'}$ by the unitality of $y$ and they
are compatible with the epimorphisms 
$\xymatrix@C=15pt{YB \ar[r]^-{e1} & B}$ and 
$\xymatrix@C=15pt{Y'B' \ar[r]^-{e'1} & B'}$ by the assumption that $e'.y=e$.

So we have well-defined functors in both directions, it remains to see that
their composites are naturally isomorphic to the identity functors. 
The composite 
$$
\mathsf{SplitEpiMon}_{\mathcal S}(\mathsf C) \to
\mathsf{DistLaw}_{\mathcal S}(\mathsf C) \to
\mathsf{SplitEpiMon}_{\mathcal S}(\mathsf C)
$$
acts as
$$
\xymatrix{
B \ar@<-2pt>@{ >->}[r]_-i \ar[d]_-b&
A \ar@<-2pt>@{->>}[l]_-s \ar[d]^-a \\
B' \ar@<-2pt>@{ >->}[r]_-{i'} &
A \ar@<-2pt>@{->>}[l]_-{s'}}
\quad \raisebox{-18pt}{$\mapsto$} \quad
\xymatrix@R=16pt{
B \ar@<-2pt>@{ >->}[r]_-{u1} \ar[d]_-b&
(A\coten B I)B \ar@<-2pt>@{->>}[l]_-{p_I1} \ar[d]^-a \\
B' \ar@<-2pt>@{ >->}[r]_-{u'1} &
(A'\coten {B'} I)B' \ar@<-2pt>@{->>}[l]_-{p_I1}}
$$
We claim that a natural isomorphism from this to the identity functor has the
components  
$(\xymatrix@C=10pt{B \ar@{=}[r] & B},$ 
$\xymatrix@C=10pt{(A\coten B I)B\ar[r]^-q & A})$.
Since $p_A$ is a monoid morphism by \cite[Proposition 3.7~(1)]{Bohm:Xmod_I},
so is $q$ by \cite[Lemma 1.5]{Bohm:Xmod_I}.
The stated pair $(1,q)$ is a morphism in $\mathsf{SplitEpiMon}_{\mathcal
S}(\mathsf C)$ by the first diagram of \eqref{eq:s.q} and by the fact that the
unitality of $p_A$ implies $q.1u=i$. Naturality with respect to any morphism 
$(\xymatrix@C=12pt{B \ar[r]^-b & B'},\xymatrix@C=12pt{A \ar[r]^-a & A'})$ in
$\mathsf{SplitEpiMon}_{\mathcal S}(\mathsf C)$ follows by the commutativity of
the first diagram of \eqref{eq:q_nat}.

Composing our functors in the opposite order
$$
\mathsf{DistLaw}_{\mathcal S}(\mathsf C) \to
\mathsf{SplitEpiMon}_{\mathcal S}(\mathsf C) \to
\mathsf{DistLaw}_{\mathcal S}(\mathsf C)
$$
we obtain the functor sending
$$
\xymatrix@C=12pt@R=20pt{
(Y, \ar[d]_-y & 
\hspace{-4cm} B, \ar[d]^-b \hspace{-4cm} &
Y \ar[r]^-e & I, &
\hspace{-.5cm} 
BY \ar[r]^-x & YB)   \\
(Y', &
\hspace{-4cm}  B', \hspace{-4cm}  &  
Y' \ar[r]^-{e'} & I, &
\hspace{-.5cm} 
B'Y' \ar[r]^-{x'} & Y'B')}
$$
to
$$
\xymatrix@C=7pt@R=20pt{
(YB \coten B I, \ar[d]_-{yb\diagcoten 1}  & 
\hspace{-2cm}B, \ar[d]^-b \hspace{-2cm} &
YB\coten B I \ar[r]^-{p_I} & I, &
\hspace{-.5cm} 
B(YB \coten B I) \ar[r]^-{u1p_{YB}} & 
(YB)^2 \ar[r]^-{1x1} &
Y^2B^2 \ar[r]^-{mm} &
YB \ar[r]^-{f1} &
(YB \coten B I)B)   \\
(Y'B' \coten {B'} I,  & 
\hspace{-2cm}B', \hspace{-2cm} &
Y'B'\coten {B'} I \ar[r]^-{\raisebox{10pt}{${}_{p_I}$}} & I, &
\hspace{-.5cm} 
B'(Y'B' \coten {B'} I) \ar[r]^-{\raisebox{10pt}{${}_{u'1p_{Y'B'}}$}} & 
(Y'B')^2 \ar[r]^-{\raisebox{10pt}{${}_{1x'1}$}} &
Y^{\prime 2} B^{\prime 2} \ar[r]^-{\raisebox{10pt}{${}_{m'm'}$}} &
Y'B' \ar[r]^-{\raisebox{10pt}{${}_{f'1}$}} &
(Y'B' \coten {B'} I)B').}
$$
We claim that a natural isomorphism from this to the identity functor has the
invertible components 
$(\xymatrix@C=10pt{B \ar@{=}[r] & B},
\xymatrix@C=10pt{Y \ar[r]^-f & YB\coten B I})$.
By construction $f$ is a monoid morphism, see \cite[Proposition
  3.7~(2)]{Bohm:Xmod_I}. 
The compatibility of the monoid morphisms $(1,f)$ with 
$\xymatrix@C=12pt{Y \ar[r]^-e & I}$ and
$\xymatrix@C=12pt{ YB \coten B I\ar[r]^-{p_I} & I}$
holds by the definition of $f$ and the compatibility with the distributive
laws $BY \to YB$
and $B (YB \coten B I) \to  (YB \coten B I) B$
holds by the commutativity of
$$
\xymatrix{
BY \ar@{=}[r] \ar[d]_-{1f} &
BY \ar[rr]^-x \ar[d]^-{11u} && 
YB \ar[ld]_-{11u} \ar@{=}[d] \ar@/^1.2pc/[rd]^-{f1} \\
B (YB \coten B I)\ar[r]_-{1p_{YB}} & 
BYB \ar[r]_-{x1} &
YB^2 \ar[r]_-{1m} &
YB \ar[r]_-{f1} &
(YB \coten B I)B.}
$$
Finally, the naturality with respect to an arbitrary morphism 
$(\xymatrix@C=12pt{B \ar[r]^-b & B'},\xymatrix@C=12pt{Y \ar[r]^-y & Y'})$
in $\mathsf{DistLaw}_{\mathcal S}(\mathsf C)$ follows by the commutativity of
the diagrams 
$$
\xymatrix{
YB \coten B I \ar[rr]^-{yb\diagcoten 1} \ar[rd]^-{p_{YB}} &&
Y'B' \coten {B'} I \ar[d]^-{p_{Y'B'}} \\
Y \ar[u]^-f \ar[r]^-{1u} \ar[d]_-y &
YB \ar[r]^-{yb} &
Y'B' \\
Y'\ar[rru]^-{1u'} \ar[rr]_-{f'} &&
Y'B' \coten {B'} I \ar[u]_-{p_{Y'B'}}}\qquad
\xymatrix{
YB \coten B I \ar[r]^-{yb\diagcoten 1} \ar[rd]^-{p_I} &
Y'B' \coten {B'} I \ar[d]^-{p_I} \\
Y \ar[u]^-f \ar[r]^-e \ar[d]_-y &
I \\
Y'\ar[ru]^-{e'} \ar[r]_-{f'} &
Y'B' \coten {B'} I \ar[u]_-{p_I}}
$$
using again that 
$\xymatrix@C=15pt{YB & \ar[l]_-{p_{YB}} YB \coten B I \ar[r]^-{p_I} & I}$
are joint monomorphisms in $\mathsf C$.
\end{proof}

\begin{example} \label{ex:SplitEpi_groupoid}
For any fixed set $X$, the category $\mathsf C$ of spans over $X$ is monoidal
via the pullback over $X$. A monoid in $\mathsf C$ is a small category with
the object set $X$ and a monoid morphism is a functor acting on the objects as
the identity map.
Moreover, $\mathsf C$ has all pullbacks (computed in the underlying category
of sets). 
So taking as $\mathcal S$ the class of all spans in $\mathsf C$, from Theorem
\ref{thm:SplitEpi_vs_DLaw} we obtain the equivalence of the following
categories. (Throughout $s$ denotes the source map in any category and $t$
denotes the target map.)
\begin{itemize}
\item[{$\mathsf{Split}$}]\hspace{-.3cm} $\mathsf{EpiMon}(\mathsf C)$ whose \\
\underline{objects} are pairs of identity-on-objects functors
$\xymatrix@C=15pt{
B \ar@<-2pt>@{ >->}[r]_-\iota &
A \ar@<-2pt>@{->>}[l]_-\sigma}$ 
between categories of the common object set $X$
such that the composite $\sigma\iota$ is the identity functor, and the map
\begin{equation} \label{eq:q_groupoid}
q:(A\coten B X)\coten X B =\{(a,x,b) \vert \sigma(a)=1_x,\ x=t(b) \}
\to A\qquad
(a,x,b) \mapsto a.\iota(b)
\end{equation}
is invertible. (The morphism of \eqref{eq:q_groupoid} is invertible e.g. if $B$ is a groupoid; then its inverse takes a morphism $a$ to $(a.\iota(\sigma(a)^{-1}), t(a),\sigma(a))$.) \\
\underline{morphisms} are pairs of identity-on-objects functors
$(\xymatrix@C=12pt{A \ar[r]^-\alpha & A'},\xymatrix@C=12pt{B \ar[r]^-\beta & B'})$
for which $\alpha\iota=\iota'\beta$ and $\beta\sigma=\sigma'\alpha$.
\item[{$\mathsf{Dist}\hspace{.08cm}$}]\hspace{-.38cm} $\mathsf{Law}(\mathsf C)$ whose \\ 
\underline{objects} consist of categories $B$ and $Y$ with the common object set $X$ such that $Y$ has no morphisms between non-equal objects (that is, its source map $s$ and target map $t$ coincide); and an {\em action} 
$\xymatrix@C=12pt{
B \coten X Y =\{(b,y) \vert s(b)=t(y) \}
\ar[r]^-\triangleright & Y}$
in the sense of \cite[Definition 1.1]{BrownIcen}; meaning the following axioms
for all morphisms $b,b'$ in $B$ and $y,y'$ in $Y$ for which $s(b')=t(b)$ and
$s(b)=t(y)=s(y)=t(y')=s(y')$. 
\begin{itemize}
\item[{(i)}] $t(b\triangleright y)=t(b)$
\item[{(ii)}] $b\triangleright (y.y')=(b\triangleright y).(b\triangleright y')$
  and $b\triangleright 1_{s(b)}=1_{t(b)}$
\item[{(iii)}] $(b'.b)\triangleright y = b'\triangleright (b\triangleright y)$
  and $1_{t(y)}\triangleright y =y$.
\end{itemize}
\underline{morphisms} are pairs of identity-on-objects functors
$(\xymatrix@C=12pt{Y \ar[r]^-\nu & Y'},\xymatrix@C=12pt{B \ar[r]^-\beta & B'})$
for which $\nu(b\triangleright y)=\beta(b)\triangleright \nu(y)$ for all
morphisms $b$ in $B$ and $y$ in $Y$ for which $s(b)=t(y)$.
\end{itemize}

Only the above description of an object in 
$\mathsf{DistLaw}(\mathsf C)$ requires some explanation.

The monoidal unit of $\mathsf C$ is the trivial span 
$\xymatrix@C=12pt{X & \ar@{=}[l] X \ar@{=}[r] & X}$. Its trivial monoid
structure yields the discrete category $\mathsf D (X)$. An identity-on-objects
functor $\xymatrix@C=12pt{Y \ar[r]^-e  & \mathsf D (X)}$ as in Theorem
\ref{thm:SplitEpi_vs_DLaw} exists if and only if the source and target maps of
$Y$ coincide. Then there is precisely one such functor sending any morphism to
the identity morphism on its equal source and target objects. 
For this functor $e$, precisely those maps 
$\xymatrix@C=12pt{B\coten X Y \ar[r]^-x & Y \coten X B}$ satisfy 
$(e\morcoten 1).x=1\morcoten e$ which are of the form $(b,y)\mapsto
(b\triangleright y,b)$ in terms of some map $\triangleright$ obeying condition
(i). It is straightforward to see that then $x$ is a distributive law if and
only if conditions (ii) and (iii) hold.  

The morphism $f$ of Theorem \ref{thm:SplitEpi_vs_DLaw}~(c') is invertible
because
\begin{equation} \label{eq:hn_groupoid}
\xymatrix{
Y \coten X C \ar[r]^-{e\diagcoten 1} \ar[d]_-{1\diagcoten g} &
C \ar[d]^-g \\
Y \coten X B \ar[r]_-{e\diagcoten 1} &
B}
\end{equation}
is clearly a pullback of $X$-spans for any span morphism $g$. 
\end{example}

\begin{example} \label{ex:CoMon_SplitEpi}
Let $\mathsf M$ be a symmetric monoidal category in which equalizers exist and are preserved by taking the monoidal product with any object. 

Take $\mathsf C$ to be the category of comonoids in $\mathsf M$ with the
monoidal admissible class $\mathcal S$ in \cite[Example 2.3]{Bohm:Xmod_I} of spans in $\mathsf C$. 
Thanks to the symmetry of $\mathsf M$, its monoidal structure is inherited by
$\mathsf C$. A monoid $A$ in $\mathsf C$ is known as a {\em bimonoid} in
$\mathsf M$. Recall that the monoidal structure of $\mathsf M$ is lifted to
the category of (left or right) modules over the monoid $A$ in $\mathsf M$. A
monoid (respectively, a comonoid) in the category of $A$-modules is known as
an {\em $A$-module monoid} (respectively, {\em $A$-module comonoid}).

Recall from \cite[Example 3.3]{Bohm:Xmod_I} that for a cospan
$\xymatrix@C=12pt{A \ar[r]^-f & B & \ar[l]_-g C}$ 
of comonoids whose legs are in $\mathcal S$, the $\mathcal S$-relative pullback is given by the so-called {\em cotensor product}, defined as the equalizer
\begin{equation}\label{eq:cotensor}
\xymatrix{
A \coten B C \ar[r]^-j &
AC \ar@<2pt>[rr]^-{1f1.\delta 1} \ar@<-2pt>[rr]_-{1g1.1\delta } &&
ABC}
\end{equation}
in $\mathsf M$ (where $\delta$ denotes both comultiplications of the comonoids $A$ and $C$.)

Below we describe the equivalent categories of Theorem \ref{thm:SplitEpi_vs_DLaw} in this context. 
\begin{itemize}
\item[{$\mathsf{Split}$}]\hspace{-.3cm} $\mathsf{EpiMon}_{\mathcal S}(\mathsf
  C)$ whose \\
\underline{objects} are split epimorphisms 
$\xymatrix@C=15pt{
B \ar@<-2pt>@{ >->}[r]_-i &
A \ar@<-2pt>@{->>}[l]_-s}$ 
of bimonoids in $\mathsf M$ subject to the following conditions.
\begin{itemize}
\item[{(a)}] The comultiplication $\delta$ of $A$ satisfies
  $c.s1.\delta=1s.\delta$. 
\item[{(b)}] In terms of the morphism $j$ of \eqref{eq:cotensor}, 
$q:=\xymatrix@C=15pt{
(A \coten B I)B \ar[r]^-{ji} & A^2 \ar[r]^-m &A}$
is invertible.
\end{itemize}
\underline{morphisms} are pairs of bimonoid morphisms which are compatible with
the epimorphisms $s$ as well as their sections $i$.
\item[{$\mathsf{Dist}\hspace{.08cm}$}]\hspace{-.38cm} $\mathsf{Law}_{\mathcal
  S}(\mathsf C)$ whose \\ 
\underline{objects} consist of a cocommutative bimonoid $B$ and a bimonoid
$Y$ in $\mathsf M$, together with a left $B$-action on $Y$ which makes $Y$
both a left $B$-module monoid and a left $B$-module comonoid. \\
\underline{morphisms} are pairs of bimonoid morphisms 
$(\xymatrix@C=12pt{B \ar[r]^-b &B'},
\xymatrix@C=12pt{Y \ar[r]^-y &Y'})$
which are compatible with the actions 
$\xymatrix@C=12pt{BY \ar[r]^-l &Y}$ and
$\xymatrix@C=12pt{B'Y' \ar[r]^-{l'} &Y'}$ in the sense that $l'.by=y.l$.
\end{itemize}

This concise description of $\mathsf{DistLaw}_{\mathcal S}(\mathsf C)$
requires a proof. Note that the monoidal unit $I$ is now a terminal object in
$\mathsf C$; the unique morphism $Y \to I$ is the counit $\varepsilon$. It
obviously satisfies 
$\xymatrix@C=12pt{
Y \ar@{=}[r] & Y \ar[r]^-\varepsilon & I}\in \mathcal S$.
The other condition
$\xymatrix@C=12pt{
B \ar@{=}[r] & B \ar@{=}[r] & B}\in \mathcal S$
in (a') of Theorem \ref{thm:SplitEpi_vs_DLaw} reduces to the requirement that
the comonoid $B$ is cocommutative. 

Next we establish a bijective correspondence between distributive laws $BY
\to YB$ satisfying property (b') of Theorem \ref{thm:SplitEpi_vs_DLaw}
and left actions $BY\to Y$ as in the description above. Starting with a
distributive law 
$\xymatrix@C=12pt{BY \ar[r]^-x & YB}$, put $l:=1\varepsilon.x$.
It is a unital action by the left unitality of $x$ and it is associative by
the left multiplicativity of $x$: 
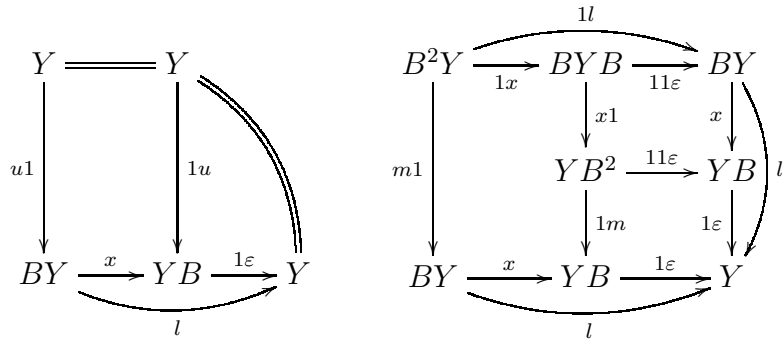
$$
\xymatrix@R=64pt{
Y \ar@{=}[r] \ar[d]_-{u1} &
Y \ar[d]^-{1u} \ar@{=}@/^1.2pc/[rd] \\
BY \ar[r]^-x \ar@/_1.1pc/[rr]_-l &
YB \ar[r]^-{1 \varepsilon} &
Y}\qquad
\xymatrix{
B^2Y \ar[r]_-{1x} \ar@/^1.1pc/[rr]^-{1l} \ar[dd]_-{m1} &
BYB \ar[r]_-{11\varepsilon} \ar[d]^-{x1} &
BY \ar[d]_-x \ar@/^1.1pc/[dd]^-l \\
& YB^2 \ar[r]^-{11\varepsilon} \ar[d]^-{1m} &
YB \ar[d]_-{1\varepsilon} \\
BY \ar[r]^-x \ar@/_1.1pc/[rr]_-l &
YB \ar[r]^-{1\varepsilon} &
Y}
$$
By the right unitality of $x$ the unit 
$\xymatrix@C=12pt{I \ar[r]^-u & Y}$ is a morphism of $B$-modules and 
by the right multiplicativity of $x$ the multiplication 
$\xymatrix@C=12pt{Y^2 \ar[r]^-m & Y}$ is a morphism of $B$-modules:
$$
\xymatrix@R=83pt{
B \ar[r]^-{1u} \ar@{=}[d] &
BY \ar[d]_-x \ar@/^1.2pc/[dd]^-l \\
B \ar[r]^-{u1} \ar[d]_-\varepsilon &
YB \ar[d]_-{1\varepsilon} \\
I \ar[r]_-u &
Y}\qquad
\xymatrix@C=20pt@R=15pt{
BY^2 \ar@{=}[rr] \ar[d]_-{\delta 11} &&
BY^2 \ar[rr]^-{1m} \ar[ddd]^-{x1} &&
BY \ar[ddddd]_-x \ar@/^2pc/[dddddd]^-l \\
B^2Y^2 \ar[r]^-{11\delta 1} \ar@{=}[d] &
B^2Y^3 \ar[ld]^-{111\varepsilon 1} \ar[dd]^-{1c11}\\
B^2Y^2 \ar[d]_-{1c1} \\
(BY)^2 \ar[d]^-{x11} \ar@/_2pc/[ddd]_-{ll} &
(BY)^2Y \ar[l]_-{111\varepsilon 1} \ar[d]^-{xx1} \ar@{}[ld]|-{\mathrm{(b')}} &
YBY \ar[d]_-{\delta\delta 1} \ar@/^1.2pc/@{=}[rd] \\
YB^2Y \ar[d]^-{11x} &
(YB)^2Y \ar[l]^-{11\varepsilon 11} &
Y^2B^2Y \ar[l]^-{1c11} \ar[r]_-{1\varepsilon\varepsilon 11} &
YBY \ar[d]^-{1x} \\
(YB)^2 \ar[d]^-{1\varepsilon 1 \varepsilon} &&&
Y^2B \ar[r]^-{m1} \ar[d]^-{11\varepsilon} &
YB \ar[d]_-{1\varepsilon} \\
Y^2 \ar@{=}[rrr] &&&
Y^2 \ar[r]_-m &
Y}
$$
(note that here we also used the comultiplicativity of $x$).
The condition that the counit 
$\xymatrix@C=12pt{Y \ar[r]^-\varepsilon & I}$ 
is a morphism of $B$-modules coincides with the counitality of $l$ and also
with the counitality of $x$. The comultiplication 
$\xymatrix@C=12pt{Y \ar[r]^-\delta & Y^2}$ 
is a morphism of $B$-modules, equivalently, $l$ is comultiplicative by the
comultiplicativity of $x$:
$$
\xymatrix@R=15pt{
BY \ar[r]^-{1\delta} \ar[d]^-x \ar@/_2pc/[dd]_-l &
BY^2 \ar[r]^-{\delta 1} &
B^2Y^2 \ar[r]^-{1c1} &
(BY)^2 \ar[d]_-{xx} \ar@/^2pc/[dd]^-{ll} \\
YB \ar[r]^-{\delta 1} \ar[d]^-{1\varepsilon} &
Y^2B \ar[r]^-{11\delta} \ar[d]^-{11\varepsilon} &
Y^2B^2 \ar[r]^-{1c1} \ar[d]^-{11\varepsilon\varepsilon} &
(YB)^2 \ar[d]_-{1\varepsilon 1 \varepsilon} \\
Y \ar[r]_-\delta &
Y^2 \ar@{=}[r] &
Y^2 \ar@{=}[r] &
Y^2}
$$

Conversely, in terms of an action $l$ as above, put 
$x:=\xymatrix@C=12pt{BY \ar[r]^-{\delta 1} &
B^2Y \ar[r]^-{1c} &
BYB \ar[r]^-{l1} &
YB.}$ 
It clearly satisfies (b') by the counitality of $l$ hence it is counital. 
It is comultiplicative by  the comultiplicativity of $l$:
$$
\xymatrix@R=15pt{
BY \ar[rr]_-{\delta 1} \ar[d]_-{\delta\delta} \ar@/^1.5pc/[rrrr]^-x &&
B^2Y \ar[r]_-{1c} \ar[d]^-{\delta\delta\delta} &
BYB \ar[r]_-{l1} \ar[d]^-{\delta\delta\delta} &
YB \ar[dd]^-{\delta\delta} \\
B^2Y^2 \ar[r]_-{\delta\delta 11} \ar[dd]_-{1c1} &
B^4Y^2 \ar[r]_-{1c111} \ar[dd]^-{11c_{B^2,Y}1} &
B^4Y^2 \ar[r]_-{11c_{B^2,Y^2}} &
B^2Y^2B^2 \ar[d]^-{1c111} \\
&&& (BY)^2B^2 \ar[r]^-{ll11} \ar[d]^-{11c_{BY,B}1} &
Y^2B^2 \ar[d]^-{1c1} \\
(BY)^2 \ar[r]^-{\delta 1\delta 1} \ar@/_1.5pc/[rrrr]_-{xx} &
(B^2Y)^2 \ar[rr]^-{1c1c} &&
(BYB)^2 \ar[r]^-{l1l1} &
(YB)^2}
$$
where the top-left region commutes by the coassociativity and cocommutativity
of the comonoid $B$. This morphism $x$ is a distributive law. Indeed, the left
unitality and the left multiplicativity follow by the unitality and the
associativity of the action $l$, respectively: 
$$
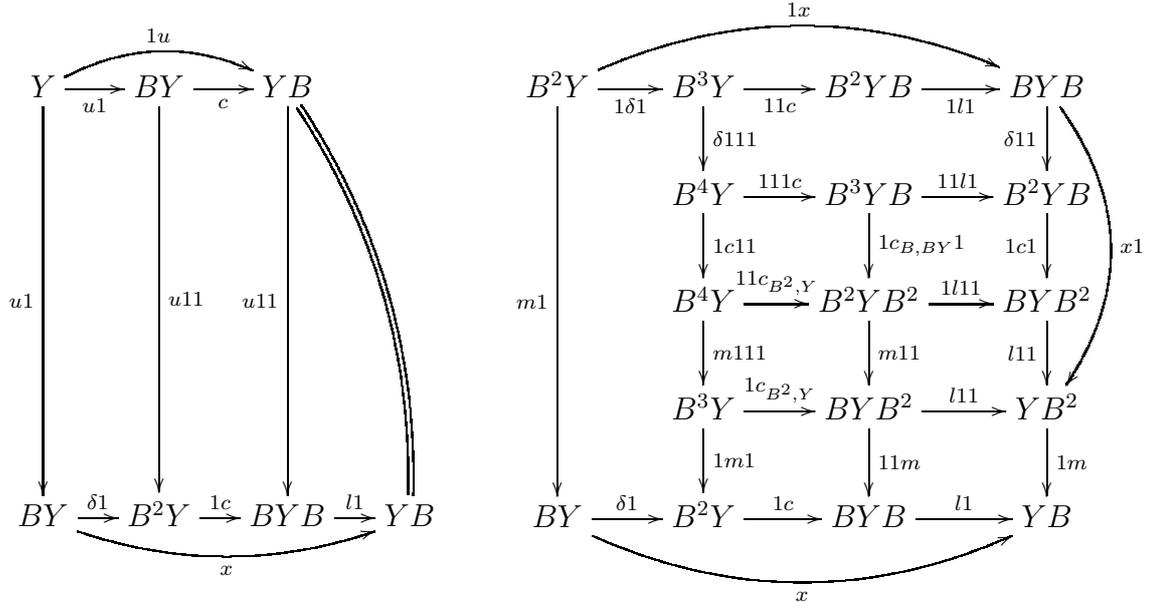
\xymatrix@C=15pt@R=146pt{
Y \ar[r]_-{u1} \ar[d]_-{u1} \ar@/^1.2pc/[rr]^-{1u}&
BY \ar[r]_-c \ar[d]^-{u11} &
YB \ar[d]_-{u11} \ar@/^1.2pc/@{=}[rd] \\
BY \ar[r]^-{\delta 1} \ar@/_1.2pc/[rrr]_-x &
B^2Y \ar[r]^-{1c} &
BYB \ar[r]^-{l1} &
YB}\qquad
\xymatrix{
B^2Y \ar[r]_-{1\delta 1} \ar[dddd]_-{m1} \ar@/^2pc/[rrr]^-{1x} &
B^3Y \ar[r]_-{11c} \ar[d]^-{\delta 111} &
B^2YB \ar[r]_-{1l1} &
BYB \ar[d]_-{\delta 11} \ar@/^2pc/[ddd]^-{x1} \\
& B^4Y \ar[r]^-{111c} \ar[d]^-{1c11} &
B^3YB \ar[r]^-{11l1} \ar[d]^-{1c_{B,BY}1} &
B^2YB \ar[d]_-{1c1} \\
& B^4Y \ar[r]^-{11c_{B^2,Y}} \ar[d]^-{m111} &
B^2YB^2 \ar[r]^-{1l11} \ar[d]^-{m11} &
BYB^2 \ar[d]_-{l11} \\
& B^3Y \ar[r]^-{1c_{B^2,Y}} \ar[d]^-{1m1} &
BYB^2 \ar[r]^-{l11} \ar[d]^-{11m} &
YB^2 \ar[d]^-{1m} \\
BY \ar[r]^-{\delta 1} \ar@/_2pc/[rrr]_-x &
B^2Y \ar[r]^-{1c} &
BYB \ar[r]^-{l1} &
YB}
$$
and the right unitality and the right multiplicativity of $x$ follow using that
the unit and the multiplication of $Y$ are $B$-module morphisms:
$$
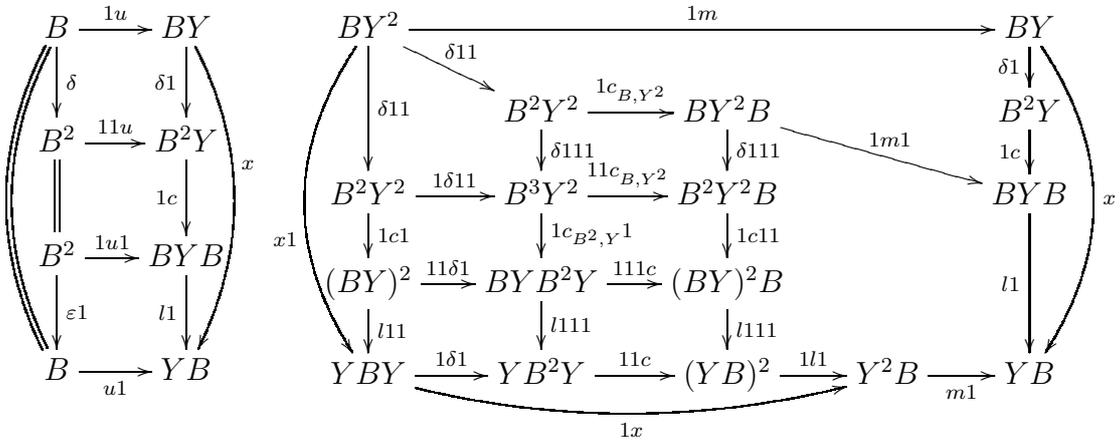
\xymatrix@R=27pt@C=20pt{
B \ar[r]^-{1u} \ar[d]^-\delta \ar@/_1.5pc/@{=}[ddd] &
BY \ar[d]_-{\delta 1} \ar@/^1.5pc/[ddd]^(.4)x \\
B^2 \ar[r]^-{11u} \ar@{=}[d] &
B^2Y \ar[d]_-{1c} \\
B^2 \ar[r]^-{1u1} \ar[d]^-{\varepsilon 1} &
BYB \ar[d]_-{l1} \\
B \ar[r]_-{u1} &
YB}
\xymatrix@R=15pt@C=20pt{
BY^2 \ar[rrrr]^-{1m} \ar[rd]^-{\delta 11} \ar[dd]^-{\delta 11} 
\ar@/_2pc/[dddd]_(.6){x1} &&&&
BY \ar[d]_-{\delta 1} \ar@/^2pc/[dddd]^-x \\
& 
B^2Y^2 \ar[r]^-{1c_{B,Y^2}} \ar[d]^-{\delta 111} &
BY^2B \ar[d]^-{\delta 111} \ar[rrd]^-{1m1} &&
B^2Y \ar[d]_-{1c} \\
B^2Y^2 \ar[r]^-{1\delta 11} \ar[d]^-{1c1} &
B^3Y^2 \ar[r]^-{11c_{B,Y^2}} \ar[d]^-{1c_{B^2,Y}1}  &
B^2Y^2B \ar[d]^-{1c11} &&
BYB \ar[dd]_-{l1} \\
(BY)^2 \ar[r]^-{11\delta 1} \ar[d]^-{l11} &
BYB^2Y \ar[r]^-{111c} \ar[d]^-{l111} &
(BY)^2 B \ar[d]^-{l111} \\
YBY \ar[r]^-{1\delta 1} \ar@/_1.2pc/[rrr]_-{1x} &
YB^2Y \ar[r]^-{11c} &
(YB)^2 \ar[r]^-{1l1} &
Y^2B \ar[r]_-{m1} &
YB}
$$

The above correspondences between  $l$ and $x$ are bijective by the
commutativity of 
$$
\xymatrix@R=21pt{
& B^2Y \ar[rr]^-{1c} &&
BYB \ar[dd]^-{x1} \\
BY \ar@/^1.2pc/[ru]^-{\delta 1} \ar[r]^-{\delta\delta} \ar[d]_-x &
B^2Y^2  \ar[r]^-{1c1} \ar[u]_-{111 \varepsilon}&
(BY)^2 \ar[ru]^-{111 \varepsilon} \ar[d]_-{xx} \\
YB \ar[r]^-{\delta\delta} \ar@/_1.2pc/@{=}[rd] &
Y^2B^2 \ar[r]^-{1c1} \ar[d]^-{1\varepsilon \varepsilon 1} &
(YB)^2 \ar[rd]^-{1\varepsilon \varepsilon 1} &
YB^2 \ar[d]^-{1\varepsilon 1}\\
& YB \ar@{=}[rr] &&
YB} \qquad
\xymatrix{
BY \ar[d]_-{\delta 1} \ar@/^1.2pc/@{=}[rd] \\
B^2Y \ar[d]_-{1c} \ar[r]_-{1\varepsilon 1}  &
BY \ar@{=}[d] \\
BYB \ar[r]^-{11\varepsilon } \ar[d]_-{l1} &
BY \ar[d]^-l \\
YB \ar[r]_-{1\varepsilon } &
Y}
$$
for a comultiplicative morphism $x$ satisfying (b') and any morphism $l$.

Finally, we show that the morphism 
$\xymatrix@C=12pt{Y \ar[r]^-f & YB \coten B I}$ 
in part (c') of Theorem \ref{thm:SplitEpi_vs_DLaw} is invertible without any
further assumption; its inverse is constructed as  
$f^{-1}:= \xymatrix@C=15pt{
YB \coten B I \ar[r]^-{p_{YB}} & YB \ar[r]^-{1\varepsilon} & Y}$.
In order to see that it is the inverse, indeed, recall that by 
\cite[Example 3.3]{Bohm:Xmod_I} the morphism $p_{YB}$ is the equalizer of 
$\xymatrix@C=15pt{YB \ar[r]^-{1\delta} & YB^2}$ and 
$\xymatrix@C=15pt{YB \ar[r]^-{11u} & YB^2}$. Hence the following diagrams
commute.
$$
\xymatrix{
Y \ar@{=}[r]\ar[d]_-f &
Y \ar[d]^-{1u} \ar@/^1.2pc/@{=}[rd] \\
YB \coten B I \ar[r]^-{p_{YB}} \ar@/_1.5pc/[rr]_-{f^{-1}} &
YB \ar[r]^-{1\varepsilon} & 
Y}\qquad
\xymatrix{
YB \coten B I \ar[r]_-{p_{YB}} \ar@/^1.5pc/[rr]^-{f^{-1}} \ar[d]_-{p_{YB}} &
YB \ar[r]_-{1\varepsilon} \ar[d]^-{11u} & 
Y \ar[d]^-{1u} \ar[r]^-f &
YB \coten B I \ar[d]^-{p_{YB}} \\
YB \ar[r]^-{1\delta} \ar@{=}@/_1.5pc/[rr] &
YB^2 \ar[r]^-{1\varepsilon 1} &
YB \ar@{=}[r] &
YB}
$$
This completes the characterization of the objects of
$\mathsf{DistLaw}_{\mathcal S}(\mathsf C)$. 
Concerning the morphisms 
$(\xymatrix@C=12pt{B\ar[r]^-b & B'},\xymatrix@C=12pt{Y\ar[r]^-y & Y'})$, the
first condition in Theorem \ref{thm:SplitEpi_vs_DLaw} is the counitality of
the bimonoid morphism $y$ hence it identically holds. The second condition 
in Theorem \ref{thm:SplitEpi_vs_DLaw} is equivalent to $y.l=l'.by$ by the
commutativity of
$$
\xymatrix@C=15pt{
BY \ar[r]_-{\delta 1} \ar[d]_-{by} \ar@/^1.2pc/[rrr]^-x&
B^2 Y \ar[r]_-{1c} \ar[d]^-{bby} &
BYB \ar[r]_-{l1} \ar[d]^-{byb} &
YB \ar[d]^-{yb} \\
B'Y' \ar[r]^-{\delta' 1}  \ar@/_1.2pc/[rrr]_-{x'}&
B^{\prime 2} Y' \ar[r]^-{1c} &
B'Y'B' \ar[r]^-{l'1} &
Y'B'}\qquad
\xymatrix{
BY \ar[r]_-x \ar[d]_-{by} \ar@/^1.2pc/[rr]^-l&
YB \ar[r]_-{1\varepsilon} \ar[d]^-{yb} &
Y \ar[d]^-y \\
B'Y'\ar[r]^-{x'} \ar@/_1.2pc/[rr]_-{l'} &
Y'B' \ar[r]^-{1\varepsilon'} &
Y'.}
$$

We can apply the current example to the particular case of a finitely complete
category $\mathsf M$ regarded with the Cartesian monoidal structure. Then the
category $\mathsf C$ of comonoids in $\mathsf M$ is isomorphic to $\mathsf M$
and the equivalent categories of Theorem \ref{thm:SplitEpi_vs_DLaw} reduce to
the following ones. 
\begin{itemize}
\item[{$\mathsf{Split}$}]\hspace{-.3cm} $\mathsf{EpiMon}_{\mathcal S}(\mathsf
  M)$ whose \\
\underline{objects} are split epimorphisms 
$\xymatrix@C=15pt{
B \ar@<-2pt>@{ >->}[r]_-i &
A \ar@<-2pt>@{->>}[l]_-s}$ 
of monoids in $\mathsf M$ such that in terms of the morphism $j$ of
\eqref{eq:cotensor},  
$q:=\xymatrix@C=15pt{
(A \coten B I)B \ar[r]^-{ji} & A^2 \ar[r]^-m &A}$
is invertible.\\
\underline{morphisms} are pairs of monoid morphisms which are compatible with
the epimorphisms $s$ as well as their sections $i$.
\item[{$\mathsf{Dist}\hspace{.08cm}$}]\hspace{-.38cm} $\mathsf{Law}_{\mathcal
  S}(\mathsf M)$ whose \\ 
\underline{objects} consist of monoids $B$ and $Y$ in $\mathsf M$, together
with a left $B$-action on $Y$ which makes the multiplication and the unit of
the monoid $Y$ left $B$-linear. \\
\underline{morphisms} are pairs of monoid morphisms 
$(\xymatrix@C=12pt{B \ar[r]^-b &B'},
\xymatrix@C=12pt{Y \ar[r]^-y &Y'})$
which are compatible with the actions 
$\xymatrix@C=12pt{BY \ar[r]^-l &Y}$ and
$\xymatrix@C=12pt{B'Y' \ar[r]^-{l'} &Y'}$ in the sense that $l'.by=y.l$.
\end{itemize}
\end{example}

Recall that a bimonoid $B$ --- with monoid structure $(m,u)$ and comonoid
structure $(\delta,\varepsilon)$  --- is a {\em Hopf monoid} provided that
there exists a morphism $\xymatrix@C=12pt{B \ar[r]^-z & B}$ --- the so-called
{\em antipode} --- which renders commutative 
$$
\xymatrix@R=5pt{
B \ar[r]^-\delta \ar[dd]_-\delta \ar[rd]_-\varepsilon & 
B^2 \ar[r]^-{z1} 
& B^2 \ar[dd]^-m \\
& I \ar[rd]^-u \\
B^2 \ar[r]_-{1z} & 
B^2 \ar[r]_-m &
B.}
$$
If the antipode exists then it is unique. It is a monoid 
morphism from $B$ to the monoid with the opposite multiplication $m.c$ and
comonoid morphism from $B$ to the comonoid with the opposite comultiplication
$c.\delta$. 

\begin{proposition} \label{prop:SplitEpi_Hopf}
\begin{itemize}
  \item[{(1)}] The equivalent categories of Example \ref{ex:CoMon_SplitEpi} have equivalent
full subcategories as follows.
   \begin{itemize}
      \item[{$\bullet$}] The category whose \\
               \underline{objects} are split epimorphisms 
              $\xymatrix@C=15pt{
              B \ar@<-2pt>@{ >->}[r]_-i &
              A \ar@<-2pt>@{->>}[l]_-s}$ 
              of bimonoids in $\mathsf M$ subject to the following conditions.
              \begin{itemize}
                 \item[{(a)}] The comultiplication $\delta$ of $A$ satisfies
                   $c.s1.\delta=1s.\delta$. 
                 \item[{(b)}] $B$ is a Hopf monoid.
             \end{itemize}
             \underline{morphisms} are pairs of bimonoid morphisms which are compatible with the epimorphisms $s$ as well as their sections $i$.
      \item[{$\bullet$}]   The category whose \\ 
              \underline{objects} consist of a cocommutative Hopf monoid $B$ and a bimonoid $Y$ in $\mathsf M$, together with a left $B$-action on $Y$ 
              which makes $Y$ both a left $B$-module monoid and a left $B$-module comonoid. \\ 
              \underline{morphisms} are pairs of bimonoid morphisms 
              $(\xymatrix@C=12pt{B \ar[r]^-b &B'},
              \xymatrix@C=12pt{Y \ar[r]^-y &Y'})$
              which are compatible with the actions 
              $\xymatrix@C=12pt{BY \ar[r]^-l &Y}$ and
              $\xymatrix@C=12pt{B'Y' \ar[r]^-{l'} &Y'}$ in the sense that $l'.by=y.l$.
   \end{itemize}
  \item[{(2)}] The equivalent categories of part (1) have equivalent
full subcategories as follows.
\begin{itemize}
      \item[{$\bullet$}]  The category whose \\
     \underline{objects} are split epimorphisms 
     $\xymatrix@C=15pt{
     B \ar@<-2pt>@{ >->}[r]_-i &
     A \ar@<-2pt>@{->>}[l]_-s}$ 
     of cocommutative Hopf monoids.
     \underline{morphisms} are pairs of bimonoid morphisms which are compatible with the epimorphisms $s$ as well as their sections $i$.
     \item[{$\bullet$}]   The category whose \\ 
     \underline{objects} consist of cocommutative Hopf monoids $B$ and $Y$ in $\mathsf M$, together with a left $B$-action on $Y$ which makes $Y$
     both a left $B$-module monoid and a left $B$-module comonoid. \\
     \underline{morphisms} are pairs of bimonoid morphisms 
     $(\xymatrix@C=12pt{B \ar[r]^-b &B'},       
      \xymatrix@C=12pt{Y \ar[r]^-y &Y'})$
      which are compatible with the actions 
      $\xymatrix@C=12pt{BY \ar[r]^-l &Y}$ and
      $\xymatrix@C=12pt{B'Y' \ar[r]^-{l'} &Y'}$ in the sense that $l'.by=y.l$.
\end{itemize}
\end{itemize}
\end{proposition}

\begin{proof}
(1) The second listed category is obviously a full subcategory of $\mathsf{DistLaw}_{\mathcal S}(\mathsf C)$ of Example \ref{ex:CoMon_SplitEpi}; thus via the equivalence of Theorem \ref{thm:SplitEpi_vs_DLaw} it is equivalent to some full subcategory of $\mathsf{SplitEpiMon}_{\mathcal S}(\mathsf C)$ of Example \ref{ex:CoMon_SplitEpi}. Our task is to show that it is the first listed category above. For that we only need to show that it is a subcategory of $\mathsf{SplitEpiMon}_{\mathcal S}(\mathsf C)$; that is, that for any object 
$\xymatrix@C=15pt{
B \ar@<-2pt>@{ >->}[r]_-i &
A \ar@<-2pt>@{->>}[l]_-s}$ 
of it, the morphism $q$ in part (b) of Example \ref{ex:CoMon_SplitEpi} is invertible. Following ideas in \cite{Radford}, we use the antipode $z$ of $B$ and the image of the equalizer \eqref{eq:cotensor} under the functor $-B$ to construct the inverse:
$$
\xymatrix{
&&&&&& (A \coten B I)B \ar[d]^-{j1} \\
A \ar[r]^(.7)\delta \ar@{-->}@/^1.5pc/[rrrrrru]^-{q^{-1}}&
A^2 \ar[r]^-{1s} &
AB \ar[r]^-{1\delta} &
AB^2 \ar[r]^-{1z1} &
AB^2 \ar[r]^-{1i1} &
A^2B \ar[r]^-{m1} & 
AB \ar@<-2pt>[d]_-{1s1.\delta 1} \ar@<2pt>[d]^-{1u1} \\
&&&&&& AB^2}
$$
This definition works because the horizontal morphism equalizes the parallel morphisms of the fork on the right; see Figure \ref{fig:qinv}.
\begin{figure} 
\centering
\begin{sideways}
$\xymatrix{
A \ar[rr]^-\delta \ar[dd]_-\delta &&
A^2 \ar[r]^-{1s} \ar[d]^-{1\delta} &
AB \ar[r]^-{1\delta} &
AB^2 \ar[rr]^-{1z1} \ar[dd]^-{\delta \delta 1} &&
AB^2 \ar[rr]^-{1i1} \ar[dd]^-{\delta \delta 1} &&
A^2B \ar[rr]^-{m1} \ar[dd]^-{\delta \delta 1} &&
AB \ar[dd]^-{\delta 1} \\
&& A^3 \ar[d]^-{\delta \delta 1} \\
A^2 \ar[r]^-{1\delta} \ar[d]_-{1s} &
A^3 \ar[r]^-{1\delta \delta} &
A^5 \ar[rr]^-{11sss} &&
A^2B^3 \ar[r]^-{11c1} \ar[dd]^-{1s111} &
A^2B^3 \ar[r]^-{11zz1} &
A^2B^3 \ar[rr]^-{11ii1} \ar[rrdd]^-{1si11} &&
A^4B \ar[r]^-{1c11} \ar[dd]^-{1s1s1} &
A^4B \ar[r]^-{mm1} \ar[dd]^-{11ss1} &
A^2B \ar[ddddddd]^-{1s1} \\
AB \ar[rd]^-{1\delta} \ar@{=}[dd] \\
& AB^2 \ar[r]^-{11\delta} \ar[ld]^-{1\varepsilon 1} &
AB^3 \ar[rr]^-{1\delta 11} \ar[lldd]^-{1\varepsilon 11} &&
AB^4 \ar[r]^-{11z11} &
AB^4 \ar[r]^-{111z1} \ar[dd]^-{1m11} &
AB^4 \ar[r]^-{111i1} &
AB^2AB \ar[r]^-{11c1} \ar[dd]^-{1m11} &
(AB)^2B \ar[r]^-{1c11} &
A^2B^3 \ar[rddddd]_-{mm1} \\
AB \ar[d]_-{1\delta} \\
AB^2 \ar[rrrrr]^-{1u11} \ar[d]_-{1z1} &&&&&
AB^3 \ar[r]^-{11z1} &
AB^3 \ar[r]^-{11i1} &
(AB)^2 \ar[rrdd]^-{1c1} \\
AB^2 \ar[d]_-{1i1} \\
A^2B \ar[d]_-{m1} \ar[rrrrrrruu]^-{1u11} \ar[rrrrrrrrr]^-{11u1} &&&&&&&&&
A^2B^2 \ar[rd]^-{m1} \\
AB \ar[rrrrrrrrrr]_-{1u1} &&&&&&&&&&
AB^2}$
\end{sideways}
\caption{Construction of $q^{-1}$}
\label{fig:qinv}
\end{figure}
The so constructed morphism $q^{-1}$ is the inverse of $q$ by the commutativity of the diagrams of Figure \ref{fig:q_qinv} (in the second case we also need to use that the columns are equal monomorphisms).
\begin{figure} 
\centering
\begin{sideways}
$\xymatrix@C=18pt@R=59pt{
A \ar[rrrr]^-{q^{-1}} \ar[d]^-\delta \ar@/_1.2pc/@{=}[dddd] &&&&
(A \coten B I)B \ar[dd]_-{j1} \ar@/^1.2pc/[dddd]^-q \\
A^2 \ar[d]^-{1s} \\
AB \ar[r]^-{1\delta} \ar[dd]^-{1\varepsilon} &
AB^2 \ar[r]^-{1z1} &
AB^2 \ar[r]^-{1i1} \ar[dd]^-{1m} &
A^2B \ar[r]^-{m1} \ar[d]^-{11i} &
AB 
\ar[d]_-{1i} \\
&&&
A^3 \ar[r]^-{m1} \ar[d]^-{1m} &
A^2 \ar[d]_-m \\
A \ar[rr]^-{1u} \ar@/_2pc/@{=}[rrrr] &&
AB \ar[r]^-{1i} &
A^2 \ar[r]^-m &
A}\quad
\xymatrix@C=15pt{
(A \coten B I)B \ar[rrr]_-{ji} \ar@/^1.5pc/[rrrrr]^-q \ar[rd]^-{1\delta} \ar[ddd]_-{j1} &&&
A^2 \ar[rr]_-m \ar[d]^-{\delta\delta} &&
A \ar[r]^-{q^{-1}} \ar[d]^-\delta &
(A \coten B I)B \ar[ddddddd]^-{j1} \\
& (A \coten B I)B^2 \ar[r]^-{1ii} \ar[rd]_-{1i1} &
(A \coten B I)A^2 \ar[d]^-{11s} &
A^4 \ar[r]^-{1c1} \ar[dd]^-{1s1s} &
A^4 \ar[r]^-{mm} \ar[dd]^-{11ss} &
A^2 \ar[dd]^-{1s} \\
&& (A \coten B I)AB \ar[d]^-{j11} \\
AB \ar[r]^-{1\delta} \ar[d]^-{1\delta} \ar@/_1.5pc/@{=}[dddd] &
AB^2 \ar[r]^-{1i1} \ar[d]^-{11\delta} &
A^2B \ar[r]^-{1u11} \ar[d]^-{11\delta} \ar@/_1.2pc/[rrr]_-{m1} &
(AB)^2 \ar[r]^-{1c1} &
A^2B^2 \ar[r]^-{mm} &
AB \ar[d]^-{1\delta} \\
AB^2 \ar[r]^-{1\delta 1} \ar[ddd]^-{1\varepsilon 1}  &
AB^3 \ar[r]^-{1i11} \ar[d]^-{11z1} &
A^2B^2 \ar[rrr]^-{m11} \ar[d]^-{11z1} &&&
AB^2 \ar[d]^-{1z1} \\
& AB^3 \ar[r]^-{1i11} \ar[dd]^-{1m1} &
A^2B^2 \ar[rrr]^-{m11} \ar[d]^-{11i1} &&&
AB^2 \ar[d]^-{1i1} \\
&& A^3B \ar[rrr]^-{m11} \ar[d]^-{1m1} &&&
A^2B \ar[rd]^-{m1} \\
AB \ar[r]^-{1u1} \ar@/_2pc/@{=}[rrrrrr] &
AB^2 \ar[r]^-{1i1} &
A^2B \ar[rrrr]^-{m1} &&&&
AB}$
\end{sideways}
\caption{Invertibility of $q$}
\label{fig:q_qinv}
\end{figure}
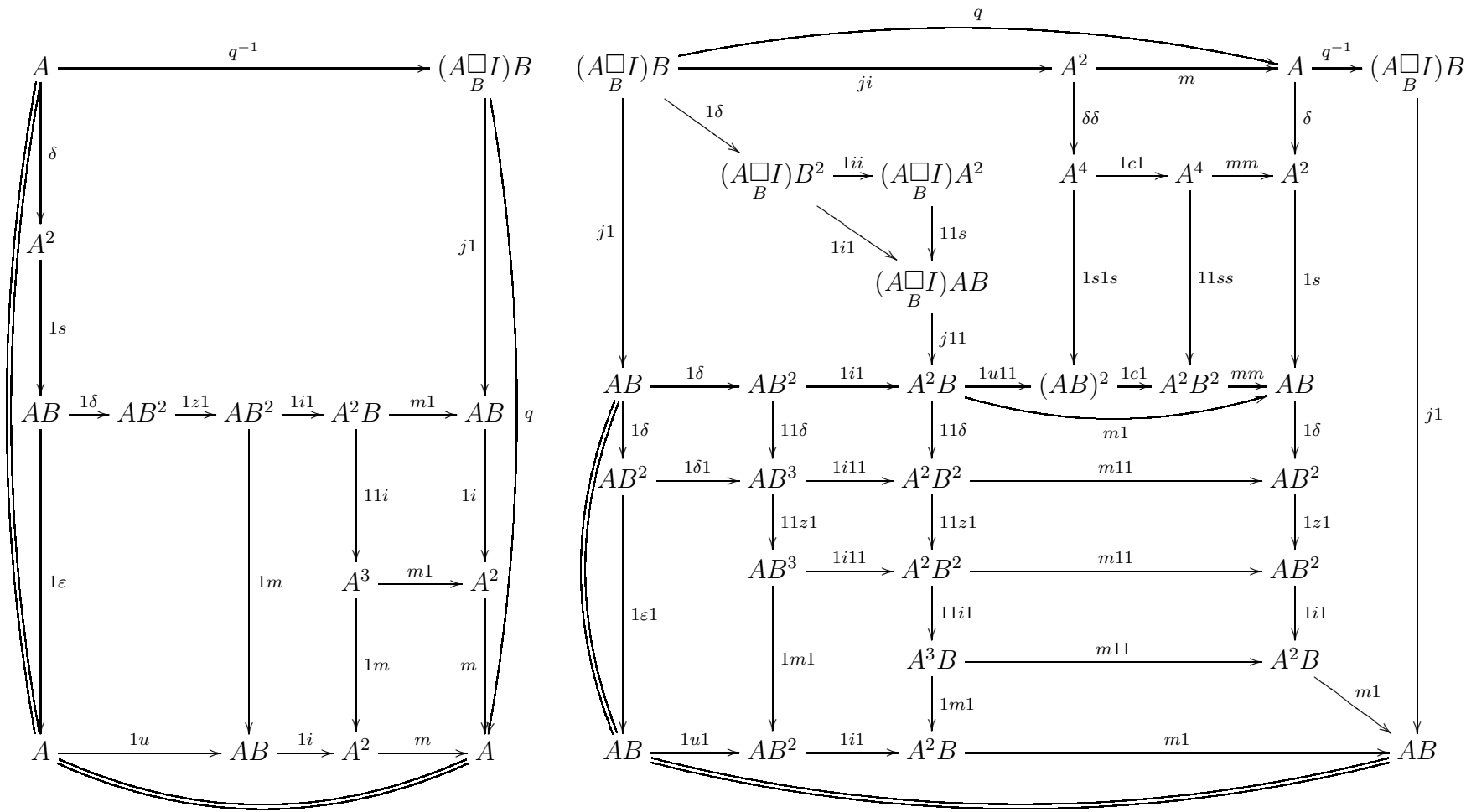

(2) If both $Y$ and $B$ are cocommutative comonoids then clearly so is $YB$; and if both $Y$ and $B$ have antipodes $z$ then
$\xymatrix@C=12pt{
YB \ar[r]^-{zz} & YB}$
is the antipode of the Hopf monoid $YB$.

Conversely, if $A$ is cocommutative then evidently so is its sub-comonoid $A\coten B I$. If furthermore $A$ has an antipode $z$ then it restricts to $A\coten B I$ by the commutativity of the following diagram.
$$
\xymatrix@R=15pt{
A\coten B I \ar[r]^-j \ar[dd]_-j &
A \ar[r]^-z \ar[d]^-\delta &
A \ar[d]^-\delta \\
& A^2 \ar[r]^-{zz} \ar[d]^-{1s} &
A^2 \ar[dd]^-{1s} \\
A \ar[r]^-{1u} \ar[d]_-z &
AB \ar[rd]^-{zz} \\
A \ar[rr]_-{1u} &&
AB}
$$
The top right region commutes by the Hopf monoid identity $\delta.z=zz.c.\delta$ and the assumed cocommutativity of $A$. The bottom right region commutes since any bimonoid morphism $s$ commutes with the antipodes. 
\end{proof}

\begin{example}
Proposition \ref{prop:SplitEpi_Hopf} can be applied in particular to a finitely complete category $\mathsf M$, regarded as a Cartesian monoidal category. From Proposition \ref{prop:SplitEpi_Hopf} we obtain equivalences between the following pairs of categories.
\begin{itemize}
  \item[{(1)}] 
   \begin{itemize}
      \item[{$\bullet$}] The category whose \\
               \underline{objects} are split epimorphisms 
              $\xymatrix@C=15pt{
              B \ar@<-2pt>@{ >->}[r]_-i &
              A \ar@<-2pt>@{->>}[l]_-s}$ 
              of monoids in $\mathsf M$ such that  $B$ is a group object. \\
             \underline{morphisms} are pairs of monoid morphisms which are compatible with the epimorphisms $s$ as well as their sections $i$.
      \item[{$\bullet$}]   The category whose \\ 
              \underline{objects} consist of a group object $B$ and a monoid $Y$ in $\mathsf M$, together with a left $B$-action on $Y$ 
              which makes $Y$ a left $B$-module monoid. \\ 
              \underline{morphisms} are pairs of monoid morphisms 
              $(\xymatrix@C=12pt{B \ar[r]^-b &B'},
              \xymatrix@C=12pt{Y \ar[r]^-y &Y'})$
              which are compatible with the actions 
              $\xymatrix@C=12pt{BY \ar[r]^-l &Y}$ and
              $\xymatrix@C=12pt{B'Y' \ar[r]^-{l'} &Y'}$ in the sense that $l'.by=y.l$.
   \end{itemize}
  \item[{(2)}] 
\begin{itemize}
      \item[{$\bullet$}]  The category whose \\
     \underline{objects} are split epimorphisms 
     $\xymatrix@C=15pt{
     B \ar@<-2pt>@{ >->}[r]_-i &
     A \ar@<-2pt>@{->>}[l]_-s}$ 
     of group objects. \\
     \underline{morphisms} are pairs of monoid morphisms which are compatible with the epimorphisms $s$ as well as their sections $i$.
     \item[{$\bullet$}]   The category whose \\ 
     \underline{objects} consist of group objects $B$ and $Y$ in $\mathsf M$, together with a left $B$-action on $Y$ which makes $Y$
     a left $B$-module group. \\
     \underline{morphisms} are pairs of bimonoid morphisms 
     $(\xymatrix@C=12pt{B \ar[r]^-b &B'},       
      \xymatrix@C=12pt{Y \ar[r]^-y &Y'})$
      which are compatible with the actions 
      $\xymatrix@C=12pt{BY \ar[r]^-l &Y}$ and
      $\xymatrix@C=12pt{B'Y' \ar[r]^-{l'} &Y'}$ in the sense that $l'.by=y.l$.
\end{itemize}
\end{itemize}
\end{example}

\begin{remark}
There are particular symmetric monoidal categories $\mathsf M$ whose cocommutative Hopf monoids  constitute semi-abelian categories $\mathsf{Hopf}(\mathsf M)$; e.g. the category of sets (which is Cartesian monoidal hence the Hopf monoids are the groups, all of them cocommutative) or the category of vector spaces over an algebraically closed field (see \cite{GKV}). In such cases the equivalence of Proposition \ref{prop:SplitEpi_Hopf}~(2) is in fact the equivalence $\mathsf{SplitEpi}(\mathsf{Hopf}(\mathsf M))\cong \mathsf{Act}(\mathsf{Hopf}(\mathsf M))$ discussed in \cite[Section 1]{Janelidze}, see \cite[Example 3.10]{Janelidze}.
\end{remark}


\section{Reflexive graphs of monoids versus pre-crossed modules} 
\label{sec:refl_graph}

Consider a monoidal admissible class $\mathcal S$ of spans in a monoidal
category $\mathsf C$ for which \cite[Assumption 4.1]{Bohm:Xmod_I} holds. 
Take an object 
$\xymatrix@C=15pt{
B \ar@<-2pt>@{ >->}[r]_-i &
A \ar@<-2pt>@{->>}[l]_-s}$ 
in the category $\mathsf{SplitEpiMon}_{\mathcal S}(\mathsf C)$ of Theorem
\ref{thm:SplitEpi_vs_DLaw}. Then by property (b) in Theorem
\ref{thm:SplitEpi_vs_DLaw}, the induced morphism 
$q:=\xymatrix@C=15pt{
(A\coten B I) B \ar[r]^-{p_A i} &
A^2 \ar[r]^-m &
A}$
is invertible. Therefore by \cite[Corollary 1.7]{Bohm:Xmod_I} there is a
bijective correspondence between the retractions $t$ of the monoid morphism
$i$ and the monoid morphisms  
$\xymatrix@C=12pt{
A \coten B I \ar[r]^-k &B}$
rendering commutative
$$
\xymatrix@C=15pt{
B(A\coten B I) \ar[rr]^-{ip_A} \ar[d]_-{1k} &&
A^2 \ar[rr]^-m &&
A \ar[rr]^-{q^{-1}} &&
(A\coten B I)B\ar[d]^-{k1} \\
B^2 \ar[rrr]_-m &&&
B &&&
B^2. \ar[lll]^-m}
$$
The correspondence is given by 
$$
t\mapsto k:=\xymatrix@C=12pt{A\coten B I \ar[r]^-{p_A} & A \ar[r]^-t & B} \qquad
k\mapsto t:=\xymatrix@C=15pt{A \ar[r]^-{q^{-1}} & (A\coten B I)B \ar[r]^-{k1} & B^2 \ar[r]^-m &B}.
$$
Combining this observation with the equivalence of Theorem \ref{thm:SplitEpi_vs_DLaw}, next we present an equivalent description of a suitable category of reflexive graphs of monoids. This leads to the notion of pre-crossed module over a monoid.

\begin{theorem} \label{thm:ReflGraph_vs_PreX}
Consider a monoidal admissible class $\mathcal S$ of spans in a monoidal category $\mathsf C$ for which \cite[Assumption 4.1]{Bohm:Xmod_I} holds. The following categories are equivalent.
\begin{itemize}
\item[{$\mathsf{Refl}$}]\hspace{-.3cm} $\mathsf{GraphMon}_{\mathcal S}(\mathsf
  C)$ whose \\
\underline{objects} are reflective graphs
$\xymatrix@C=20pt{
B \ar@{ >->}[r]|(.55){\, i\, } &
A \ar@{->>}@<-4pt>[l]_-s  \ar@{->>}@<4pt>[l]^-t}$ 
of monoids in $\mathsf C$ subject to the following conditions.
\begin{itemize}
\item[{(a)}] 
$\xymatrix@C=12pt{
A \ar@{=}[r] & A \ar[r]^-s & B}\in \mathcal S$ (hence the $\mathcal S$-relative pullback $A \coten B I$ in Theorem \ref{thm:SplitEpi_vs_DLaw} exists).
\item[{(b)}]
$q:=\xymatrix@C=18pt{
(A \coten B I)B \ar[r]^-{p_Ai} &
A^2 \ar[r]^-m &
A}$ is invertible.
\end{itemize}
\underline{morphisms} are pairs of monoid morphisms 
$(\xymatrix@C=12pt{B \ar[r]^-b & B'},\xymatrix@C=12pt{A \ar[r]^-a & A'})$
such that $s'.a=b.s$, $t'.a=b.t$ and $i'.b=a.i$.
\smallskip

\item[{$\mathsf{Pre}\hspace{.08cm}$}]\hspace{-.38cm} $\mathsf{X}_{\mathcal
  S}(\mathsf C)$ whose \\ 
\underline{objects} consist of monoids $B$ and $Y$, monoid morphisms
$\xymatrix@C=12pt{Y \ar[r]^-e & I}$ and
$\xymatrix@C=12pt{Y \ar[r]^-k& B}$
and a distributive law
$\xymatrix@C=12pt{BY \ar[r]^-x & YB}$ subject to the following conditions.
\begin{itemize}
\item[{(a')}] 
$\xymatrix@C=12pt{
Y \ar@{=}[r] & Y \ar[r]^-e & I}\in \mathcal S$
and
$\xymatrix@C=12pt{
B \ar@{=}[r] & B \ar@{=}[r] & B}\in \mathcal S$.
\item[{(b')}] $e1.x=1e$ and $m.k1.x=m.1k$.
\item[{(c')}] The morphism $f$ of Theorem \ref{thm:SplitEpi_vs_DLaw}~(c') is invertible.
\end{itemize}
\underline{morphisms} are pairs of monoid morphisms 
$(\xymatrix@C=12pt{B \ar[r]^-b & B'},\xymatrix@C=12pt{Y \ar[r]^-y & Y'})$
such that $e'.y=e$, $k'.y=b.k$ and $x'.by=yb.x$.
\end{itemize}
\end{theorem}

\begin{proof}
We show that the equivalence functors of Theorem \ref{thm:SplitEpi_vs_DLaw} lift to the equivalence of the claim. 
In the direction $\mathsf{ReflGraphMon}_{\mathcal S}(\mathsf C)\to \mathsf{PreX}_{\mathcal S}(\mathsf C)$ we send
$$
\xymatrix@C=30pt{
B \ar@{ >->}[r]|(.55){\, i\, }  \ar[d]_-b&
A \ar@<-4pt>@{->>}[l]_-s \ar@<4pt>@{->>}[l]^-t \ar[d]^-a \\
B' \ar@{ >->}[r]|(.55){\, i'\, } &
A' \ar@<-4pt>@{->>}[l]_-{s'}  \ar@<4pt>@{->>}[l]^-{t'}}
$$
to
$$
\xymatrix@C=12pt@R=12pt{
(A\coten B I, \ar[d]^-{a\morcoten 1}  & 
\hspace{-2cm}B, \ar[d]^-b \hspace{-2cm} &
A\coten B I \ar[r]^-{p_I} & I,  &
\hspace{-.5cm}  A\coten B I \ar[r]^-{p_A} & A \ar[r]^-t &B, &
\hspace{-.5cm}  B(A\coten B I) \ar[r]^-{ip_A} & A^2 \ar[r]^-m & A \ar[r]^-{q^{-1}} & 
(A\coten B I)B)   \\
(A'\coten {B'} I, &
\hspace{-2cm}  B', \hspace{-2cm}  &  
A' \coten {B'} I \ar[r]^-{p_I} & I, &
\hspace{-.5cm}  A'\coten {B'} I \ar[r]^-{p_{A'}} & A' \ar[r]^-{t'} &B', &
\hspace{-.5cm}  B'(A'\coten {B'} I) \ar[r]^-{i'p_{A'}} & A^{\prime 2} \ar[r]^-{m'} & A' 
\ar[r]^-{q^{\prime -1}} & 
(A'\coten {B'} I)B'). }
$$
By \cite[Proposition 3.7~(1)]{Bohm:Xmod_I}, $p_A$ is a monoid morphism hence so is $t.p_A$. The second condition in (b') holds by the considerations preceding the theorem. Hence in light of the proof of Theorem \ref{thm:SplitEpi_vs_DLaw} the object map is well-defined. Concerning the morphisms, the second condition holds by the commutativity of
$$
\xymatrix{
A\coten B I \ar[r]^-{p_A} \ar[d]_-{a\diagcoten 1} &
A \ar[r]^-t \ar[d]^-a &
B \ar[d]^-b \\
A' \coten {B'} I \ar[r]_-{p_{A'}}  &
A' \ar[r]_-{t'} &
B'.}
$$
Thus using again the proof of Theorem \ref{thm:SplitEpi_vs_DLaw} we conclude that this functor is well-defined.

In the opposite direction
$\mathsf{PreX}_{\mathcal S}(\mathsf C) \to
\mathsf{ReflGraphMon}_{\mathcal S}(\mathsf C)$
we put
$$
\xymatrix@C=12pt@R=20pt{
(Y, \ar[d]_-y  & 
\hspace{-.5cm}B, \ar[d]^-b \hspace{-.5cm} &
Y \ar[r]^-e & I, &
Y \ar[r]^-k &B, &
\hspace{-.5cm} 
BY \ar[r]^-x & YB)   \\
(Y', &
\hspace{-.5cm}  B', \hspace{-.5cm}  &  
Y' \ar[r]^-{e'} & I, &
Y' \ar[r]^-{k'} &B', &
\hspace{-.5cm} 
B'Y' \ar[r]^-{x'} & Y'B')}
\quad \raisebox{-18pt}{$\mapsto$} \quad
\xymatrix@C=30pt{
B \ar@{ >->}[r]|-(.55){\, u1\, } \ar[d]_-b&
YB \ar@<-4pt>@{->>}[l]_-{e1} \ar@<4pt>@{->>}[l]^-{m.k1} \ar[d]^-{yb} \\
B' \ar@{ >->}[r]|-{\, u'1\, } &
Y'B'. \ar@<-4pt>@{->>}[l]_-{e'1} \ar@<4pt>@{->>}[l]^-{m'.k'1}}
$$
By the considerations preceding the theorem $m.k1$ is a monoid morphism. It is a retraction of 
$\xymatrix@C=12pt{B \ar[r]^-{u1} & AB}$ by the unitality of $k$. The monoid morphisms $(b, yb)$ are compatible with $m.k1$ by the compatibility of $(b,y)$ with $k$ and the multiplicativity of $b$. So using again the proof of Theorem \ref{thm:SplitEpi_vs_DLaw} we conclude that this functor is well-defined too.

By the commutativity of 
$$
\xymatrix@R=15pt{
(A\coten B I)B \ar[r]^-{p_A 1} \ar[d]^-{p_A i} \ar@/_1.2pc/[dd]_-q &
AB \ar[r]^-{t1} &
B^2 \ar[r]^-m 
&
B \ar@{=}[dd] \\
A^2 \ar[rru]_-{tt} \ar[d]^-m \\
A \ar[rrr]_-t &&&
B}\qquad
\xymatrix@R=47pt{
Y \ar@{=}[r] \ar[d]_-f &
Y \ar[r]^-k \ar[d]^-{1u} &
B \ar[d]^-{1u} \ar@{=}@/^1.5pc/[rd] \\
YB \coten B I \ar[r]_-{p_{YB}} &
YB \ar[r]_-{k1} &
B^2 \ar[r]_-m &
B}
$$
the components $(1,q)$ and $(1,f)$ of the natural isomorphisms  in the proof
of Theorem \ref{thm:SplitEpi_vs_DLaw} are morphisms in the appropriate
category. This proves that the stated functors are mutually inverse equivalences.
\end{proof}

\begin{lemma} \label{lem:t-k_in_S}
Consider a monoidal admissible class $\mathcal S$ of spans in a monoidal
category $\mathsf C$ for which \cite[Assumption 4.1]{Bohm:Xmod_I} holds. 
For any object
$\xymatrix@C=20pt{
B \ar@{ >->}[r]|(.55){\, i\, } &
A \ar@{->>}@<-4pt>[l]_-s  \ar@{->>}@<4pt>[l]^-t}$ 
of the category $\mathsf{ReflGraphMon}_{\mathcal S}(\mathsf C)$ of Theorem \ref{thm:ReflGraph_vs_PreX}, the following assertions are equivalent. 
\begin{itemize}
\item[{(i)}] $\xymatrix@C=12pt{B & \ar[l]_-t A \ar@{=}[r] & A}\in \mathcal S$.
\item[{(ii)}] $\xymatrix@C=12pt{B && \ar[ll]_-{k:=t.p_A} A\coten B I  \ar@{=}[r] & A \coten B I}\in \mathcal S$.
\end{itemize}
\end{lemma}

\begin{proof}
Assertion (i) implies (ii) by \cite[Lemma 3.4]{Bohm:Xmod_I}.
Conversely, since $\xymatrix@C=10pt{B & \ar@{=}[l] B \ar@{=}[r] & B}\in \mathcal S$ by assumption, (ii) implies 
$\xymatrix@C=12pt{B^2 & \ar[l]_-{k1} (A\coten B I)B \ar@{=}[r] & (A\coten B I)B}\in \mathcal S$
by the multiplicativity of $\mathcal S$. Hence by (PRE) also
$\xymatrix@C=15pt{B^2 & \ar[l]_-{k1} (A\coten B I)B & \ar[l]_-{q^{-1}} A \ar[r]^-{q^{-1}} & (A\coten B I)B}\in \mathcal S$.
Then using the identity $t=m.k1.q^{-1}$ from the proof of Theorem \ref{thm:ReflGraph_vs_PreX}, (i) follows by (POST) (composing by $m$ on the left and by $q$ on the right).
\end{proof}

\begin{example} \label{ex:ReflGraph_groupoid}
As in Example \ref{ex:SplitEpi_groupoid}, take the (evidently admissible and
monoidal) class of all spans in the monoidal category $\mathsf C$ of spans
over a fixed set $X$. Then the equivalent categories of Theorem
\ref{thm:ReflGraph_vs_PreX} take the following forms.
\begin{itemize}
\item[{$\mathsf{Refl}$}]\hspace{-.3cm} $\mathsf{GraphMon}(\mathsf C)$ whose \\
\underline{objects} are reflective graphs
$\xymatrix@C=20pt{
B \ar@{ >->}[r]|(.55){\, \iota\, } &
A \ar@{->>}@<-4pt>[l]_-\sigma  \ar@{->>}@<4pt>[l]^-\tau}$ 
of categories with the common object set $X$ and identity-on-objects functors
between them, such that the map \eqref{eq:q_groupoid} in Example
\ref{ex:SplitEpi_groupoid} is invertible (recall that this holds e.g. if $B$
is a groupoid). \\
\underline{morphisms} are pairs of compatible identity-on-objects functors.
\item[{$\mathsf{Pre}\hspace{.08cm}$}]\hspace{-.38cm} $\mathsf{X}(\mathsf C)$
  whose \\  
\underline{objects} consist of categories $B$ and $Y$ of the common object set
$X$ such that in $Y$ there are no morphisms between non-equal objects; 
an action (cf. Example \ref{ex:SplitEpi_groupoid}) 
$\xymatrix@C=12pt{B \coten X Y \ar[r]^-\triangleright & Y}$ and
an identity-on-objects functor 
$\xymatrix@C=12pt{Y \ar[r]^-\kappa & B}$ such that
\begin{equation} \label{eq:preX_groupoid}
\kappa(b \triangleright y).b=b.\kappa(y)
\end{equation}
for all morphisms $b$ in $B$ and $y$ in $Y$ for which $s(b)=t(y)$.
(If $B$ is a groupoid then \eqref{eq:preX_groupoid} has the equivalent form 
$\kappa(b \triangleright y)=b.\kappa(y).b^{-1}$; so when both $B$ and $Y$ are
groupoids we recover the notion of {\em pre-crossed module} of groupoids in
\cite[Definition 1.2]{BrownIcen}.) \\
\underline{morphisms} are pairs of identity-on-objects functors
$(\xymatrix@C=12pt{B\ar[r]^-\beta & B'},\xymatrix@C=12pt{Y\ar[r]^-\nu & Y'})$
such that $\kappa'\nu=\beta\kappa$ and $\nu(b \triangleright y)=\beta(b)
\triangleright \nu(y)$ for all morphisms $b$ in $B$ and $y$ in $Y$ for which
$s(b)=t(y)$. 
\end{itemize}
\end{example}

\begin{example} \label{ex:CoMon_ReflGraph}
In the setting of Example \ref{ex:CoMon_SplitEpi}, the equivalent categories of Theorem \ref{thm:ReflGraph_vs_PreX} take the following explicit forms. 
\begin{itemize}
\item[{$\mathsf{Refl}$}]\hspace{-.3cm} $\mathsf{GraphMon}_{\mathcal S}(\mathsf C)$ whose \\
\underline{objects} are reflective graphs
$\xymatrix@C=20pt{
B \ar@{ >->}[r]|(.55){\, i\, } &
A \ar@{->>}@<-4pt>[l]_-s  \ar@{->>}@<4pt>[l]^-t}$ 
of bimonoids in $\mathsf M$ subject to the following conditions.
\begin{itemize}
\item[{(a)}] The comultiplication $\delta$ of $A$ satisfies
  $c.s1.\delta=1s.\delta$. 
\item[{(b)}] In terms of the morphism $j$ of \eqref{eq:cotensor}, 
$q:=\xymatrix@C=15pt{
(A \coten B I)B \ar[r]^-{ji} & A^2 \ar[r]^-m &A}$
is invertible.
\end{itemize}
\underline{morphisms} are pairs of bimonoid morphisms 
$(\xymatrix@C=12pt{B \ar[r]^-b & B'},\xymatrix@C=12pt{A \ar[r]^-a & A'})$
such that $s'.a=b.s$, $t'.a=b.t$ and $i'.b=a.i$.
\smallskip

\item[{$\mathsf{Pre}\hspace{.08cm}$}]\hspace{-.38cm} $\mathsf{X}_{\mathcal
  S}(\mathsf C)$ whose \\ 
\underline{objects} consist of a cocommutative bimonoid $B$ and a bimonoid
$Y$ in $\mathsf M$, together with a left $B$-action on $Y$ which makes $Y$
both a left $B$-module monoid and a left $B$-module comonoid, and a bimonoid morphism 
$\xymatrix@C=12pt{Y \ar[r]^-k& B}$ for which the following diagram commutes.
\begin{equation}\label{eq:PreX_CoMon}
\xymatrix@C=10pt{
BY \ar[rr]^-{\delta 1} \ar[d]_-{1k} &&
B^2 Y \ar[rr]^-{1c} &&
BYB \ar[rr]^-{l1} &&
YB \ar[d]^-{k1} \\
B^2 \ar[rrr]_-m &&&
B &&&
B^2 \ar[lll]^-m}
\end{equation}
\underline{morphisms} are pairs of bimonoid morphisms 
$(\xymatrix@C=12pt{B \ar[r]^-b &B'},
\xymatrix@C=12pt{Y \ar[r]^-y &Y'})$
which are compatible with the actions 
$\xymatrix@C=12pt{BY \ar[r]^-l &Y}$ and
$\xymatrix@C=12pt{B'Y' \ar[r]^-{l'} &Y'}$ in the sense that $l'.by=y.l$ and which satisfy $k'.y=b.k$.
\end{itemize}
\end{example}

\begin{remark}\label{eq:preX_Hopf}
Clearly, the equivalent categories of Example \ref{ex:CoMon_ReflGraph} have equivalent full subcategories for whose objects the bimonoid $B$ is a cocommutative Hopf monoid (then condition (b) becomes redundant by Example \ref{prop:SplitEpi_Hopf}). Note that whenever $B$ has an antipode $z$, the commutative diagram \eqref{eq:PreX_CoMon} has an equivalent form
\begin{equation} \label{eq:PreX_Hopf}
\xymatrix{
BY \ar[d]_-l \ar[r]^-{\delta k} &
B^3 \ar[r]^-{1c} &
B^3 \ar[r]^-{mz} &
B^2 \ar[d]^-m \\
Y \ar[rrr]_-k &&&
B}
\end{equation}
occurring in \cite[Definition 12~(iv)]{Villanueva}. Their equivalence follows by the commutativity of the diagrams of Figure \ref{fig:PreX}.
\begin{figure} 
\centering
\begin{sideways}
$\xymatrix@C=15pt@R=45pt{
BY \ar[r]^-{\delta 1} \ar[dd]^-{\delta 1} \ar@/_1.5pc/@{=}[ddddd]_-{\ } &
B^2Y \ar[r]^-{11k} \ar[rd]_-{1c} \ar[dd]^-{\delta 11} &
B^3 \ar[rrrr]^-{1c} &&&&
B^3  \ar[dd]_-{m1} \\
&& BYB \ar[rrrru]_-{1k1} \ar[d]^-{\delta 11} \ar@{}[rrrr]|-{\eqref{eq:PreX_CoMon}} &&&& \\
B^2Y \ar[r]^-{1\delta 1} \ar[ddd]^-{1\varepsilon 1} &
B^3Y \ar[r]^-{11c} \ar[d]^-{11z1} &
B^2YB \ar[r]^-{1c1} &
BYB^2 \ar[r]^-{l11} &
YB^2 \ar[r]^-{k11} &
B^3 \ar[r]^-{m1} &
B^2 \ar[d]_-{1z}  \\
& B^3Y \ar[r]^-{11c} \ar[d]^-{1m1} &
B^2YB \ar[r]^-{1c1} &
BYB^2 \ar[r]^-{l11} \ar[d]^-{11m} &
YB^2 \ar[r]^-{k11} &
B^3 \ar[r]^-{m1} \ar[d]^-{1m} &
B^2 \ar[dd]_-m \\
& B^2Y \ar[rr]^-{1c} &&
BYB \ar[r]^-{l1} &
YB \ar[r]^-{k1} &
B^2 \ar[rd]^-m \\
BY \ar[ru]^-{1u1} \ar[rrru]_-{11u} \ar[r]_-l &
Y \ar[r]_-k &
B \ar[rrru]_-{1u} \ar@{=}[rrrr] &&&&
B}\qquad
\xymatrix@C=15pt@R=45pt{
BY \ar[r]^-{\delta 1} \ar[d]^-{\delta 1} \ar@/_1.8pc/@{=}[ddddd] &
B^2Y \ar[rr]^-{1c} \ar[d]^-{\delta 11} &&
BYB \ar[rrr]^-{l1} \ar[d]^-{\delta k 1} \ar@{}[rrrddd]|-{\eqref{eq:PreX_Hopf} }&&&
YB \ar[ddd]_-{k1} \\
B^2Y \ar[r]^-{1\delta 1} \ar[rd]_-{1c} \ar[dddd]^-{1\varepsilon 1} &
B^3Y \ar[r]^-{11c} &
B^2YB \ar[r]^-{11k1} \ar[d]^-{1c1} &
B^4 \ar[dd]^-{1c1} \\
& BYB \ar[r]^-{11\delta} \ar[d]^-{1k1} &
BYB^2 \ar[d]^-{1k11} \\
& B^3 \ar[r]^-{11\delta} \ar[dd]^-{11\varepsilon} &
B^4 \ar@{=}[r] &
B^4 \ar[r]^-{11z1} &
B^4 \ar[r]^-{m11} \ar[d]_-{11m} &
B^3 \ar[r]^-{m1} \ar[dd]^-{1m} &
B^2 \ar[dd]_-m \\
&&&& B^3 \ar[rd]^-{m1} \\
BY \ar[r]_-{1k} &
B^2 \ar[rrru]^-{11u} \ar[rr]_-m &&
B \ar[rr]^-{1u} \ar@/_1.1pc/@{=}[rrr] &&
B^2 \ar[r]^-m &
B}$
\end{sideways}
\caption{Equivalence of \eqref{eq:PreX_CoMon} and \eqref{eq:PreX_Hopf} for Hopf monoids $B$}
\label{fig:PreX}
\end{figure}
\end{remark}


\section{Relative categories of monoids versus crossed modules}
\label{sec:rel_cat}

Consider again a monoidal admissible class $\mathcal S$ of spans in a monoidal
category $\mathsf C$ for which \cite[Assumption 4.1]{Bohm:Xmod_I} holds. 
Take an object
$\xymatrix@C=20pt{
B \ar@{ >->}[r]|(.55){\, i\, } &
A \ar@{->>}@<-4pt>[l]_-s  \ar@{->>}@<4pt>[l]^-t}$ 
of the category $\mathsf{ReflGraphMon}_{\mathcal S}(\mathcal C)$ of Theorem
\ref{thm:ReflGraph_vs_PreX} such that also 
$\xymatrix@C=12pt{B & \ar[l]_-t A \ar@{=}[r] & A}\in \mathcal S$; that is, the
legs of the cospan 
$\xymatrix@C=12pt{A \ar[r]^-s & B & \ar[l]_-t A}$ are in $\mathcal S$
(hence there exists its $\mathcal S$-relative pullback 
$\xymatrix@C=12pt{A & \ar[l]_-{p_1} A \coten B A \ar[r]^-{p_2} &
A}$). Whenever the morphism
\begin{equation} \label{eq:q2}
q_2:=\xymatrix@C=20pt{
(A \coten B I)A \ar[r]^-{p_A 1} &
A^2 \ar[rr]^-{(1\diagcoten \, i)(i \, \diagcoten 1)} &&
(A \coten B A)^2 \ar[r]^-m &
A\coten B A}
\end{equation}
is invertible, we infer form \cite[Corollary 1.7]{Bohm:Xmod_I} that
there exists at most one monoid morphism $d$ rendering commutative 
\begin{equation} \label{eq:d_exists}
\xymatrix{
A \coten B I \ar[r]^-{p_A} \ar@/_1.2pc/[rrd]_(.4){p_A} &
A \ar[r]^-{1\diagcoten \, i} & 
A\coten B A \ar@{-->}[d]^-d &&
A \ar[ll]_-{i \, \diagcoten 1} \ar@/^1.3pc/@{=}[lld] \\
&& A}
\end{equation}
which is our candidate to serve as the composition morphism of a relative
category. 
By this motivation, in this section we investigate first the condition that
\eqref{eq:q2} is invertible. Assuming so, next we show that whenever the
morphism $d$ of \eqref{eq:d_exists} exists, it makes the object 
$\xymatrix@C=20pt{
B \ar@{ >->}[r]|(.55){\, i\, } &
A \ar@{->>}@<-4pt>[l]_-s  \ar@{->>}@<4pt>[l]^-t}$ 
of $\mathsf{ReflGraphMon}_{\mathcal S}(\mathcal C)$ to an $\mathcal
S$-relative category. Finally, based on Theorem
\ref{thm:ReflGraph_vs_PreX}, we give an equivalent description of the category
of $\mathcal S$-relative categories in the category of monoids in $\mathsf C$,
in terms of crossed modules introduced hereby.

\subsection{Invertibility of some canonical morphisms}

\begin{lemma}\label{lem:hn}
Consider a monoidal admissible class $\mathcal S$ of spans in a monoidal
category $\mathsf C$ for which \cite[Assumption 4.1]{Bohm:Xmod_I} holds. 
For any monoid $B$ in $\mathsf C$ for which
$\xymatrix@C=10pt{B & \ar@{=}[l] B \ar@{=}[r] & B}$ is in $\mathcal S$, for
any span of monoids   
$\xymatrix@C=12pt{B & \ar[l]_-t A \ar[r]^-s & B}$
with legs in $\mathcal S$, and for any natural number $n$, the following
assertions hold. (Recall the convention $A^{\expcoten B 0}:=B$ from
\cite[Corollary 4.6]{Bohm:Xmod_I}.)
\begin{itemize}
\item[{(1)}] There exists the $\mathcal S$-relative pullback 
$$
\xymatrix@C=40pt{
(A \coten B I)B \coten B A^{\expcoten B n} \ar[r]^-{p_{A^{\diagcoten_B n}}} 
\ar[d]_-{p_{(A \diagcoten_B I)B}} &
A^{\expcoten B n}\ar[d]^-{t.p_1} \\
(A \coten B I)B \ar[r]_-{p_I 1} &
B.}
$$
\item[{(2)}] There is a unique morphism $h_n$ rendering commutative
$$
\xymatrix@R=15pt{
(A \coten B I)A^{\expcoten B n} \ar@/^2pc/[rrd]^-{p_I 1}
\ar@{-->}[rd]^-{h_n} \ar[dd]_-{1p_1} \\
& (A \coten B I)B \coten B A^{\expcoten B n} \ar[r]^-{p_{A^{\diagcoten_B n}}} 
\ar[d]_-{p_{(A \diagcoten_B I)B}} &
A^{\expcoten B n}\ar[d]^-{t.p_1} \\
(A \coten B I)A \ar[r]_-{1t} &
(A \coten B I)B \ar[r]_-{p_I 1} &
B.}
$$
\item[{(3)}] For a common section $i$ of $s$ and $t$, consider the morphism
\begin{equation} \label{eq:qn}
q_{n+1}:=\xymatrix@C=18pt{
(A \coten B I)A^{\expcoten B n} \ar[r]^-{p_A1} &
AA^{\expcoten B n} \ar[rrrr]^-{(1\diagcoten \, i \, \diagcoten \cdots \diagcoten
  i)(i \, \diagcoten 1 \diagcoten \cdots \diagcoten 1)} &&&&
(A^{\expcoten B n+1})^2 \ar[r]^-m &
A^{\expcoten B n+1}}
\end{equation}
(it is well-defined by \cite[Proposition 3.5]{Bohm:Xmod_I} and $q_1$ is
equal to $q$ in Theorem \ref{thm:SplitEpi_vs_DLaw}~(b)).
If $q_{n+1}$ is invertible for some $n$, then $q_k$ is invertible
for all $0<k\leq n$.
\item[{(4)}] For a common section $i$ of $s$ and $t$ the following are equivalent.
\begin{itemize}
\item[{(i)}] $h_n$ in part (2) and $q_1$ in part (3) are invertible.
\item[{(ii)}] $q_{n+1}$ in part (3) is invertible.
\end{itemize}
\end{itemize}
\end{lemma}

\begin{proof} 
(1) By assumption 
$\xymatrix@C=12pt{B & \ar[l]_-t  A \ar@{=}[r] & A}\in \mathcal S$ 
and by the unitality of $\mathcal S$, 
$\xymatrix@C=10pt{I & \ar@{=}[l] I \ar@{=}[r] & I}\in \mathcal S$.
Then by \cite[Lemma 3.4]{Bohm:Xmod_I},  
\begin{equation} \label{eq:t.p1_in_S;pI_in_S}
\xymatrix@C=12pt{B & \ar[l]_-t A & \ar[l]_-{p_1} A^{\expcoten B n}  \ar@{=}[r] & 
A^{\expcoten B n}}\in \mathcal S
\quad \textrm{and}\quad 
\xymatrix@C=12pt{A \coten B I & \ar@{=}[l] A \coten B I \ar[r]^-{p_I} & 
I}\in \mathcal S.
\end{equation} 
By assumption also 
$\xymatrix@C=10pt{B & \ar@{=}[l] B \ar@{=}[r] & B}\in \mathcal S$
hence by the second assertion in \eqref{eq:t.p1_in_S;pI_in_S} and the
multiplicativity of $\mathcal S$ 
\begin{equation} \label{eq:pI1_in_S}
\xymatrix@C=12pt{(A \coten B I)B & \ar@{=}[l] (A \coten B I)B 
\ar[r]^-{p_I1} & B}\in \mathcal S.
\end{equation} 
The first assertion of \eqref{eq:t.p1_in_S;pI_in_S} and \eqref{eq:pI1_in_S}
say that the legs of  
$\xymatrix@C=12pt{(A \coten B I)B \ar[r]^-{p_I1} & B & 
\ar[l]_-{t.p_1} A^{\expcoten B n}}$ are in $\mathcal S$ hence their $\mathcal
S$-relative pullback exists by assumption.

(2) By \eqref{eq:t.p1_in_S;pI_in_S} and the multiplicativity of $\mathcal S$, 
$$
\xymatrix@C=15pt{
(A\coten B I)B & 
\ar[l]_-{1t} (A\coten B I)A &
\ar[l]_-{1p_1} (A \coten B I)A^{\expcoten B n} \ar[r]^-{p_I 1} &
A^{\expcoten B n}}\in \mathcal S.
$$
Hence by the evident commutativity of the exterior of the diagram in part (2),
universality of the $\mathcal S$-relative pullback in its codomain implies 
the existence of the unique morphism $h_n$. 

(3) For some positive integer $n$ assume that $q_{n+1}$ is invertible. Then so
is $q_n$ with the inverse
\begin{equation}\label{eq:qn^-1}
\xymatrix@C=40pt{
A^{\expcoten B n} \ar[r]^-{1\diagcoten \, i} &
A^{\expcoten B n+1} \ar[r]^-{q_{n+1}^{-1}} &
(A\coten B I)A^{\expcoten B n} \ar[r]^-{1p_{1\dots n-1}} &
(A\coten B I)A^{\expcoten B n-1}.}
\end{equation}
Indeed, \eqref{eq:qn^-1} renders commutative both diagrams
$$
\xymatrix@C=7pt@R=10pt{
(A\coten B I)A^{\expcoten B n-1}\ar@{=}[rr] \ar[rd]_-{1(1\diagcoten \, i)}
\ar[dd]_-{q_n} &&
(A\coten B I)A^{\expcoten B n-1} \\
& (A\coten B I)A^{\expcoten B n} \ar@{=}[rd] \ar[d]_-{q_{n+1}} \\
A^{\expcoten B n} \ar[r]_-{1\diagcoten \, i} &
A^{\expcoten B n+1} \ar[r]_-{q_{n+1}^{-1}} &
(A\coten B I)A^{\expcoten B n} \ar[uu]^(.6){1p_{1\dots n-1}} }
\xymatrix@C=20pt@R=16pt{
A^{\expcoten B n+1} \ar@{=}[rd]
\ar[r]^-{\raisebox{8pt}{${}_{q_{n+1}^{-1}}$}}  &
(A\coten B I)A^{\expcoten B n} \ar[d]^-{q_{n+1}}
\ar[r]^-{\raisebox{8pt}{${}_{1p_{1\dots n-1}}$}}  &
(A\coten B I)A^{\expcoten B n-1} \ar[dd]_-{q_n} \\
& A^{\expcoten B n+1} \ar[rd]^-{p_{1\dots n}} \\
A^{\expcoten B n} \ar[uu]_(.4){1\diagcoten \, i} \ar@{=}[rr] &&
A^{\expcoten B n}} 
$$
The leftmost region of the first diagram commutes by the explicit expression
\eqref{eq:qn} of $q_n$ and $q_{n+1}$, multiplicativity of $1\morcoten i$ and the
functoriality of $\Box$, see \cite[Proposition 3.5~(2)]{Bohm:Xmod_I}.
The rightmost region of the second diagram commutes again by the explicit
expression \eqref{eq:qn} of $q_n$ and $q_{n+1}$ and the multiplicativity of
$p_{1\dots n}$.

(4) Our strategy is to prove that $q_{n+1}$ can be rewritten as
\begin{equation}\label{eq:qBox1.hn}
\xymatrix{
(A\coten B I)A^{\expcoten B n} \ar[r]^-{h_n} &
(A\coten B I)B \coten B A^{\expcoten B n} \ar[r]^-{q\, \diagcoten 1} &
A^{\expcoten B n+1}.}
\end{equation}
Then (i) obviously implies (ii) and in view of part (3) also the opposite
implication holds. 

The occurring morphism $q\morcoten 1$ 
is defined as the unique morphism rendering commutative 
$$
\xymatrix@C=7pt@R=10pt{
(A\coten B I)B \coten B A^{\expcoten B n} \ar[ddd]_-{p_{(A\diagcoten_B I)B}}
\ar@/^1.2pc/[rrrd]^-{p_{A^{\Box_B n}}} \ar@{-->}[rd]^-{q\, \diagcoten 1} \\
& A^{\expcoten B n+1} \ar[rr]_-{p_{2\dots n}} \ar[dd]_-{p_1} &&
A^{\expcoten B n} \ar[dd]^-{t.p_1} \\
\\
(A\coten B I)B \ar[r]_-q &
A \ar[rr]_-s &&
B}
$$
It is well-defined by the commutativity of the first diagram of
\eqref{eq:s.q}; see \cite[Proposition 3.5~(2)]{Bohm:Xmod_I}. The morphism of
\eqref{eq:qBox1.hn} is equal to $q_{n+1}$ by the commutativity of both diagrams 
\begin{equation}\label{eq:hn_1}
\xymatrix@C=20pt@R=10pt{
&& (A\coten B I)B \coten B A^{\expcoten B n} \ar[rrr]^-{q\, \diagcoten 1}
\ar[d]^-{p_{(A\diagcoten_B I)B}} &&&
A^{\expcoten B n+1} \ar[d]^-{p_1} \\
(A\coten B I) A^{\expcoten B n} \ar[r]^-{1p_1} \ar@/_1.7pc/[rrrdd]_-{p_A1} 
\ar@/^1.5pc/[rru]^-{h_n} &
(A\coten B I) A \ar[r]^-{1t} &
(A\coten B I)B \ar[r]_-{p_A 1} \ar@/^1.1pc/[rrr]^-q &
AB \ar[r]_-{1i} &
A^2 \ar[r]_-m &
A \\
&&& A^2 \ar[u]^-{1t} \\
&&& AA^{\expcoten B n} \ar[u]^-{1p_1} 
\ar[r]_-{\raisebox{-8pt}{${}_{(1\diagcoten \, i \, \diagcoten \cdots \diagcoten \, i)(i \, \diagcoten 1)}$}} &
(A^{\expcoten B n+1})^2 \ar[uu]_-{p_1p_1} \ar[r]_-m &
A^{\expcoten B n+1} \ar[uu]_-{p_1}}
\end{equation}
\begin{equation}\label{eq:hn_2}
\xymatrix@C=39pt@R=10pt{
& (A\coten B I)B \coten B A^{\expcoten B n} \ar[rrr]^-{q\, \diagcoten 1}
\ar[d]^-{p_{A^{\Box_B n}}} &&&
A^{\expcoten B n+1} \ar[d]^-{p_{2\dots n}} \\
(A\coten B I) A^{\expcoten B n} \ar[r]^-{p_I 1} \ar@/_1.2pc/[rrd]_-{p_A1} 
\ar@/^1.5pc/[ru]^-{h_n} &
A^{\expcoten B n} \ar[r]_-{u1} \ar@/^1.2pc/@{=}[rrr] &
BA^{\expcoten B n} \ar[r]_-{(i \, \diagcoten \cdots \diagcoten \, i)1}  &
(A^{\expcoten B n})^2 \ar[r]_-m &
A^{\expcoten B n} \\
&& AA^{\expcoten B n} \ar[u]^-{s1} 
\ar[r]_-{\raisebox{-8pt}{${}_{(1\diagcoten \, i \, \diagcoten \cdots \diagcoten \, i)(i \, \diagcoten 1)}$}} &
(A^{\expcoten B n+1})^2 \ar[u]_-{p_{2\dots n}p_{2\dots n}} \ar[r]_-m &
A^{\expcoten B n+1} \ar[u]_-{p_{2\dots n}}}
\end{equation}
whose right verticals are joint monomorphisms.
\end{proof}

\begin{example} \label{ex:hn_groupoid}
In the category $\mathsf C$ of spans over a given set $X$ from Example
\ref{ex:SplitEpi_groupoid}, the morphisms $h_n$ of Lemma \ref{lem:hn}~(2) are
isomorphisms, see the pullback \eqref{eq:hn_groupoid}. Hence for any reflexive
graph 
$\xymatrix@C=20pt{
B \ar@{ >->}[r]|(.55){\, \iota\, } &
A \ar@{->>}@<-4pt>[l]_-\sigma  \ar@{->>}@<4pt>[l]^-\tau}$ 
of categories with common object set $X$ and identity-on-objects functors
between them, all morphisms $\{q_n\}_{n>0}$ in Lemma \ref{lem:hn}~(3)
are invertible if and only if $q_1$ is so; see Lemma \ref{lem:hn}~(3). The
latter condition holds e.g. if $B$ is a groupoid, see Example
\ref{ex:SplitEpi_groupoid}.   
\end{example}

\begin{example} \label{ex:hn_CoMon}
In the context of Example \ref{ex:CoMon_SplitEpi}
we know from \cite[Example 4.3]{Bohm:Xmod_I} that 
\cite[Assumption 4.1]{Bohm:Xmod_I} holds for the monoidal admissible class
$\mathcal S$ in \cite[Example 2.3]{Bohm:Xmod_I} and 
\cite[Example 2.7]{Bohm:Xmod_I} of spans in $\mathsf C$. 

In this situation, for any cocommutative comonoid $B$ in $\mathsf M$ and any
comonoid morphism  
$\xymatrix@C=12pt{C \ar[r]^-f & B}$ such that the comultiplication $\delta$ of
$C$ satisfies
$f1.\delta=f1.c.\delta$, there is a unique {\em isomorphism} $h$ rendering
commutative 
$$
\xymatrix@C=10pt@R=10pt{
AC \ar@/^1.2pc/[rrrd]^-{\varepsilon 1} \ar@/_1.2pc/[rddd]_-{1f} \ar@{-->}[rd]^-h \\
& AB \coten B C \ar[rr]^-{p_C} \ar[dd]_-{p_{AB}} &&
C \ar[dd]^-f \\
\\
& AB \ar[rr]_-{\varepsilon 1} &&
B}
$$
with the inverse 
$\xymatrix@C=12pt{AB \coten B C \ar[r]^-{j} &
ABC \ar[r]^-{1\varepsilon 1} &
AC}$
(where $j=p_{AB}p_C.\delta$ is the equalizer of $1\delta 1$ and
$11f1.11\delta$ as in \eqref{eq:cotensor}; and $\varepsilon$ stands for both
counits of $A$ and $B$). Indeed, the following diagrams commute.
$$
\xymatrix{
AB\coten B C \ar[rr]_-\delta \ar[d]_-{p_{AB}} \ar@/^1.1pc/[rrr]^-j &&
(AB\coten B C)^2 \ar[r]_-{p_{AB}p_C} \ar[d]_-{p_{AB} p_{AB}} &
ABC \ar[r]^-{1\varepsilon 1} \ar[d]^-{11f} &
AC \ar[r]^-h \ar[d]^-{1f} &
AB\coten B C  \ar[d]^-{p_{AB}} \\
AB \ar[r]^-{\delta \delta} \ar@/_1.2pc/@{=}[rrrr] &
A^2B^2 \ar[r]^-{1c1} &
(AB)^2  \ar[r]^-{11\varepsilon 1} &
AB^2  \ar[r]^-{1\varepsilon 1} &
AB \ar@{=}[r] &
AB}
$$
$$
\xymatrix@C=38pt{
AB\coten B C \ar[r]_-\delta \ar@/_1.2pc/@{=}[rd] \ar@/^1.1pc/[rr]^-j &
(AB\coten B C)^2 \ar[r]_-{p_{AB}p_C} \ar[d]_-{\varepsilon 1} &
ABC \ar[r]^-{1\varepsilon 1} &
AC \ar[r]^-h \ar[d]^-{\varepsilon 1} &
AB\coten B C  \ar[d]^-{p_C} \\
& AB\coten B C  \ar[rr]_-{p_C}  &&
C\ar@{=}[r] &
C}
$$
\smallskip

$$
\xymatrix@C=32pt{
AC \ar[r]_-{\delta \delta} \ar[d]_-h \ar@/^1.4pc/@{=}[rrrr] &
A^2C^2 \ar[r]_-{1c1} &
(AC)^2 \ar@{=}[r] \ar[d]^-{hh} &
(AC)^2 \ar[d]^-{1f\varepsilon 1} \ar[r]_-{1\varepsilon \varepsilon 1} &
AC \ar@{=}[d] \\
AB\coten B C \ar[rr]^-\delta  \ar@/_1.1pc/[rrr]_-j &&
(AB\coten B C)^2 \ar[r]^-{p_{AB}p_C} &
ABC \ar[r]_-{1\varepsilon 1} &
AC}
$$

By \cite[Example 2.8]{Bohm:Xmod_I} there is an induced monoidal admissible
class (also denoted by $\mathcal S$) in the category of monoids in $\mathsf C$
(that is, the category of bimonoids in $\mathsf M$) also satisfying
\cite[Assumption 4.1]{Bohm:Xmod_I} by \cite[Example 4.4]{Bohm:Xmod_I}.  So
whenever the above morphism $f$ is a monoid morphism as well, there is a bimonoid isomorphism $h$ in the diagram, see \cite[Proposition 3.7]{Bohm:Xmod_I}. Consequently, in the category of bimonoids in $\mathsf M$, the morphisms $h_n$ of Lemma
\ref{lem:hn}~(2) are isomorphisms. Therefore $q_n$ in Lemma \ref{lem:hn}~(3) is an isomorphism for all positive integer $n$ if and only if it is invertible for $n=1$; and this holds whenever $B$ is a Hopf monoid, see Proposition \ref{prop:SplitEpi_Hopf}.
\end{example}

\begin{lemma} \label{lem:bn}
Let $\mathcal S$ be a monoidal admissible class  of spans in a monoidal
category $\mathsf C$ for which \cite[Assumption 4.1]{Bohm:Xmod_I} holds and let
$(B,Y,
\xymatrix@C=12pt{Y \ar[r]^-e & I},
\xymatrix@C=12pt{Y \ar[r]^-k & B},
\xymatrix@C=12pt{BY \ar[r]^-x & YB})$
be an object of the category $\mathsf{PreX}_{\mathcal S}(\mathsf C)$ in Theorem
\ref{thm:ReflGraph_vs_PreX} such that  
$\xymatrix@C=12pt{B & \ar[l]_-k Y  \ar@{=}[r] & Y}\in \mathcal S$.
For any natural number $n$ denote by 
$\xymatrix@C=15pt{B^{n+1} \ar[r]^-{m^{(n)}} & B}$ the $n$-times iterated
multiplication (unique by the associativity of $m$; by definition the identity
morphism for $n=0$) and consider the span
\begin{equation} \label{eq:YnB}
\xymatrix@C=35pt{
B & 
\ar[l]_-{m^{(n)}} B^{n+1} &
\ar[l]_-{k\dots k1} Y^n B \ar[r]^-{e\dots e 1} &
B}.
\end{equation} 
For any natural number $n$ the following assertions hold.
\begin{itemize}
\item[{(1)}] The cospan 
$\xymatrix{
YB \ar[r]^-{e1} & B &
\ar[l]_-{m^{(n)}} B^{n+1} &
\ar[l]_-{k\dots k1} Y^n B}$ has its legs in $\mathcal S$ (hence there exists
  its $\mathcal S$-relative pullback $YB \coten B Y^nB$).
\item[{(2)}] There exists a unique morphism $b_{n+1}$ of spans (for the spans
\eqref{eq:YnB}) rendering commutative
$$
\xymatrix@C=10pt@R=10pt{
Y^{n+1}B \ar@/^1.1pc/[rrrd]^-{e1\dots 11} \ar[ddd]_-{1k\dots k1}
\ar@{-->}[rd]^(.55){b_{n+1}} \\
& YB \coten B Y^n B \ar[rr]^-{p_{Y^nB}} \ar[dd]_-{p_{YB}} &&
Y^n B \ar[d]^-{k\dots k1} \\
&&& B^{n+1} \ar[d]^-{m^{(n)}} \\
YB^{n+1} \ar[r]_-{1m^{(n)}} & 
YB \ar[rr]_-{e1} &&
B.}
$$
\item[{(3)}] If $b_{n+1}$ in part (2) is an isomorphism then also $b_k$ is an
  isomorphism for all $0<k\leq n$.  
\item[{(4)}] For the morphism
$$
q_{n+1}\!:=\!\!\!\xymatrix@C=14pt{
(YB \coten B I)(YB)^{\expcoten B n}\, \ar[r]^-{p_{YB}1} &
YB(YB)^{\expcoten B n} 
\ar[rrrr]^-{(1\diagcoten u1 \diagcoten \cdots \diagcoten u1)(u1\diagcoten 1)} 
&&&&
((YB)^{\expcoten B n+1})^2 \ar[r]^-m &
(YB)^{\expcoten B n+1}}
$$
the following diagram commutes
$$
\xymatrix{
Y^{n+1} B \ar[d]_-{f\dots f1} \ar[r]^-{b_{n+1}} &
YB \coten B Y^{n}B \ar[r]^-{1\diagcoten\, b_{n}} &
\cdots 
\ar[r]^-{1\diagcoten\, b_1} &
(YB)^{\expcoten B n+1} \coten B B \ar[d]^-{p_{(YB)^{\expcoten B n+1}}} \\
(YB \coten B I)^{n+1} B \ar[r]_-{1\dots 1q_1} &
(YB \coten B I)^{n} YB \ar[r]_-{1\dots 1q_2} &
\cdots 
\ar[r]_-{q_{n+1}} &
(YB)^{\expcoten B n+1}}
$$
where $f$ is the isomorphism in Theorem \ref{thm:SplitEpi_vs_DLaw}~(c').
\item[{(5)}] $b_{n+1}$ in part (2) is an isomorphism if and only if $q_{n+1}$
  in part (4) is an isomorphism. 
\end{itemize}
\end{lemma}

\begin{proof}
(1) By definition the first two spans in
\begin{equation} \label{eq:lembn_1}
\xymatrix@C=12pt{Y & \ar@{=}[l] Y \ar[r]^-{e} & I} \qquad
\xymatrix@C=12pt{B & \ar@{=}[l] B \ar@{=}[r] & B} \qquad
\xymatrix@C=12pt{B & \ar[l]_-k Y \ar@{=}[r] & Y} \qquad
\xymatrix@C=12pt{YB & \ar@{=}[l] YB \ar[r]^-{e1} & B}
\end{equation}
belong to $\mathcal S$ hence so does the last one by the multiplicativity of
$\mathcal S$. 
Again, by definition the second and the third spans of \eqref{eq:lembn_1}
belong to $\mathcal S$ hence by the multiplicativity of $\mathcal S$ so does
the first one in \begin{equation} \label{eq:lembn_2}
\xymatrix@C=10pt{B^{n+1} && \ar[ll]_-{k\dots k1} Y^nB \ar@{=}[r] & Y^nB}
\qquad
\xymatrix@C=10pt{
B && \ar[ll]_-{m^{(n)}} B^{n+1} 
&&
\ar[ll]_-{k\dots k1} Y^nB \ar@{=}[r] & Y^nB.}
\end{equation}
Then the second span of \eqref{eq:lembn_2} is in $\mathcal S$ by (POST).

(2) Since the first span of \eqref{eq:lembn_1} and the second span of
\eqref{eq:lembn_2} are in $\mathcal S$, the multiplicativity of $\mathcal S$
implies that so is
$$
\xymatrix{
YB && \ar[ll]_-{1m^{(n)}} YB^{n+1} 
&&
\ar[ll]_-{1k\dots k1} Y^{n+1}B \ar[r]^-{e1\dots 11} & Y^nB.}
$$
So by the evident commutativity of the exterior of the diagram of part (2) the
stated morphism $b_{n+1}$ exists. It is a morphism of spans (for the spans
\eqref{eq:YnB}) by the commutativity of the following diagrams.
$$
\xymatrix{
Y^{n+1}B \ar[rr]^-{e\dots e1} \ar[d]_-{b_{n+1}} \ar[rd]^-{e1\dots 1} &&
B \ar@{=}[d] \\
YB\coten B Y^nB \ar[r]_-{p_{Y^nB}} &
Y^nB \ar[r]_-{e\dots e1} &
B}\quad 
\xymatrix{
Y^{n+1}B \ar[r]^-{1k\dots k1} \ar[d]_-{b_{n+1}} &
YB^{n+1} \ar[r]^-{k1\dots 1} \ar[d]^-{1m^{(n)}} &
B^{n+2} \ar[r]^-{m^{(n+1)}} \ar[d]^-{1m^{(n)}} &
B \ar@{=}[d] \\
YB\coten B Y^nB \ar[r]_-{p_{YB}} &
YB \ar[r]_-{k1} &
B^2 \ar[r]_-m & 
B}
$$

(3) Since for a positive integer $n$, $\xymatrix{Y^{n-1}B \ar[r]^-{1\dots 1u1}
& Y^nB}$ is a morphism between the spans of \eqref{eq:YnB}, the morphism in
the top row of the following diagram is well-defined by \cite[Proposition
3.5]{Bohm:Xmod_I}.  
$$
\xymatrix@R=15pt@C=10pt{
&& YB\coten B Y^{n-1}B \ar[rd]^-{p_{YB}} 
\ar[r]^-{\raisebox{8pt}{${}_{1\diagcoten 1\dots 1u1}$}} &
YB\coten B Y^nB \ar[d]^-{p_{YB}} \\
Y^nB \ar@{=}[r] \ar@/^1.2pc/[rru]^-{b_n} \ar@/_1.2pc/[rdd]_-{1\dots 1u1} &
Y^n B \ar[r]^-{1k\dots k1} &
YB^n \ar[r]^(.45){1m^{(n-1)}} &
YB \\
&& YB^{n+1} \ar[u]^-{1\dots 1m1} \ar[ru]_-{1m^{(n)}}\\
& Y^{n+1} B \ar[rr]_-{b_{n+1}} \ar[ru]_-{1k\dots k1} \ar[uu]_-{1\dots 1m1} &&
YB\coten B Y^nB \ar[uu]_-{p_{YB}}}
\xymatrix@R=30pt@C=12pt{
& YB\coten B Y^{n-1}B \ar[d]^-{p_{Y^{n-1}B}} 
\ar[r]^-{\raisebox{8pt}{${}_{1\diagcoten 1\dots 1u1}$}} &
YB\coten B Y^nB \ar[d]^-{p_{Y^nB}} \\
Y^nB \ar[r]^-{e1\dots 11} 
\ar@/^1.2pc/[ru]^-{b_n} \ar@/_1.2pc/[rd]_-{1\dots 1u1} &
Y^{n-1} B \ar[r]^-{1\dots 1u1} &
Y^nB \\
& Y^{n+1} B \ar[r]_-{b_{n+1}} \ar[ru]^-{e1\dots 1} &
YB\coten B Y^nB \ar[u]_-{p_{Y^nB}}}
$$
By their commutativity we infer $b_{n+1}.1\dots 1u1=(1\morcoten 1\dots 1u1).b_n$ 
Similarly, since for $n>0$ also $\xymatrix{Y^nB \ar[r]^-{1\dots 1m1} &
Y^{n-1}B}$ is a morphism between the spans of \eqref{eq:YnB}, the morphism in
the top row of the following diagram is well-defined by 
\cite[Proposition 3.5]{Bohm:Xmod_I}. 
$$
\xymatrix@R=15pt@C=10pt{
& YB\coten B Y^nB \ar[rr]^-{1\diagcoten 1\dots 1m1} \ar[rrd]^-{p_{YB}} &&
YB\coten B Y^{n-1}B \ar[d]^-{p_{YB}} \\
Y^{n+1}B \ar[r]^-{1k\dots k1}
\ar@/^1.2pc/[ru]^-{b_{n+1}} \ar@/_1.7pc/[rrdd]_-{1\dots 1m1} &
YB^{n+1} \ar[rr]^-{1m^{(n)}} \ar@/_.9pc/[rd]_-{11\dots 1m1} &&
YB \\
&& YB^n \ar@/_1pc/[ru]_(.45){1m^{(n-1)}} \\
&& Y^nB \ar[u]_-{1k\dots k1} \ar[r]_-{b_n} &
YB\coten B Y^{n-1}B \ar[uu]_-{p_{YB}}}
\xymatrix@R=30pt@C=10pt{
& YB\coten B Y^nB  \ar[d]^-{p_{Y^nB}} 
\ar[r]^-{\raisebox{8pt}{${}_{1\diagcoten 1\dots 1m1}$}} &
YB\coten B Y^{n-1}B \ar[d]^-{p_{Y^{n-1}B}} \\
Y^{n+1}B \ar[r]^-{e1\dots 1}
\ar@/^1.2pc/[ru]^-{b_{n+1}} \ar@/_1.5pc/[rd]_-{1\dots 1m1} &
Y^nB \ar[r]^-{1\dots 1m1} &
Y^{n-1}B \\
& YB^n \ar[ru]^(.45){e1\dots 1}  \ar[r]_-{b_n} &
YB\coten B Y^{n-1}B \ar[u]_-{p_{Y^{n-1}B}}}
$$
By their commutativity, $b_n.1\dots 1m1=(1\morcoten 1\dots 1m1).b_{n+1}$. 
It follows from these identities and the unitality of the monoid $Y$ that
whenever $b_{n+1}$ is invertible then so is $b_n$ with the inverse
$$
\xymatrix@C=40pt{
YB\coten B Y^{n-1}B \ar[r]^-{1\diagcoten 1\dots 1u1} &
YB\coten B Y^nB \ar[r]^-{b_{n+1}^{-1}} &
Y^{n+1}B \ar[r]^-{1\dots 1m1} &
Y^nB.}
$$

(4) We proceed by induction in $n$. For $n=0$ the diagram in the claim reduces
to 
$$
\xymatrix{
YB \ar[r]^-{b_1} \ar[d]_-{f1} \ar@{=}[rd] &
YB \coten B B \ar[d]^-{p_{YB}} \\
(YB \coten B I)B \ar[r]_-{q_1} &
YB}
$$
whose upper half commutes by construction (see part (2)) and the lower half
commutes since $f1$ and $q_1$ are mutual inverses (see the proof of Theorem
\ref{thm:SplitEpi_vs_DLaw}). 

For any positive value of $n$, denote the top-right path in the diagram of the
claim by $\widetilde b_{n+1}$ and the bottom row by $\widetilde q_{n+1}$. Then
the diagram takes the form
$$
\xymatrix@R=15pt{
Y^{n+1}B \ar[rr]^-{\widetilde b_{n+1}} \ar[d]_-{f1\dots 11} &&
(YB)^{\expcoten B n+1} \ar@{=}[dd] \\
(YB \coten B I)Y^nB \ar[d]_-{1f\dots f1} \ar[rd]^-{1\widetilde b_n} \\
(YB \coten B I)^{n+1} B \ar[r]^-{1\widetilde q_n}
\ar@/_1.1pc/[rr]_-{\widetilde q_{n+1}}&
(YB \coten B I)(YB)^{\expcoten B n} \ar[r]^-{q_{n+1}} &
(YB)^{\expcoten B n+1}.}
$$
The region at the bottom left corner commutes if the claim holds for $n-1$;
and the commutativity of the large region is proven in Figure
\ref{fig:b-q_induction}.  
\begin{figure} 
\centering
\begin{sideways}
$\xymatrix@C=15pt@R=15pt{
&&& (YB \coten B I)(YB)^{\expcoten B n} \ar[r]_-{p_{YB}1} \ar[d]^-{1p_1}
\ar@/^1.2pc/[rrrr]^-{q_{n+1}} & 
YB(YB)^{\expcoten B n} \ar[d]_-{11p_1}
\ar[r]_-{\raisebox{-8pt}{${}_{
(1\diagcoten u1\diagcoten \dots \diagcoten u1)(u1\diagcoten 1)}$}} &
((YB)^{\expcoten B {n+1}})^2 \ar[rr]_-m \ar[dd]^-{p_1p_1} &&
(YB)^{\expcoten B {n+1}} \ar[dddd]^-{p_1} \\
&& (YB \coten B I)(YB \coten B Y^{n-1}B) \ar[r]^-{1p_{YB}}
& 
(YB \coten B I)YB\ar[d]^-{1k1} &
(YB)^2 \ar[d]_-{11k1} \\
&&& (YB \coten B I)B^2 \ar[d]^-{1m} &
YB^3\ar[d]_-{11m} &
(YB)^2 \ar[rd]^-{1x1} \\
Y^{n+1}B  \ar@/^4pc/[rrruuu]^-{f\widetilde b_n}
\ar[r]^(.6){\raisebox{8pt}{${}_{11k\dots k1}$}}
\ar@/_1.5pc/[rrrdd]^-{b_{n+1}} \ar@/^1pc/[rruu]_-{fb_n} 
\ar@/_5pc/[rrrrrrrdd]_-{\widetilde b_{n+1}}&
Y^2B^n \ar[r]^-{1k1\dots 1} \ar@/^1.2pc/[rruu]^(.3){f1 m^{(n-1)}} &
YB^{n+1}  \ar@/^1pc/[ru]^-{f1 m^{(n-1)}}
\ar@/_1pc/[rd]_-{1m^{(n)}} &
(YB \coten B I)B \ar[r]^-{p_{YB}1} \ar[rrrrd]^-{q_1} &
YB^2 \ar[ru]^-{11u1} &&
Y^2B^2 \ar[rd]^-{mm} \\
&&&
YB \ar@{=}[rrrr] \ar[u]_-{f1} &&&&
YB \\
&&& YB \coten B Y^nB \ar[rrrr]^-{1\diagcoten\, \widetilde b_n}
\ar[u]^-{p_{YB}} &&&&
(YB)^{\expcoten B {n+1}} \ar[u]_-{p_1} 
\\
\\
\\
&&& (YB \coten B I)(YB)^{\expcoten B n} \ar[r]_-{p_{YB}1} \ar[d]^-{p_I 1}
\ar@/^1.4pc/[rrrr]^-{q_{n+1}} & 
YB(YB)^{\expcoten B n} \ar[d]_-{e11}
\ar[rr]_-{
(1\diagcoten u1\diagcoten \dots \diagcoten u1)(u1\diagcoten 1)}
&&
((YB)^{\expcoten B {n+1}})^2 \ar[r]_-m \ar[d]^-{p_{2\dots n+1}p_{2\dots n+1}} &
(YB)^{\expcoten B {n+1}} \ar[d]^-{p_{2\dots n+1}} \\
Y^{n+1}B \ar[rr]^-{e1\dots 11} \ar@/^1.1pc/[rrru]^-{f\widetilde b_n}
\ar@/_1pc/[rrrd]^-{b_{n+1}}
\ar@/_4pc/[rrrrrrrd]_-{\widetilde b_{n+1}} &&
Y^nB \ar[r]^-{\widetilde b_n} &
(YB)^{\expcoten B n} \ar[r]^-{u1} \ar@/_1.2pc/@{=}[rrrr] &
B (YB)^{\expcoten B n} \ar[rr]^-{(u1\diagcoten \cdots \diagcoten u1)1}  &&
((YB)^{\expcoten B n})^2 \ar[r]^-m &
(YB)^{\expcoten B n} \\
&&& YB \coten B Y^nB \ar[rrrr]^-{1\diagcoten\, \widetilde b_n}
\ar[lu]_(.4){p_{Y^nB}} &&&&
(YB)^{\expcoten B {n+1}} \ar[u]_-{p_{2\dots n+1}}}$
\end{sideways}
\caption{Proof of 
$\widetilde b_{n+1}=q_{n+1}.f \widetilde b_n$
}
\label{fig:b-q_induction}
\end{figure}

(5) By Theorem \ref{thm:SplitEpi_vs_DLaw} $q_1$ is an isomorphism without any
further assumption; it is the inverse of the isomorphism
$\xymatrix@C=15pt{YB \ar[r]^-{f1} &(YB \coten B I)B}$.
Also $b_1$ is an isomorphism; the inverse of the isomorphism
$\xymatrix@C=12pt{YB \coten B B \ar[r]^-{p_{YB}} & YB}$ in 
\cite[Proposition 3.6~(1)]{Bohm:Xmod_I}. 

Assume that $b_l$ is iso for some $l>1$. 
Take the diagram of part (4) for $n=1$; it says $b_2=q_2.f11$. Since $f$
is an isomorphism by definition and $b_2$ is an isomorphism by part (3), also
$q_2$ is an isomorphism. If $l=2$ then this completes the proof.
If $l>2$ then take next the diagram of part (4) for $n=2$; it says 
$(1\morcoten b_2).b_3=q_3.1q_2.ff11$. All of the occurring morphisms but $q_3$
are known to be isomorphisms proving that so is $q_3$.
Repeating this reasoning for all $n\leq l$ we conclude that $q_n$ is an
isomorphism for all $0<n\leq l$.

The opposite implication is proven by the same steps. Assume that $q_l$ is iso
for some $l>1$.  
Take the diagram of part (4) for $n=1$; it says $b_2=q_2.f11$. Since $f$
is an isomorphism by definition and $q_2$ is an isomorphism by Lemma
\ref{lem:hn}~(3), also $b_2$ is an isomorphism. If $l=2$ then this completes
the proof. 
If $l>2$ then take next the diagram of part (4) for $n=2$; it says 
$(1\morcoten b_2).b_3=q_3.1q_2.ff11$. All of the occurring morphisms but $b_3$
are known to be isomorphisms proving that so is $b_3$.
Repeating this reasoning for all $n\leq l$ we conclude that $b_n$ is an
isomorphism for all $0<n\leq l$.
\end{proof}

\begin{example} \label{ex:bn_groupoid}
Take $\mathcal S$ to be the (monoidal and admissible) class of all spans in the monoidal category $\mathsf C$ of spans over a given set. For any object of the category $\mathsf{ReflGraphMon}(\mathsf C)$ of Example \ref{ex:ReflGraph_groupoid} and for any positive integer $n$, the morphism $b_n$ in Lemma \ref{lem:bn}~(2) in invertible, see the pullback \eqref{eq:hn_groupoid}. 
\end{example}

\begin{example} \label{ex:bn_CoMon}
In the setting of Example \ref{ex:CoMon_SplitEpi}
we know from Example \ref{ex:hn_CoMon} that the morphism $q_n$ of Lemma
\ref{lem:hn}~(3) is invertible for any positive integer $n$ and for any object
of $\mathsf{ReflGraphMon}_{\mathcal S}(\mathsf C)$. By the isomorphism of
Theorem \ref{thm:ReflGraph_vs_PreX} this means that the morphism $q_n$ of
Lemma \ref{lem:bn}~(4) is invertible for any object of
$\mathsf{PreX}_{\mathcal S}(\mathsf C)$. Then also the morphism $b_n$ of Lemma
\ref{lem:bn}~(2) is invertible by Lemma \ref{lem:bn}~(5). Since the diagram 
$$
\xymatrix@C=20pt{
Y^nB \ar@{=}[dd] \ar[rd]_-{\delta_Y \delta_{Y^{n-1}B}} \ar[rrrrd]^(.6){\delta_Y1\dots 11} \ar@{=}[rrrrr] &&&&& 
Y^n B  \\
& Y^2(Y^{n-1}B)^2 \ar[rrr]_(.4){11\varepsilon_{Y^{n-1}B}1\dots 11} \ar[d]_-{1c_{Y,Y^{n-1}B}1\dots 11} &&&
Y^{n+1}B \ar[ru]^(.45){1\varepsilon_Y1\dots 11} \\
Y^nB \ar[r]_-{\delta_{Y^nB}} \ar[d]_-{b_n} &
(Y^nB)^2 \ar[r]_-{1k\dots k11\dots 11} \ar[d]^-{b_nb_n} \ar[rrru]^-(.2){1\varepsilon_{Y^{n-1}B}1\dots 11}&
YB^nY^nB \ar[rr]_-{1m^{(n-1)}1\dots 1} &&
YBY^nB \ar[u]^-{1\varepsilon_B1\dots 11} \ar[rd]_-{11\varepsilon_Y1\dots 11} \\
YB \coten B Y^{n-1}B\ar[r]^-\delta \ar@/_1.4pc/[rrrrr]_-j &
(YB \coten B Y^{n-1}B)^2 \ar[rrrr]^-{p_{YB}p_{Y^{n-1}B}} &&&&
YBY^{n-1}B \ar[uuu]_-{1\varepsilon_B1\dots 11} }
$$
commutes, we conclude that the morphism in its bottom-right path --- involving the equalizer $j$ as in \eqref{eq:cotensor} --- is the inverse of $b_n$. 
\end{example}

\begin{lemma} \label{lem:b2-u}
Let $\mathcal S$ be a monoidal admissible class  of spans in a monoidal
category $\mathsf C$ for which \cite[Assumption 4.1]{Bohm:Xmod_I} holds and let
$(B,Y,
\xymatrix@C=12pt{Y \ar[r]^-e & I},
\xymatrix@C=12pt{Y \ar[r]^-k & B},
\xymatrix@C=12pt{BY \ar[r]^-x & YB})$
be an object of the category $\mathsf{PreX}_{\mathcal S}(\mathsf C)$ in Theorem
\ref{thm:ReflGraph_vs_PreX} such that  
$\xymatrix@C=12pt{B & \ar[l]_-k Y  \ar@{=}[r] & Y}\in \mathcal S$.
For any positive integer $n$ the morphism $b_n$ in Lemma \ref{lem:bn}~(2) satisfies the following identities.
\begin{itemize}
\item[{(1)}] $b_2.u11=u1\morcoten 1$
\item[{(2)}] $b_2.1u1=1\morcoten u1$
\end{itemize}
\end{lemma}

\begin{proof}
Assertion (1) follows by the commutativity of the diagrams 
$$
\xymatrix@C=15pt@R=10pt{
&& Y^2B \ar[r]^-{b_2} \ar[d]^-{1k1} &
YB \coten B YB \ar[dd]^-{p_1} \\
&& YB^2 \ar@/^1pc/[rd]^-{1m} \\
YB \ar[r]^-{k1} \ar@/^1.2pc/[rruu]^-{u11} \ar@/_1.2pc/[rrrd]_-{u1\diagcoten 1} &
B^2 \ar@/^.9pc/[ru]^-{u11} \ar[r]^-m & 
B \ar[r]^-{u1} & 
YB \\
&&& YB \coten B YB \ar[u]_-{p_1}}
\qquad
\xymatrix@C=15pt@R=24pt{
& Y^2B \ar[r]^-{b_2} \ar[rd]_-{e11} &
YB \coten B YB \ar[d]^-{p_2} \\
YB \ar@{=}[rr]\ar@/^1pc/[ru]^-{u11} \ar@/_1.2pc/[rrd]_-{u1\diagcoten 1} && 
YB \\
&& YB \coten B YB \ar[u]_-{p_2}}
$$
and part (2) follows by the commutativity of
$$
\xymatrix@C=15pt{
& Y^2B \ar[r]^-{b_2} \ar[d]^-{1k1} &
YB \coten B YB \ar[d]^-{p_1} \\
YB \ar[r]^-{1u1} \ar@/^1pc/[ru]^-{1u1} \ar@/_1.2pc/[rrd]_-{1\diagcoten u1} \ar@/_1.1pc/@{=}[rr] &
YB^2  \ar[r]^-{1m} & 
YB \\
&& YB \coten B YB \ar[u]_-{p_1}}
\qquad
\xymatrix@C=15pt{
& Y^2B \ar[r]^-{b_2} \ar[rd]_-{e11} &
YB \coten B YB \ar[d]^-{p_2} \\
YB \ar[r]^-{e1} \ar@/^1pc/[ru]^-{1u1} \ar@/_1.2pc/[rrd]_-{1\diagcoten u1} &
B \ar[r]^-{u1} & 
YB \\
&& YB \coten B YB. \ar[u]_-{p_2}}
$$
\end{proof}

\subsection{The composition morphism of a relative category of monoids}

\begin{proposition} \label{prop:relcat-monoid}
Consider a monoidal admissible class $\mathcal S$ of spans in a monoidal
category $\mathsf C$ such that \cite[Assumption 4.1]{Bohm:Xmod_I} holds. 
Take an object
$\xymatrix@C=20pt{
B \ar@{ >->}[r]|(.55){\, i\, } &
A \ar@{->>}@<-4pt>[l]_-s  \ar@{->>}@<4pt>[l]^-t}$ 
of the category $\mathsf{ReflGraphMon}_{\mathcal S}(\mathsf C)$ of Theorem
\ref{thm:ReflGraph_vs_PreX} such that the following properties hold.
\begin{itemize}
\item $\xymatrix@C=12pt{B& \ar[l]_-t A \ar@{=}[r] & A}$ belongs to $\mathcal S$
\item the morphism $q_3$ of Lemma \ref{lem:hn}~(3) is invertible.
\end{itemize}
The following assertions hold.
\begin{itemize}
\item[{(1)}] There is at most one monoid morphism $d$ rendering commutative
$$
\xymatrix@C=35pt@R=20pt{
A \ar[r]^-{1\diagcoten \, i} \ar@/_1.2pc/@{=}[rd] &
A \coten B A  \ar@{-->}[d]^-d &
A. \ar[l]_-{i \, \diagcoten 1} \ar@/^1.2pc/@{=}[ld] \\
& A}
$$
\item[{(2)}] The monoid morphism $d$ of part (1) exists if and only if the following
  diagram commutes (recall that $q_2$ is invertible by Lemma \ref{lem:hn}~(3)). 
$$
\xymatrix@C=35pt{
A(A \coten B I)\ar[r]^-{1p_A} \ar[d]_-{1p_A} &
A^2 \ar[r]^-{(i \, \diagcoten 1)(1\diagcoten \, i)} &
(A \coten B A)^2 \ar[r]^-m &
A \coten B A \ar[r]^-{q_2^{-1}} &
(A \coten B I)A \ar[d]^-{p_A 1} \\
A^2 \ar[rr]_-m &&
A &&
A^2 \ar[ll]^-m}
$$
Moreover, in this case $d$ is equal to
$\xymatrix@C=15pt{
A\coten B A \ar[r]^-{q_2^{-1}} &
(A\coten B I)A \ar[r]^-{p_A1} &
A^2 \ar[r]^-m &
A.}$
\item[{(3)}] Whenever the monoid morphism $d$ of part (1) exists, 
$\xymatrix@C=20pt{
B \ar@{ >->}[r]|(.55){\, i\, } &
A \ar@{->>}@<-4pt>[l]_-s  \ar@{->>}@<4pt>[l]^-t &
\ar[l]_-d A \coten B A}$
is an $\mathcal S$-relative category in the category of monoids in $\mathsf
C$. 
\end{itemize}
\end{proposition}

\begin{proof}
The proof is built on \cite[Corollary 1.7]{Bohm:Xmod_I}.

(1) Since the morphism $q_2$ in Lemma \ref{lem:hn}~(3) is invertible, we know
from \cite[Corollary 1.7]{Bohm:Xmod_I} that there is at most one monoid
morphism rendering commutative 
\begin{equation}\label{eq:lem_d}
\xymatrix@C=15pt@R=20pt{
A \coten B I \ar[r]^-{p_A} \ar@/_1.2pc/[rrd]_-{p_A} &
A \ar[r]^-{1\diagcoten \, i} 
&
A \coten B A  \ar@{-->}[d] &&&
A. \ar[lll]_-{i \, \diagcoten 1} \ar@/^1.2pc/@{=}[llld] \\
&& A}
\end{equation}
Since a monoid morphism $d$ as in part (1) obviously renders commutative
\eqref{eq:lem_d}, this proves its uniqueness.

(2) By \cite[Corollary 1.7]{Bohm:Xmod_I} commutativity of the diagram
of part (2) is equivalent to the existence of a (unique) monoid morphism making
\eqref{eq:lem_d} commute. Since a monoid morphism $d$ in part (1) provides such a
morphism, its existence implies commutativity of the diagram of part (2).

In order to prove the converse implication, we show that any monoid morphism $d$
making \eqref{eq:lem_d} commute renders commutative also the diagram of part
(1). Recall from \cite[Lemma 1.2]{Bohm:Xmod_I} that the invertibility of $q$ in
Theorem \ref{thm:ReflGraph_vs_PreX}~(b) implies that $p_A$ and $i$
are joint epimorphisms of monoids. Hence if $d$ makes \eqref{eq:lem_d} commute then it
does so the left hand side of the diagram of part (1) by $d.(1\morcoten
i).i=d.(i\morcoten 1).i=i$. 

The stated expression of $d$ immediately follows from 
\cite[Corollary 1.7]{Bohm:Xmod_I}. 

(3) In order to see that the monoid morphism $d$ in part (1) is a morphism of spans, we use that by the invertibility of $q_2$ there are unique morphisms rendering commutative the respective diagrams 
$$
\xymatrix@C=15pt@R=20pt{
A \coten B I \ar[r]^-{p_A} \ar@/_1.2pc/[rrd]_-{s.p_A} &
A \ar[r]^-{1\diagcoten \, i} 
&
A \coten B A  \ar@{-->}[d] &&&
A \ar[lll]_-{i \, \diagcoten 1} \ar@/^1.2pc/[llld]^-s \\
&& B}
\quad \textrm{and} \quad
\xymatrix@C=15pt@R=20pt{
A \coten B I \ar[r]^-{p_A} \ar@/_1.2pc/[rrd]_-{t.p_A} &
A \ar[r]^-{1\diagcoten \, i} 
&
A \coten B A  \ar@{-->}[d] &&&
A, \ar[lll]_-{i \, \diagcoten 1} \ar@/^1.2pc/[llld]^-t \\
&& B}
$$
see \cite[Corollary 1.7]{Bohm:Xmod_I}. Now $s.d$ obviously makes the
first diagram commute and so does 
$\xymatrix@C=12pt{A\coten B A \ar[r]^-{p_2} & A \ar[r]^-s & B}$ by the
commutativity of 
$$
\xymatrix@R=15pt{
A\coten B I \ar[r]^-{p_A} &
A  \ar[r]^-{1\diagcoten \, i} \ar[d]_-s &
A \coten B A \ar[d]^-{p_2} \\
& B \ar[r]^-i \ar@/_1.2pc/@{=}[rd] &
A \ar[d]^-s \\
&& B}\qquad \textrm{and} \qquad
\xymatrix@R=15pt{
A \ar[r]^-{i \, \diagcoten 1} \ar@/_1.2pc/@{=}[rd] &
A \coten B A \ar[d]^-{p_2} \\
& A \ar[d]^-s \\
& B}
$$
Thus they are equal.
Similarly, both $t.d$ and 
$\xymatrix@C=12pt{A\coten B A \ar[r]^-{p_1} & A \ar[r]^-t & B}$ render
commutative the second diagram proving that they are equal.

The to-be composition morphism  $d$ in part (1) admits the unit $i$ by
construction. Its associativity follows again by 
\cite[Corollary 1.7]{Bohm:Xmod_I} since by the invertibility of $q_3$ there is at
most one morphism rendering commutative 
$$
\xymatrix@C=20pt@R=20pt{
A \coten B I \ar[r]^-{p_A} \ar@/_1.1pc/[rrd]_-{p_A} &
A \ar[r]^-{1\diagcoten \, i \, \diagcoten \, i} 
&
A \coten B A \coten B A \ar@{-->}[d] &&&
A\coten B A . \ar[lll]_-{i \, \diagcoten 1} \ar@/^1.1pc/[llld]^-d \\
&& A}
$$
Since both $d.(d\morcoten 1)$ and $d.(1\morcoten d)$ do so by the
commutativity of
$$
\xymatrix@C=15pt@R=15pt{
A \coten B I \ar[r]^-{p_A} &
A \ar[r]^-{1\diagcoten \, i \, \diagcoten \, i} \ar[rd]|-{1\diagcoten \, i}
\ar@/_1.2pc/@{=}[rdd] &
A^{\expcoten B 3} \ar[d]^-{d\diagcoten 1} \\
&& A \coten B A \ar[d]^-d \\
&& A}\ 
\xymatrix@C=15pt@R=15pt{
A \coten B I \ar[r]^-{p_A} &
A \ar[r]^-{1\diagcoten \, i \, \diagcoten \, i} \ar[rd]|-{1\diagcoten \, i}
\ar@/_1.2pc/@{=}[rdd] &
A^{\expcoten B 3} \ar[d]^-{1\diagcoten d} \\
&& A \coten B A \ar[d]^-d \\
&& A}\ 
\xymatrix@C=15pt@R=15pt{
A \coten B A \ar[r]^-{i \, \diagcoten 1\diagcoten 1} 
\ar@/_1.2pc/@{=}[rd] &
A^{\expcoten B 3} \ar[d]^-{d\diagcoten 1} \\
& A \coten B A \ar[d]^-d \\
& A}\ 
\xymatrix@C=15pt@R=15pt{
A \coten B A \ar[r]^-{i \, \diagcoten 1\diagcoten 1} \ar[d]_-d &
A^{\expcoten B 3} \ar[d]^-{1\diagcoten d} \\
A \ar[r]^-{i \, \diagcoten 1} \ar@/_1.2pc/@{=}[rd] & 
A \coten B A \ar[d]^-d \\
& A}
$$
this proves their
equality (modulo the omitted associativity isomorphism in 
\cite[Proposition 3.6]{Bohm:Xmod_I}). 
\end{proof}

\begin{proposition} \label{prop:S-functor_monoid}
Consider a monoidal admissible class $\mathcal S$ of spans in a monoidal
category $\mathsf C$ such that \cite[Assumption 4.1]{Bohm:Xmod_I} holds. 
Between $\mathcal S$-relative categories
in the category of monoids in $\mathsf C$ for which the morphisms $q_2$ in
Lemma \ref{lem:hn}~(3) are invertible, 
any morphism of reflexive graphs of monoids
is in fact an $\mathcal S$-relative functor.
\end{proposition}

\begin{proof}
Take $\mathcal S$-relative categories 
$\xymatrix@C=20pt{
B \ar@{ >->}[r]|(.55){\, i\, } &
A \ar@{->>}@<-4pt>[l]_-s  \ar@{->>}@<4pt>[l]^-t &
\ar[l]_-d A \coten B A}$ and
$\xymatrix@C=20pt{
B' \ar@{ >->}[r]|(.55){\, i'\, } &
A' \ar@{->>}@<-4pt>[l]_-{s'}  \ar@{->>}@<4pt>[l]^-{t'} &
 \ar[l]_-{d'} A' \coten {B'} A'}$
as in the claim. We need to check the compatibility of any morphism of reflexive graphs 
$(\xymatrix@C=12pt{B \ar[r]^-b & B'},\xymatrix@C=12pt{A \ar[r]^-a & A'})$
with the composition morphisms $d$ and $d'$. 
The first diagram of  
$$
\xymatrix@C=14pt{
(A \coten B I)A \ar[r]_-{p_A1} \ar[d]_-{(a\diagcoten 1)a} 
\ar@/^1.1pc/[rrrr]^-{q_2} &
A^2 \ar[rr]_-{(1\diagcoten \, i)(i \, \diagcoten 1)} \ar[d]^-{aa} &&
(A \coten B A)^2 \ar[r]_-m \ar[d]^-{(a\diagcoten a)(a\diagcoten a)} &
A \coten B A \ar[d]^-{a\diagcoten a} \\
(A' \coten {B'} I)A' \ar[r]^-{p_{A'}1} \ar@/_1.1pc/[rrrr]_-{q'_2}&
A^{\prime 2} \ar[rr]^-{(1\diagcoten \, i')(i'\diagcoten 1)}  &&
(A' \coten {B'} A')^2 \ar[r]^-{m'}  &
A' \coten {B'} A' }
\quad
\xymatrix@C=12pt@R=24pt{
A \coten B A \ar[r]_-{q_2^{-1}} \ar[d]_-{a\diagcoten a} 
\ar@/^1.1pc/[rrr]^-d &
(A \coten B I)A \ar[r]_-{p_A1} \ar[d]^-{(a\diagcoten 1)a} &
A^2 \ar[r]_-m \ar[d]^-{aa} &
A \ar[d]^-a \\
A' \coten {B'} A' \ar[r]^-{q_2^{\prime -1}}  \ar@/_1.1pc/[rrr]_-{d'} &
(A' \coten {B'} I)A' \ar[r]^-{p_{A'}1}  &
A^{\prime 2} \ar[r]^-{m'} &
A' }
$$
commutes since $a\morcoten a$ is multiplicative by \cite[Proposition 3.7~(2)]{Bohm:Xmod_I} and
by the functoriality of $\Box$; see \cite[Proposition 3.5~(2)]{Bohm:Xmod_I}. It is used to prove the commutativity of the second
diagram. 
\end{proof}

\subsection{The equivalence between relative categories and crossed modules of monoids} 

\begin{theorem}\label{thm:SCat_vs_Xmod}
Consider a monoidal admissible class $\mathcal S$ of spans in a monoidal category $\mathsf C$ such that \cite[Assumption 4.1]{Bohm:Xmod_I} holds. Use the same notation $\mathcal S$ for the induced admissible class of spans in the category of monoids in $\mathsf C$ from \cite[Example 2.8]{Bohm:Xmod_I} (also satisfying \cite[Assumption 4.1]{Bohm:Xmod_I} by \cite[Example 4.4]{Bohm:Xmod_I}). The following categories are equivalent.
\begin{itemize}
\item[{$\mathsf{Cat}$}]\hspace{-.3cm} $\mathsf{Mon}_{\mathcal S}(\mathsf C)$ whose \\
\underline{objects} are $\mathcal S$-relative categories 
$\xymatrix@C=20pt{
B \ar@{ >->}[r]|(.55){\, i\, } &
A \ar@{->>}@<-4pt>[l]_-s  \ar@{->>}@<4pt>[l]^-t &
A\coten B A \ar[l]_-d}$ 
in the category of monoids in $\mathsf C$ such that the morphisms $q_n$ of \eqref{eq:qn} are invertible for any positive integer $n$. \\
\underline{morphisms} are $\mathcal S$-relative functors in the category of monoids in $\mathsf C$.
\smallskip

\item[{$\mathsf{X}\hspace{.3cm}$}]\hspace{-.58cm} $\mathsf{mod}_{\mathcal S}(\mathsf C)$ whose \\ 
\underline{objects} consist of monoids $B$ and $Y$, monoid morphisms
$\xymatrix@C=12pt{Y \ar[r]^-e & I}$ and
$\xymatrix@C=12pt{Y \ar[r]^-k& B}$
and a distributive law
$\xymatrix@C=12pt{BY \ar[r]^-x & YB}$ subject to the following conditions.
\begin{itemize}
\item[{(a')}] 
$\xymatrix@C=12pt{
B & \ar[l]_-k  Y \ar@{=}[r] & Y}\in \mathcal S$,
$\xymatrix@C=12pt{
Y \ar@{=}[r] & Y \ar[r]^-e & I}\in \mathcal S$
and
$\xymatrix@C=12pt{
B \ar@{=}[r] & B \ar@{=}[r] & B}\in \mathcal S$.
\item[{(b')}] $e1.x=1e$ and $m.k1.x=m.1k$.
\item[{(c')}] The morphism $f$ of Theorem \ref{thm:SplitEpi_vs_DLaw}~(c') is invertible and the morphisms $b_n$ of Lemma \ref{lem:bn}~(2) are invertible for all positive integers $n$. 
\item[{(d')}] Regarding $YB$ as a monoid via the structure induced by the distributive law $x$, the following diagram commutes.
$$
\xymatrix@C=10pt{
YBY \ar[rr]^-{u111uu} \ar[d]_-{1x} &&
(Y^2B)^2 \ar[rr]^-{b_2b_2} &&
(YB \coten B YB)^2 \ar[rr]^-m &&
YB \coten B YB \ar[d]^-{b_2^{-1}} \\
Y^2B \ar[rrr]_-{m1} &&&
YB &&&
Y^2B \ar[lll]^-{m1}}
$$
\end{itemize}
\underline{morphisms} are pairs of monoid morphisms 
$(\xymatrix@C=12pt{B \ar[r]^-b & B'},\xymatrix@C=12pt{Y \ar[r]^-y & Y'})$
such that $e'.y=e$, $k'.y=b.k$ and $x'.by=yb.x$.
\end{itemize}
\end{theorem}

\begin{proof}
It follows by Proposition \ref{prop:relcat-monoid} and Proposition \ref{prop:S-functor_monoid} that $\mathsf{CatMon}_{\mathcal S}(\mathsf C)$ is a full subcategory of $\mathsf{ReflGraphMon}_{\mathcal S}(\mathsf C)$ and obviously $\mathsf{Xmod}_{\mathcal S}(\mathsf C)$ is a full subcategory of $\mathsf{PreX}_{\mathcal S}(\mathsf C)$. Below we show that the mutually inverse functors of Theorem \ref{thm:ReflGraph_vs_PreX} restrict to functors between these subcategories thus establishing the stated equivalence.

Regarding an object 
$\xymatrix@C=20pt{
B \ar@{ >->}[r]|(.55){\, i\, } &
A \ar@{->>}@<-4pt>[l]_-s  \ar@{->>}@<4pt>[l]^-t&
A\coten B A \ar[l]_-d}$ 
of $\mathsf{CatMon}_{\mathcal S}(\mathsf C)$ as an object
$\xymatrix@C=20pt{
B \ar@{ >->}[r]|(.55){\, i\, } &
A \ar@{->>}@<-4pt>[l]_-s  \ar@{->>}@<4pt>[l]^-t}$
of $\mathsf{ReflGraphMon}_{\mathcal S}(\mathsf C)$, the functor in the proof of Theorem \ref{thm:ReflGraph_vs_PreX} takes it to the object
$(B,A\coten B I,
\xymatrix@C=12pt{A\coten B I \ar[r]^-{p_I} & I},
\xymatrix@C=12pt{A\coten B I \ar[r]^-{p_A} & A \ar[r]^-t & B},
\xymatrix@C=12pt{B(A\coten B I) \ar[r]^-{ip_A} & A^2 \ar[r]^-m & A \ar[r]^-{q^{-1}} & (A\coten B I)B})$
of the category  $\mathsf{PreX}_{\mathcal S}(\mathsf C)$; we claim that it is in fact an object of 
$\mathsf{Xmod}_{\mathcal S}(\mathsf C)$.

It satisfies the condition
$\xymatrix@C=12pt{B & \ar[l]_-t A & \ar[l]_-{p_A}A\coten B I \ar@{=}[r] & A}\in \mathcal S$ by Lemma \ref{lem:t-k_in_S}. 

From Lemma \ref{lem:bn}~(5) we know that the morphism $b_n$ of Lemma \ref{lem:bn}~(2) is invertible  if and only if the left column of the commutative diagram
$$
\xymatrix@C=55pt @R=15pt{
((A\coten B I)B \coten B I)((A\coten B I)B)^{\expcoten B n-1} \ar[d]_-{p_{(A\diagcoten_B I)B}1} \ar[r]^-{(q\diagcoten 1)q^{\diagcoten n-1}} &
(A \coten B I)A^{\expcoten B n-1} \ar[d]_-{p_A 1}   \ar@/^3pc/[ddd]^-{q_n} \\
(A\coten B I)B((A\coten B I)B)^{\expcoten B n-1} \ar[r]^-{qq^{\diagcoten n-1}}
\ar[d]_-{(1\diagcoten u1 \diagcoten \cdots \diagcoten u1)(u1\diagcoten 1)} &
AA^{\expcoten B n-1} \ar[d]_-{(1\diagcoten \, i \, \diagcoten \cdots \diagcoten \, i)(i \, \diagcoten 1)}  \\
(((A\coten B I)B)^{\expcoten B n})^2 \ar[d]_-m \ar[r]^-{q^{\diagcoten n}q^{\diagcoten n}} &
(A^{\expcoten B n})^2 \ar[d]_-m \\
((A\coten B I)B)^{\expcoten B n} \ar[r]_-{q^{\diagcoten n}} &
A^{\expcoten B n} }
$$
is invertible. Recognize the isomorphism $q_n$ in the right column. Since also the rows are isomorphisms  by assumption, so is the left column and hence $b_n$.

The proof of the commutativity of the diagram in part (d') requires some preparation. The commutativity of
$$
\xymatrix@R=10pt@C=10pt{
&&&& (A\coten B I)B \coten B (A\coten B I)B \ar[r]^-{q\diagcoten q} \ar[d]^-{p_1} &
A \coten B A\ar[d]^-{p_1} \\
(A\coten B I)^2B \ar@/^1.2pc/[rrrru]^-{b_2} \ar@/_1.5pc/[rrrrdd]_-{1q} \ar[r]_-{1p_A1} &
(A\coten B I)AB \ar@{=}[r] \ar[rd]_-{11i} &
(A\coten B I)AB \ar[r]^-{1t1} &
(A\coten B I)B^2 \ar[r]^-{1m} &
(A\coten B I)B \ar[r]^-q  &
A \\
&& (A\coten B I)A^2\ar[u]_-{11t} \ar[rrd]^-{1m}  \\
&&&& (A\coten B I)A\ar[uu]_-{1t} \ar[r]_-{q_2} &
A \coten B A  \ar[uu]_-{p_1}}
$$
$$
\xymatrix@R=10pt@C=112pt{
& (A\coten B I)B \coten B (A\coten B I)B \ar[r]^-{q\diagcoten q} \ar[d]^-{p_2} &
A \coten B A\ar[d]^-{p_2} \\
(A\coten B I)^2B \ar@/^1.2pc/[ru]^-{b_2} \ar@/_1.1pc/[rd]_-{1q} \ar[r]^-{p_I 11} &
(A\coten B I)B \ar[r]^-q   &
A \\
&(A\coten B I)A\ar[ru]^(.45){p_I 1} \ar[r]_-{q_2} &
A \coten B A  \ar[u]_-{p_2}}
$$
proves $(q\morcoten q).b_2=q_2.1q$. (Here the bottom-right region of the first diagram commutes since the lower half of the diagram of \eqref{eq:hn_1} commutes and the bottom-right region of the second diagram commutes since the lower half of the diagram of \eqref{eq:hn_2} commutes.) By the associativity of $A$ and the multiplicativity of $\xymatrix@C=12pt{A \coten B I \ar[r]^-{p_A} & A}$ also the following diagram commutes. 
$$
\xymatrix@R=15pt{
(A\coten B I)^2B \ar[r]_-{1p_Ai} \ar[dd]_-{m1} \ar@/^1.1pc/[rr]^-{1q} &
(A\coten B I)A^2 \ar[r]_-{1m} \ar[d]^-{p_A11} &
(A\coten B I)A \ar[d]^-{p_A1} \\
& A^3 \ar[r]^-{1m} \ar[d]^-{m1} &
A^2 \ar[d]^-m \\
(A\coten B I)B \ar[r]^-{p_Ai} \ar@/_1.1pc/[rr]_-q &
A^2 \ar[r]^-m &
A}
$$
With the help of these identities and Lemma \ref{lem:b2-u}, and using that the region marked by $(\ast)$ commutes by Proposition \ref{prop:relcat-monoid}~(2), the diagram of Figure \ref{fig:d'} is seen to commute. This proves that the stated object belongs to $\mathsf{Xmod}_{\mathcal S}(\mathsf C)$ indeed.
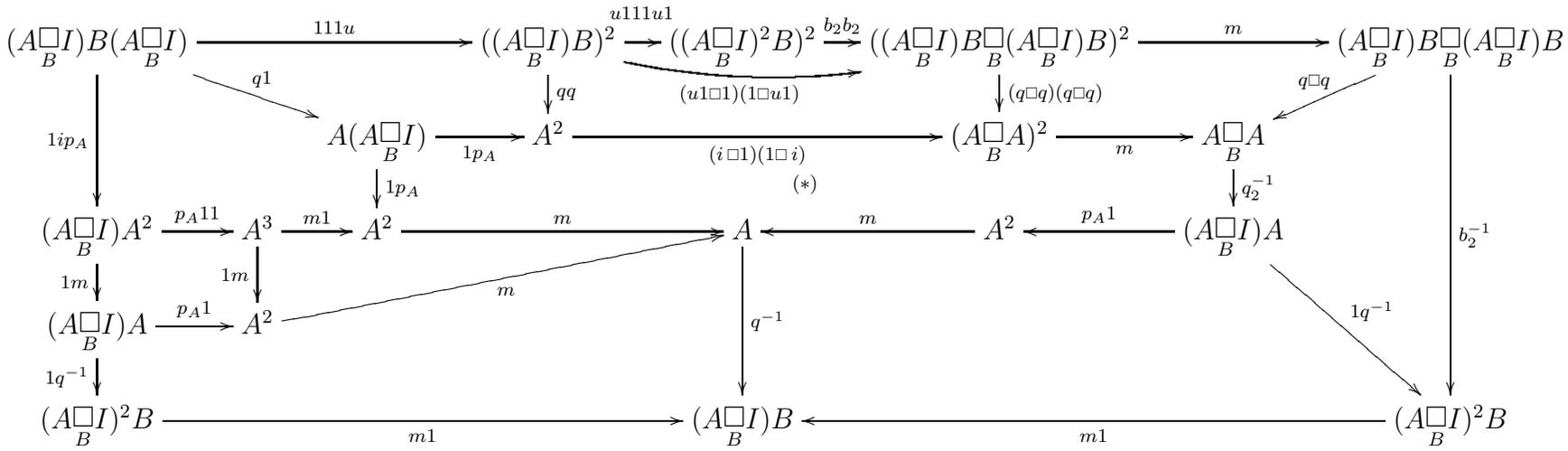
\begin{figure} 
\centering
\begin{sideways}
$\xymatrix@C=15pt@R=15pt{
(A\coten B I)B(A\coten B I) \ar[rrr]^-{111u} \ar[rrd]^-{q1} \ar[dd]_-{1ip_A} &&&
((A\coten B I)B)^2 \ar[r]^-{\raisebox{8pt}{${}_{u111u1}$}}  \ar[d]^-{qq}\ar@/_1.2pc/[rr]_-{(u1\diagcoten 1)(1\diagcoten u1)} &
((A\coten B I)^2B)^2\ar[r]^-{b_2b_2} &
((A\coten B I)B \coten B (A\coten B I)B)^2 \ar[rr]^-m \ar[d]^-{(q\diagcoten q)(q\diagcoten q)} &&
(A\coten B I)B \coten B (A\coten B I)B \ar[ld]_-{q\diagcoten q} \ar[dddd]^-{b_2^{-1}} \\
&& A(A\coten B I)\ar[r]_-{1p_A} \ar[d]^-{1p_A}  \ar@{}[rrrrd]|-{(\ast)}&
A^2 \ar[rr]_-{(i \, \diagcoten 1)(1\diagcoten \, i)} &&
(A\coten B A)^2 \ar[r]_-m &
A \coten B A \ar[d]^-{q_2^{-1}} \\
(A\coten B I)A^2 \ar[r]^-{p_A11} \ar[d]_-{1m} &
A^3 \ar[r]^-{m1} \ar[d]_-{1m} &
A^2 \ar[rr]^-m &&
A \ar[dd]^-{q^{-1}} &
A^2 \ar[l]_-m &
(A\coten B I)A\ar[l]_-{p_A1} \ar[rdd]^-{1q^{-1}} \\
(A\coten B I)A \ar[r]^-{p_A1} \ar[d]_-{1q^{-1}} &
A^2 \ar[rrru]_-m \\
(A\coten B I)^2B \ar[rrrr]_-{m1} &&&&
(A\coten B I)B &&&
(A\coten B I)^2B\ar[lll]^-{m1}}$
\end{sideways}
\caption{Commutativity of the diagram in part (d')}
\label{fig:d'}
\end{figure}

In the opposite direction, consider an object 
$(B,Y,
\xymatrix@C=12pt{Y \ar[r]^-e&I},
\xymatrix@C=12pt{Y \ar[r]^-k&B},
\xymatrix@C=12pt{BY \ar[r]^-x&YB})$
of $\mathsf{Xmod}_{\mathcal S}(\mathsf C)$ as an object of $\mathsf{PreX}_{\mathcal S}(\mathsf C)$. The functor in the proof of Theorem \ref{thm:ReflGraph_vs_PreX} takes it to the object 
$\xymatrix@C=30pt{
B \ar@{ >->}[r]|(.55){\, u1\, } &
A \ar@{->>}@<-4pt>[l]_-{e1}  \ar@{->>}@<4pt>[l]^-{m.k1}}$
of $\mathsf{ReflGraphMon}_{\mathcal S}(\mathsf C)$; we claim that it can be seen as an object of $\mathsf{CatMon}_{\mathcal S}(\mathsf C)$.

By Lemma \ref{lem:bn}~(1) the span $\xymatrix@C=12pt{B & \ar[l]_-m B^2 & \ar[l]_-{k1} YB \ar@{=}[r] & YB}$ belongs to $\mathcal S$.

The morphism $q_n$ of Lemma \ref{lem:bn}~(4) is invertible for all positive integers $n$ by Lemma \ref{lem:bn}~(5).

By Proposition \ref{prop:relcat-monoid}~(2) and (3), the reflexive graph of monoids
$\xymatrix@C=30pt{
B \ar@{ >->}[r]|(.55){\, u1\, } &
A \ar@{->>}@<-4pt>[l]_-{e1}  \ar@{->>}@<4pt>[l]^-{m.k1}}$
extends to an $\mathcal S$-relative category in the category of monoids in $\mathsf C$ by the commutativity of
$$
\xymatrix@C=15pt{
YB(YB \coten B I) \ar[r]^-{11p_{YB}} \ar[ddd]_-{11p_{YB}} \ar[rd]_-{11f^{-1}} &
(YB)^2 \ar[rr]^-{(u1\diagcoten 1)(1\diagcoten u1)} \ar@/_.8pc/[rd]^-{u111u1} &&
(YB \coten B YB)^2 \ar[r]^-m &
YB \coten B YB \ar[dd]^-{b_2^{-1}} \ar[r]^-{q_2^{-1}} &
(YB\coten B I)YB \ar@/^1.4pc/[ddd]^-{p_{YB}11} \ar[dd]_-{f^{-1}11}\\
& YBY \ar[d]^-{1x} \ar[ldd]_-{111u} \ar[u]^-{111u} &
(Y^2B)^2 \ar@/_.8pc/[ru]^-{b_2b_2} \ar@{}[rrd]|-{\textrm{(d')}} \\
& Y^2B \ar[rrd]^-{m1} &&& 
Y^2B \ar[ld]_-{m1} \ar@{=}[r] &
Y^2B\ar[d]_-{1u11} \ar[ld]_(.45){11u1}\\
(YB)^2 \ar[r]_-{1x1} &
Y^2B^2 \ar[rr]_-{mm} \ar[u]_-{11m} &&
YB &
Y^2B^2 \ar[l]^-{mm} \ar[u]_-{11m} &
(YB)^2 . \ar[l]^-{1x1}}
$$
The region at the top-right corner is the commutative diagram of Lemma \ref{lem:bn}~(4) for $n=1$. 
The region bounded from below by the curved arrows commutes by Lemma \ref{lem:b2-u}. 
The region marked by (d')  is coincides with the diagram of part (d') hence it commutes.
\end{proof}

\begin{example} \label{ex:SCat_groupoid}
As in Example \ref{ex:SplitEpi_groupoid}, take the (evidently admissible and
monoidal) class of all spans in the category $\mathsf C$ of spans over a given
set $X$. Then the equivalent categories of Theorem \ref{thm:SCat_vs_Xmod} take
the following forms. 
\begin{itemize}
\item[{$\mathsf{Cat}$}]\hspace{-.3cm} $\mathsf{Mon}(\mathsf C)$whose \\
\underline{objects} are the double categories with the object set $X$ and only identity horizontal morphisms and such that the morphism \eqref{eq:q_groupoid} is invertible. (This last condition holds e.g. if the vertical edge category is a groupoid.) \\
\underline{morphisms} are the double functors which are identities on the objects (and hence on the horizontal morphisms).
\item[{$\mathsf{X}\hspace{.3cm}$}]\hspace{-.58cm} $\mathsf{mod}(\mathsf C)$
  whose \\
\underline{objects} consist of categories $B$ and $Y$ with the common object
set $X$ such that in $Y$ there are no morphisms between different objects; 
an action (see Example \ref{ex:SplitEpi_groupoid}) 
$\xymatrix@C=12pt{B \coten X Y \ar[r]^-\triangleright & Y}$
and an identity-on-objects functor 
$\xymatrix@C=12pt{Y \ar[r]^-\kappa & B}$
such that 
$$
\kappa(b \triangleright y).b=b.\kappa(y)
\qquad \textrm{and}\qquad
(\kappa(y) \triangleright y').y=y.y'
$$
for all morphisms $b$ in $B$ and $y,y'$ in $Y$ for which $s(b)=t(y)=t(y')$.\\
\underline{morphisms} are the same as the morphisms in $\mathsf{PreXMon(C)}$,
see Example \ref{ex:ReflGraph_groupoid}.
\end{itemize}

These equivalent categories have equivalent full subcategories in whose objects the
occurring category $B$ is a groupoid; and other equivalent full subcategories
in whose objects both occurring categories are groupoids. In the latter case
these are the category of categories in the category of groupoids; and the
category of {\em crossed modules of groupoids} in \cite[Definition
1.2]{BrownIcen}, respectively.   
\end{example}

\begin{example} \label{ex:SCat_CoMon}
In the setting of Example \ref{ex:CoMon_SplitEpi}, the equivalent categories of Theorem \ref{thm:SCat_vs_Xmod} take the following explicit forms. 
\begin{itemize}
\item[{$\mathsf{Cat}$}]\hspace{-.3cm} $\mathsf{Mon}_{\mathcal S}(\mathsf C)$ whose \\
\underline{objects} are $\mathcal S$-relative categories 
$\xymatrix@C=20pt{
B \ar@{ >->}[r]|(.55){\, i\, } &
A \ar@{->>}@<-4pt>[l]_-s  \ar@{->>}@<4pt>[l]^-t &
A\coten B A \ar[l]_-d}$ 
in the category of mon\-oids in $\mathsf C$ --- that is, in the category of bimonoids in $\mathsf M$ --- such that the morphism $q$ of Theorem \ref{thm:SplitEpi_vs_DLaw}~(b) is invertible.  
\\
\underline{morphisms} are $\mathcal S$-relative functors in the category of monoids in $\mathsf C$ --- that is, in the category of bimonoids in $\mathsf M$.
\smallskip

\item[{$\mathsf{X}\hspace{.3cm}$}]\hspace{-.58cm} $\mathsf{mod}_{\mathcal S}(\mathsf C)$ whose \\ 
\underline{objects} consist of a bimonoid $Y$ and a cocommutative bimonoid $B$ together with a left action 
$\xymatrix@C=12pt{BY \ar[r]^-l & Y}$ which makes $Y$ both a $B$-module monoid and a $B$-module comonoid and a bimonoid morphism 
$\xymatrix@C=12pt{Y \ar[r]^-k & B}$ for which the following diagrams commute.
$$
\xymatrix@C=10pt{
Y \ar[r]^-\delta \ar[d]_-\delta & 
Y^2 \ar[r]^-c &
Y^2 \ar[d]^-{k1} \\
Y^2 \ar[rr]_-{k1} &&
BY}
\quad
\xymatrix@C=2pt{
BY \ar[rr]^-{\delta 1} \ar[d]_-{1k} &&
B^2Y \ar[rr]^-{1c} &&
BYB \ar[rr]^-{l1} &&
YB \ar[d]^-{k1} \\
B^2 \ar[rrr]_-m &&&
B &&&
B^2 \ar[lll]^-m}
\quad
\xymatrix@C=2pt{
Y^2 \ar[rr]^-{\delta 1} \ar@{=}[d] &&
Y^3 \ar[rr]^-{1c} &&
Y^3 \ar[rr]^-{k11} &&
BY^2 \ar[d]^-{l1} \\
Y^2 \ar[rrr]_-m &&&
Y &&&
Y^2 \ar[lll]^-m}
$$
The third condition appears in \cite[Definition 12~(v)]{Villanueva} under the name {\em Peiffer condition} (motivated by the terminology for groups). \\
\underline{morphisms} are pairs of monoid morphisms 
$(\xymatrix@C=12pt{B \ar[r]^-b & B'},\xymatrix@C=12pt{Y \ar[r]^-y & Y'})$
such that $e'.y=e$, $k'.y=b.k$ and $x'.by=yb.x$.
\end{itemize}
These equivalent categories are equivalent furthermore to the full subcategory of $\mathsf{ReflGraphMon}_{\mathcal S}(\mathsf C)$ of Example \ref{ex:ReflGraph_groupoid} for whose objects 
$\xymatrix@C=20pt{
B \ar@{ >->}[r]|(.55){\, i\, } &
A \ar@{->>}@<-4pt>[l]_-s  \ar@{->>}@<4pt>[l]^-t}$
the following diagrams commute.
\begin{equation}
\xymatrix@C=142pt{
A \ar[r]^-\delta \ar[d]_-\delta & 
A^2 \ar[r]^-c &
A^2 \ar[d]^-{t1} \\
A^2 \ar[rr]_-{t1} &&
BA}
\end{equation}
\begin{equation}\label{eq:d_exists_CoMon}
\xymatrix{
A(A \coten B I) \ar[r]^-{\delta 1} \ar[d]_-{1p_A} &
A^2(A \coten B I) \ar[r]^-{tc} &
B(A \coten B I)A\ar[r]^-{ip_A1} &
A^3 \ar[r]^-{m1} &
A^2 \ar[r]^-{q^{-1}1} &
(A \coten B I)BA \ar[d]^-{p_A\varepsilon_B  1} \\
A^2 \ar[rr]_-m &&
A &&&
A^2 \ar[lll]^-m}
\end{equation}

The above description of $\mathsf{CatMon}_{\mathcal S}(\mathsf C)$ requires no further explanation. 
In the description of $\mathsf{Xmod}_{\mathcal S}(\mathsf C)$ we need to show that the third diagram (the Peiffer condition) is equivalent to the diagram of Theorem \ref{thm:SCat_vs_Xmod}~(d') in the current setting. The path on the right hand side of the diagram of Theorem \ref{thm:SCat_vs_Xmod}~(d') appears as the left bottom path of the commutative diagram
$$
\xymatrix@C=40pt@R=15pt{
YBY \ar[d]_-{u111uu} \ar[r]^-{\delta_{YB}1} &
(YB)^2Y \ar[d]^-{u11u111uu} \ar[rr]^-{k1111} &&
B^2YBY \ar[dd]^-{m111} \ar[ld]_-{u11111uu} \\
(Y^2B)^2 \ar[d]_-{b_2b_2} \ar[r]^-{\delta_{Y^2B}111} &
(Y^2B)^3 \ar[r]^-{1k1\varepsilon_Y 11111} &
YB^2YBY^2B \ar[d]^-{1m11111} \\
(YB \coten B YB)^2 \ar[dddd]_-m \ar[r]_-{jj} \ar@/^1pc/[rr]^-{jb_2^{-1}} &
(YB)^4 \ar[dd]^-{11c_{YB,YB}11} \ar[r]_-{11111\varepsilon_B 11} &
(YB)^2Y^2B\ar[d]^-{11c_{YB,Y}11} &
(BY)^2 \ar[d]^-{1c_{YB,Y}} \ar[l]_-{u1111uu}  \\
&& YBY^2BYB \ar[d]^-{1l1111} &
BY^2B \ar[d]^-{l11} \ar[l]_-{u1111uu}  \\
& (YB)^4 \ar[d]^-{1x11x1} \ar[r]^-{1l\varepsilon_B1111} &
Y^3BYB \ar[d]^-{111x1} &
Y^2B \ar@{=}[dd] \ar[l]_-{u111uu} \ar[ld]^-{u11u1u} \\
& (Y^2B^2)^2 \ar[d]^-{mmmm} \ar[r]^-{11\varepsilon_B\varepsilon_B 1111} &
Y^4B^2 \ar[rd]^-{mmm} \\
YB \coten B YB \ar[r]^-j \ar@/_1.8pc/[rrr]_-{b_2^{-1}} &
(YB)^2 \ar[rr]^-{1\varepsilon_B 11}  &&
Y^2B \ar[r]_-{m1} &
YB}
$$
(in which $x$ stands for the distributive law 
$\xymatrix@C=12pt{BY \ar[r]^-{\delta 1} & B^2Y \ar[r]^-{1c} & BYB \ar[r]^-{l1} & YB}$ 
of Example \ref{ex:CoMon_SplitEpi}). Hence it can be replaced by the 
top right path yielding the equivalent form 
\begin{equation}\label{eq:d'_equivalent}
\xymatrix@C=10pt@R=3pt{
YBY \ar[rr]^-{\delta_{YB} 1} \ar[dd]_-{1\delta_B 1} &&
(YB)^2Y \ar[rr]^-{k1111} &&
B^2YBY \ar[rr]^-{m111} &&
(BY)^2 \ar[ddd]^-{1c_{YB,Y}} \\
\\
YB^2Y \ar[dd]_-{11c} \\
&&&&&& BY^2B \ar[ddd]^-{l11} \\
(YB)^2 \ar[dd]_-{1l1} \\
\\
Y^2B \ar[rrr]_-{m1} &&&
YB &&&
Y^2B \ar[lll]^-{m1}}
\end{equation}
of the diagram of Theorem \ref{thm:SCat_vs_Xmod}~(d').
The first diagram of Figure \ref{fig:d'_equivalent} shows that if the diagram of \eqref{eq:d'_equivalent} commutes then the Peiffer condition in the above presentation of $\mathsf{Xmod}_{\mathcal S}(\mathsf C)$ holds. The opposite implication is proven by the second diagram of Figure \ref{fig:d'_equivalent}.
\begin{figure} 
\centering
\begin{sideways}
$\xymatrix@C=55pt@R=15pt{
Y^2 \ar[rrr]^-{\delta 1} \ar@{=}[ddddd] \ar[rd]^-{1u1} &&&
Y^3 \ar[rr]^-{k11} \ar[d]^-{1u1u1} \ar[ld]_-{11uu1} &&
BY^2 \ar@/^1.1pc/@{=}[rd] \ar[d]^-{11u1} \ar[ld]_-{1u1u1}  \\
& YBY \ar[r]^-{\delta\delta 1} \ar[rd]_-{1\delta 1} \ar@{=}[ddd] &
Y^2B^2Y \ar[r]^-{1c11} &
(YB)^2Y \ar[r]^-{k1111} &
B^2YBY \ar[r]^-{m111} &
(BY)^2 \ar[r]^-{11\varepsilon_B 1} \ar[d]^-{1c_{YB,Y}} &
BY^2 \ar[d]^-{1c} \\
&& YB^2Y \ar[ldd]_(.4){11\varepsilon_B 1} \ar[d]^-{11c} &&&
BY^2B \ar[r]^-{111\varepsilon_B} \ar[d]^-{l11} &
BY^2 \ar[d]^-{l1} \\
& 
&
(YB)^2\ar[ld]_(.3){111\varepsilon_B} \ar[d]^-{1l1} &&&
Y^2B \ar[d]^-{m1} &
Y^2 \ar[dd]^-m \\
& YBY \ar[ld]_-{1l} &
Y^2B \ar[lld]_(.3){11\varepsilon_B} \ar[rrr]^-{m1} &&&
YB \ar[rd]^-{1\varepsilon_B} \\
Y^2 \ar[rrrrrr]_-m &&&&&&
Y \\
\\
YBY \ar[rr]^-{1\delta 1} &&
YB^2Y \ar[r]^-{\delta 111} \ar[d]_-{11c} &
Y^2B^2Y \ar[r]^-{1c11} &
(YB)^2Y \ar[r]^-{k1111} &
B^2YBY \ar[r]^-{m111} &
(BY)^2 \ar[d]^-{11c} \\
&& (YB)^2 \ar[r]^-{\delta 111} \ar[ddd]_-{1l1} &
Y^2BYB \ar[r]^-{1c11} \ar[dd]^-{11l1} &
YBY^2B \ar[r]^-{k1111} \ar[d]^-{11c1} &
B^2Y^2B \ar[r]^-{m111} &
BY^2B \ar[d]^-{1c1} \\
&&&& YBY^2B \ar[r]^-{k1111} \ar[d]^-{1l11} &
B^2Y^2B \ar[r]^-{m111} \ar[d]^-{1l11} &
BY^2B \ar[d]^-{l11} \\
&&& Y^3B \ar[r]^-{1c1} &
Y^3B \ar[r]^-{k111} &
BY^2B \ar[r]^-{l11} &
Y^2B\ar[d]^-{m1} \\
&& Y^2B \ar[rrrr]_-{m1} \ar[ru]^-{\delta 11}  &&&&
YB}
$
\end{sideways}
\caption{Derivation of the Peiffer condition for crossed modules of bimonoids}
\label{fig:d'_equivalent}
\end{figure}

In order to justify the further equivalent characterization of these categories as a full subcategory of $\mathsf{ReflGraphMon}_{\mathcal S}(\mathsf C)$, we need to see  the equivalence of the diagram of Proposition \ref{prop:relcat-monoid}~(2) in the current setting and the diagram of \eqref{eq:d_exists_CoMon}. This follows by noting that the top row of the diagram of Proposition \ref{prop:relcat-monoid}~(2) in the current setting appears in the left-bottom path of the commutative diagram 
$$
\xymatrix@R=15pt{
A(A\coten B I) \ar[r]^-{\delta 1} \ar[dd]_-{1p_A} \ar[rd]^-{\delta \delta} &
A^2(A\coten B I) \ar[rd]^-{11p_A} \ar[rrr]^-{1c} &&&
A(A\coten B I)A \ar[ld]_-{1p_A1} \ar[d]^-{t11} \\
& A^2(A\coten B I)^2 \ar[d]^-{11p_Ap_A} \ar[r]^(.5){11p_Ap_I} &
A^3 \ar[r]^-{1c} \ar[d]^-{t11} &
A^3 \ar[rd]^-{t11} &
B(A\coten B I)A \ar[d]^-{1p_A1} \\
A^2 \ar[r]^-{\delta \delta} \ar[dd]_-{(i \, \diagcoten 1)(1\diagcoten \, i)} &
A^4 \ar[d]^-{t11s} &
BA^2 \ar[d]^-{i11} \ar[ld]_-{111u} &&
BA^2 \ar[d]^-{i11} \\
& BA^2B \ar[d]^-{i11i} &
A^3 \ar[rr]^-{1c} \ar[ld]_-{111u} &&
A^3 \ar[dd]^-{m1} \ar[ld]_-{111u} \\
(A\coten B A)^2 \ar[r]^-{jj} \ar[dd]_-m &
A^4 \ar[rr]^-{1c1} &&
A^4 \ar[rd]^-{mm} \\
&&&& A^2 \ar[d]^-{q^{-1}1} \\
A\coten B A \ar[rr]_-{q^{-1}\diagcoten 1} \ar[rrrru]^-j \ar@/_1.1pc/[rrrrd]_-{q_2^{-1}} &&
(A\coten B I)B \coten B A \ar[rr]_-j \ar[rrd]_(.45){h_1^{-1}} \ar@{}[d]|-{\eqref{eq:qBox1.hn}} &&
(A\coten B I)BA \ar[d]^-{1\varepsilon_B 1} \\
&&&& (A\coten B I)A}
$$
hence it can be replaced by the top-right path. (The expression of $h_1^{-1}$ in the bottom-right corner was computed in Example \ref{ex:hn_CoMon}.)
\end{example}

\begin{proposition} \label{prop:relcat_Hopf}
The equivalent categories of Example \ref{ex:SCat_CoMon} have equivalent full subcategories as follows.
\begin{itemize}
\item The full subcategory of $\mathsf{CatMon}_{\mathcal S}(\mathsf C)$ for whose objects 
$\xymatrix@C=20pt{
B \ar@{ >->}[r]|(.55){\, i\, } &
A \ar@{->>}@<-4pt>[l]_-s  \ar@{->>}@<4pt>[l]^-t &
A\coten B A \ar[l]_-d}$ 
the bimonoid $B$ in $\mathsf M$ is a Hopf monoid.
\item The full subcategory of $\mathsf{Xmod}_{\mathcal S}(\mathsf C)$ for whose objects 
$(B,Y,\xymatrix@C=12pt{BY \ar[r]^-l & Y},\xymatrix@C=12pt{Y \ar[r]^-k & B})$
the bimonoid $B$ in $\mathsf M$ is a Hopf monoid.
\item The full subcategory of $\mathsf{ReflGraphMon}_{\mathcal S}(\mathsf C)$ for whose objects 
$\xymatrix@C=20pt{
B \ar@{ >->}[r]|(.55){\, i\, } &
A \ar@{->>}@<-4pt>[l]_-s  \ar@{->>}@<4pt>[l]^-t}$ 
the following conditions hold.
\begin{itemize}
\item $B$ is a Hopf monoid (with antipode $z$)
\item $t1.\delta=t1.c.\delta$
\item for the morphisms 
$$
\overrightarrow s:=\xymatrix@C=10pt{
A \ar[r]^-\delta & A^2 \ar[r]^-{1s} & AB \ar[r]^-{1z} & AB \ar[r]^-{1i} & A^2 \ar[r]^-m & A}, \ 
\overleftarrow t:=\xymatrix@C=10pt{
A \ar[r]^-\delta & A^2 \ar[r]^-{t1} & BA \ar[r]^-{z1} & BA \ar[r]^-{i1} & A^2 \ar[r]^-m & A}
$$
the following diagram commutes.
\begin{equation} \label{eq:vil_d}
\xymatrix{
A^2 \ar[r]^-{\overrightarrow s \overleftarrow t}   \ar[d]_-{\overrightarrow s \overleftarrow t} &
A^2 \ar[r]^-c &
A^2 \ar[d]^-m \\
A^2 \ar[rr]_-m &&
A}
\end{equation}
\end{itemize}
\end{itemize}
\end{proposition}

\begin{proof}
The only ingredient that requires a proof is the equivalence of diagrams \eqref{eq:d_exists_CoMon} and \eqref{eq:vil_d} in the case when $B$ has an antipode $z$. 
The proof will repeatedly use the identity on $\overrightarrow s$ encoded in the following commutative diagram.
\begin{equation} \label{eq:sbar_id}
\xymatrix@R=15pt{
A^2 \ar[rrrrrr]^-m \ar[rd]^-{\delta 1} \ar[d]_-{1\overrightarrow s} &&&&&&
A \ar[d]_-\delta \ar@/^1.2pc/[dddddd]^-{\overrightarrow s} \\
A^2 \ar[ddddd]_-{\delta 1} &
A^3 \ar[r]^-{11\delta} \ar[dd]^-{1c} \ar[lddddd]_-{11\overrightarrow s} &
A^4 \ar[rr]^-{1c1} &&
A^4 \ar[rr]^-{mm} \ar[d]^-{11ss} \ar[ldd]_-{11c} &&
A^2 \ar[d]_-{1s} \\
&&&& A^2B^2 \ar[rr]^-{mm} \ar[d]^-{11c} &&
AB \ar[dd]_-{1z} \\
& A^3 \ar[rr]^-{1\delta 1} \ar[ddd]^-{1\overrightarrow s 1} &&
A^4 \ar[r]^-{11ss} &
A^2B^2 \ar[d]^-{11zz} \\
&&&& A^2B^2 \ar[rr]^-{mm} \ar[d]^-{11ii} &&
AB \ar[d]_-{1i} \\
&&&& A^4 \ar[rr]^-{mm} \ar[d]^-{1m1} &&
A^2 \ar[d]_-m \\
A^3 \ar[r]_-{1c} &
A^3 \ar[r]_-{11s} &
A^2B \ar[r]_-{11z}  &
A^2B \ar[r]_-{11i} &
A^3 \ar[r]_-{m1} &
A^2 \ar[r]_-m &
A}
\end{equation}

Recall from \cite{Radford} that if $B$ has an antipode $z$ then $\xymatrix@C=12pt{A\coten B I \ar[r]^-{p_A} & A}$ is a split monomorphism in $\mathsf M$; a retraction is provided by 
$g_A:=\xymatrix@C=15pt{
A \ar[r]^-{q^{-1}} &
(A \coten B I)B \ar[r]^-{1\varepsilon} &
A \coten B I}$.
Indeed,
$$
g_A.p_A=
1\varepsilon. q^{-1}.q.1u=
1\varepsilon.1u=1.
$$ 
On the other hand, since in Proposition \ref{prop:SplitEpi_Hopf} $q^{-1}$ was constructed as the unique solution of $p_A1.q^{-1}=\overrightarrow s s.\delta$, also the equality 
$$
p_A.g_A=
p_A.1\varepsilon. q^{-1}=
1\varepsilon. p_A 1.q^{-1}=
1\varepsilon.\overrightarrow s s.\delta=
\overrightarrow s
$$
holds, proving that $\overrightarrow s$ is idempotent.

Pre-composing both paths around \eqref{eq:d_exists_CoMon} with the split epimorphism $1g_A$, we obtain the equivalent diagram
\begin{equation} \label{eq:pre_gA}
\xymatrix@C=20pt@R=20pt{
A^2 \ar[rr]^-{\delta 1} \ar[rd]^-{1g_A} \ar[ddd]_-{1 \overrightarrow s} &&
A^3 \ar[r]^-{1c} \ar[d]^-{11g_A} &
A^3 \ar[r]^-{t11} &
BA^2 \ar[r]^-{i11} \ar[d]^-{1g_A 1} &
A^3 \ar[r]^-{m1} \ar[d]^-{1\overrightarrow s 1} &
A^2 \ar[ddd]^-{\overrightarrow s 1} \\
& A(A \coten B I) \ar[r]^-{\delta 1} \ar[ldd]^-{1p_A} &
A^2(A \coten B I) \ar[r]^-{1c} &
A(A \coten B I)A \ar[r]^-{t11} &
B(A \coten B I)A \ar[r]^-{ip_A 1} &
A^3 \ar[d]^-{m1} \\
&&&&& A^2 \ar[rd]^-{\overrightarrow s 1} \\
A^2 \ar[rrr]_-m &&&
A &&&
A^2. \ar[lll]^-m}
\end{equation}
Its rightmost region commutes by \eqref{eq:sbar_id} and the fact that $\overrightarrow s$ is idempotent. 

The morphism around the right hand side of \eqref{eq:pre_gA} occurs as the left-bottom path of the commutative diagram
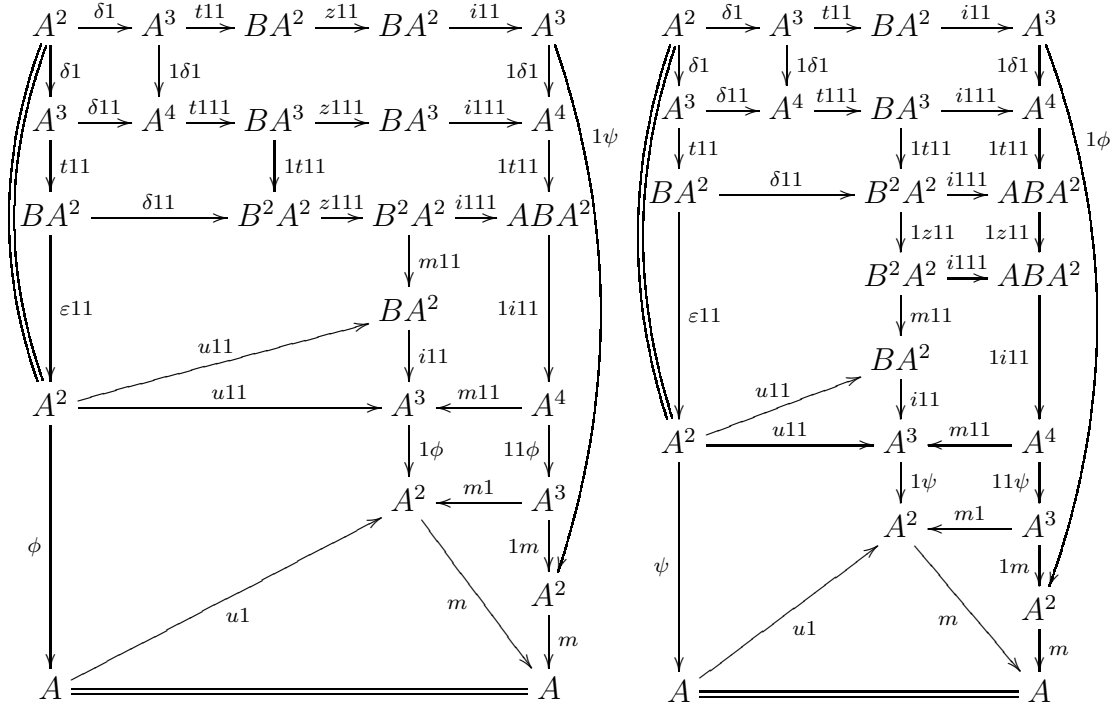
$$
\xymatrix@R=15pt{
A^2 \ar[r]^-{\delta 1} \ar[d]_-{\delta 1} &
A^3 \ar[r]^-{1c} \ar[d]^-{1\delta 1} &
A^3 \ar[rr]^-{t11} \ar[dd]^-{11\delta} &&
BA^2 \ar[ddddd]^-{i11} \\
A^3 \ar[r]^-{\delta 11} \ar[d]_-{1c} &
A^4 \ar[d]^-{11c} \\
A^3 \ar[r]^-{\delta 11} \ar[d]_-{t11} &
A^4 \ar[r]^-{1c1} \ar[d]^-{tt11} &
A^4 \ar[d]^-{t1t1} \\
BA^2 \ar[d]_-{i11} &
B^2A^2 \ar[d]^-{ii11} &
(BA)^2 \ar[d]^-{i1i1} \\
A^3 \ar[r]^-{\delta 11} \ar[ddd]_-{m1} &
A^4 \ar[r]^-{1c1} \ar@{}[dddddd]|-{\eqref{eq:sbar_id}} &
A^4 \ar[d]^-{11s1} \\
&& A^2BA \ar[d]^-{11z1} &
A^4 \ar[l]_-{11t1} &
A^3 \ar[l]_-{11\delta} \ar[d]^-{11\overleftarrow t} \\
&& A^2BA \ar[r]^-{11i1} &
A^4 \ar[r]^-{11m} \ar[d]^-{1\overrightarrow s 11} &
A^3 \ar[d]^-{1\overrightarrow s 1} \\
A^2 \ar[ddd]_-{\overrightarrow s 1} &&&
A^4 \ar[r]^-{11m} \ar[ld]^-{m11} &
A^3 \ar[ddd]^-{1m}  \\
&& A^3 \ar[lldd]^-{m1} \\
\\
A^2 \ar[rr]_-m &&
A &&
A^2. \ar[ll]^-m}
$$
Hence it can be replaced by the top-right path yielding the equivalent form
\begin{equation} \label{eq:pre_gA_2}
\xymatrix{
A^2 \ar[r]^-{\delta 1} \ar[d]_-{1\overrightarrow s} &
A^3 \ar[r]^-{t11} &
BA^2 \ar[rr]^-{i11} &&
A^3 \ar[r]^-{1c} &
A^3 \ar[r]^-{1\overrightarrow s \overleftarrow t} &
A^3 \ar[d]^-{1m} \\
A^2 \ar[rrr]_-m &&&
A &&&
A^2 \ar[lll]^-m}
\end{equation}
of \eqref{eq:pre_gA}.

Finally, observe that for any morphisms
$\xymatrix@C=15pt{A^2 \ar[r]^-{\phi,\psi} & A}$
the following diagrams are equivalent:
\begin{equation} \label{eq:conv}
\xymatrix{
A^2 \ar[r]^-{\delta 1} \ar[d]_-\psi &
A^3 \ar[r]^-{t11} &
BA^2 \ar[r]^-{i11} &
A^3 \ar[d]^-{1\phi} \\
A &&& 
A^4 \ar[lll]^-m}\qquad
\xymatrix{
A^2 \ar[r]^-{\delta 1} \ar[d]_-\phi &
A^3 \ar[r]^-{t11} &
BA^2 \ar[r]^-{z11} &
BA^2 \ar[r]^-{i11} &
A^3 \ar[d]^-{1\psi} \\
A &&&&
A^4. \ar[llll]^-m}
\end{equation}
Indeed, the first diagram below shows that if the first diagram of \eqref{eq:conv} commutes then so does the second one; and the opposite implication follows by the second diagram below.
$$
\scalebox{1.01}{$
\xymatrix@C=15pt@R=19pt{
A^2 \ar[r]^-{\delta 1} \ar[d]^-{\delta 1} \ar@/_1.2pc/@{=}[dddd] &
A^3 \ar[r]^-{t11} \ar[d]^-{1 \delta 1} &
BA^2 \ar[r]^-{z11} &
BA^2 \ar[r]^-{i11} &
A^3 \ar[d]_-{1\delta 1} \ar@/^1.6pc/[dddddd]^(.2){1\psi} \\
A^3 \ar[r]^-{\delta 11}  \ar[d]^-{t11} &
A^4 \ar[r]^-{t111}  &
BA^3 \ar[r]^-{z111} \ar[d]^-{1t11} &
BA^3 \ar[r]^-{i111} &
A^4 \ar[d]_-{1t1 1} \\
BA^2 \ar[rr]^-{\delta 11}  \ar[dd]^-{\varepsilon 11} &&
B^2A^2 \ar[r]^-{z111}  &
B^2A^2 \ar[r]^-{i111} \ar[d]^-{m11} &
ABA^2 \ar[dd]_-{1i1 1} \\
&&& BA^2 \ar[d]^-{i11} \\
A^2 \ar[rrr]^-{u11} \ar[rrru]^-{u11} \ar[ddd]_-\phi &&&
A^3 \ar[d]^-{1\phi} &
A^4 \ar[l]_-{m11} \ar[d]_-{11\phi} \\
&&& A^2 \ar[rdd]_-m &
A^3 \ar[l]_-{m1} \ar[d]_-{1m} \\
&&&& A^2 \ar[d]^-m \\
A \ar@{=}[rrrr] \ar[uurrr]_-{u1} &&&&
A}$}\ 
\xymatrix@C=15pt@R=15pt{
A^2 \ar[r]^-{\delta 1} \ar[d]^-{\delta 1} \ar@/_1.2pc/@{=}[ddddd] &
A^3 \ar[r]^-{t11} \ar[d]^-{1 \delta 1} &
BA^2 \ar[r]^-{i11} &
A^3 \ar[d]_-{1\delta 1} \ar@/^1.8pc/[ddddddd]^(.2){1\phi} \\
A^3 \ar[r]^-{\delta 11}  \ar[d]^-{t11} &
A^4 \ar[r]^-{t111}  &
BA^3 \ar[r]^-{i111} \ar[d]^-{1t1 1}&
A^4 \ar[d]_-{1t1 1} \\
BA^2 \ar[rr]^-{\delta 11}  \ar[ddd]^-{\varepsilon 11} &&
B^2A^2 \ar[r]^-{i111}  \ar[d]^-{1z11} &
ABA^2 \ar[d]_-{1z11} \\
&& B^2A^2 \ar[r]^-{i111} \ar[d]^-{m11} &
ABA^2 \ar[dd]_-{1i1 1} \\
&& BA^2 \ar[d]^-{i11} \\
A^2 \ar[rr]^-{u11} \ar[rru]^-{u11} \ar[ddd]_-\psi &&
A^3 \ar[d]^-{1\psi} &
A^4 \ar[l]_-{m11} \ar[d]_-{11\psi} \\
&& A^2 \ar[rdd]_-m &
A^3 \ar[l]_-{m1} \ar[d]_-{1m} \\
&&& A^2 \ar[d]^-m \\
A \ar@{=}[rrr] \ar[uurr]_-{u1} &&&
A}
$$
Applying the equivalence of the diagrams of \eqref{eq:conv} to
$\phi:=\xymatrix@C=8pt{
A^2\ar[r]^-c & A^2\ar[rr]^-{\overrightarrow s \overleftarrow t} && A^2 \ar[r]^-m &A}$
and
$\psi:=\xymatrix@C=15pt{
A^2\ar[r]^-{1 \overrightarrow s } & A^2 \ar[r]^-m &A}$,
we obtain from \eqref{eq:pre_gA_2} the equivalent form
$$
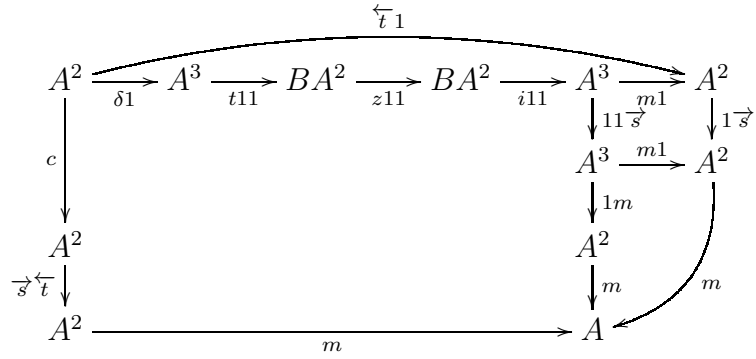
\xymatrix@R=15pt{
A^2 \ar[r]_-{\delta 1} \ar[dd]_-c \ar@/^1.4pc/[rrrrr]^-{\overleftarrow t 1} &
A^3 \ar[r]_-{t11} &
BA^2 \ar[r]_-{z11} &
BA^2 \ar[r]_-{i11} &
A^3 \ar[r]_-{m1} \ar[d]^-{11\overrightarrow s } &
A^2 \ar[d]^-{1\overrightarrow s} \\
&&&& A^3 \ar[r]^-{m1} \ar[d]^-{1m} &
A^2 \ar@/^1.5pc/[ldd]^-m \\
A^2\ar[d]_-{\overrightarrow s  \overleftarrow t} &&&&
A^2 \ar[d]^-m \\
A^2 \ar[rrrr]_-m &&&&
A }
$$
which is equivalent to \eqref{eq:vil_d} by the naturality of the symmetry $c$.
\end{proof}

The equivalent categories of Proposition \ref{prop:relcat_Hopf} have equivalent full subcategories in whose objects both occurring bimonoids are Hopf monoids, and other equivalent full subcategories in whose objects they are both cocommutative Hopf monoids.
In this way,  Proposition \ref{prop:relcat_Hopf} includes 
\cite[Proposition 11]{Villanueva} and \cite[Theorem 14]{Villanueva} about the equivalence between 
the category of so-called $\mathsf{Cat}^1$-Hopf algebras and the category of crossed modules over Hopf algebras; hence in particular the equivalence between the category of $\mathsf{Cat}^1$-groups and the category of crossed modules over groups in \cite[Section 3.9]{Janelidze}.


\bibliographystyle{plain}

\begin{thebibliography}{10}

\bibitem{Aguiar}
Marcelo Aguiar,
{\em Internal Categories and Quantum Groups}, 
Ph.D. thesis, Cornell University 1997.
 
\bibitem{Bohm:Xmod_I}
Gabriella B\"ohm,
{\em Crossed modules of monoids I. Relative categories},
preprint available at
\href{https://arxiv.org/abs/1803.03418}{arXiv:1803.03418}

\bibitem{Bohm:Xmod_III}
Gabriella B\"{o}hm,
{\em Crossed modules of monoids III. Simplicial monoids of Moore length 1},
in preparation.

\bibitem{BrownIcen}
Ronald Brown and \.{I}lhan \.{I}\c{c}en,
{\em Homotopies and Automorphisms of Crossed Modules of Groupoids},
Appl. Categ. Structures 11 no. 2  (2003) 185-206. 

\bibitem{BrownSpencer}
Ronald Brown and Christopher B. Spencer,
{\em $\mathcal G$-groupoids, crossed modules and the fundamental groupoid of a topological group},
Indagationes Mathematicae (Proceedings) 79 no. 4 (1976) 296-302.

\bibitem{Duskin}
John Williford Duskin, 
{\em Preliminary remarks on groups},
as quoted in \cite{BrownSpencer}:  Unpublished notes, Tulane University, (1969).

\bibitem{Emir}
Kadir Emir,
{\em 2-Crossed Modules of Hopf Algebras,}
talk given at ``Category Theory 2016'' Halifax, Canada.
slides available at 
\href{http://mysite.science.uottawa.ca/phofstra/CT2016/slides/Emir.pdf}{\tt http://mysite.science.uottawa.ca/phofstra/CT2016/slides/Emir.pdf}. 

\bibitem{FariaMartins}
Jo\~{a}o Faria Martins,
{\em Crossed modules of Hopf algebras and of associative algebras and
two-dimensional holonomy}, 
J. Geom. Phys. 99 (2016), 68-110. 

\bibitem{GKV}
Marino Gran, Gabriel Kadjo and Joost Vercruysse, 
{\em A torsion theory in the category of cocommutative Hopf
algebras,} 
Appl. Categ. Structures 24 no. 3 (2016), 269-282.

\bibitem{Janelidze}
George Janelidze,
{\em Internal crossed modules,}
Georgian Math. J. 10 (2003), 99-114.

\bibitem{Majid}
Shahn Majid,
{\em Strict quantum 2-groups,}
preprint available at 
\href{https://arxiv.org/abs/1208.6265}{arXiv:1208.6265}.

\bibitem{Paoli}
Simona Paoli,
{\em Internal categorical structures in homotopical algebra,}
in: ``Towards Higher Categories'' J. C. Baez and J. P. May (eds.), 
IMA volumes in Mathematics and its applications pp 85-103, 
Springer 2009.

\bibitem{Porter:Menagerie}
Timothy Porter, 
{\em The Crossed Menagerie:
an introduction to crossed gadgetry and cohomology in algebra and
topology,}
Notes initially prepared for the XVI Encuentro Rioplatense de
\' Algebra y Geometr\'{i}a Algebraica, in Buenos Aires, 12-15 December
2006, extended for an MSc course (Summer 2007) at Ottawa. 
available at 
\href{https://ncatlab.org/timporter/show/crossed+menagerie}{\tt https://ncatlab.org/timporter/show/crossed+menagerie}.

\bibitem{Porter:HQFT}
Timothy Porter,
{\em Homotopy Quantum Field Theories meets the Crossed Menagerie:
an introduction to HQFTs and their relationship with things
simplicial and with lots of crossed gadgetry,} 
Notes prepared for the Workshop and School on Higher Gauge
Theory, TQFT and Quantum Gravity Lisbon, February, 2011.
available at 
\href{https://ncatlab.org/timporter/show/HQFTs+meet+the+Crossed+Menagerie}
{\tt https://ncatlab.org/timporter/show/HQFTs+meet+the+Crossed+Menagerie}.

\bibitem{Radford}
David E. Radford,
{\em The structure of Hopf algebras with a projection,}
J. Algebra 92 no. 2 (1985), 322-347.

\bibitem{Villanueva}
Jos\'e Manuel Fern\'andez Vilaboa, Mar\'{\i}a Purificaci\'on L\'opez L\'opez 
and Emilo Villanueva Novoa,
{\em $\mathsf{Cat}^1$-Hopf Algebras and Crossed Modules,}
Commun. Algebra 35 no. 1 (2006) 181-191.

\bibitem{Whitehead}
John Henry Constantine Whitehead,
{\em On adding relations to homotopy groups}, 
Ann. of Math. 42 no. 2 (1941) 409-428.



\end{thebibliography}

\end{document}